\documentclass{amsart}

\usepackage{amssymb}
\usepackage{amsmath,enumerate,rotate}
\usepackage{amssymb,latexsym}
\usepackage{epsfig}
\usepackage{graphicx}
\usepackage[english]{babel}
\usepackage{color}

\usepackage{a4wide}
\usepackage{natbib}

\newtheorem{definition}{Definition}

\newtheorem{proposition}[definition]{Proposition}
\newtheorem{corollary}[definition]{Corollary}

\def \CK{\mathcal{K}}

\def\CX {\mathcal{X}}

\def\CM { \mathcal{M}}

\def\CD {\mathcal{D}}

\def\a {\alpha}

\def\s {\sigma}

\def\a {\alpha}

\def\k {\kappa}

\def\J {\mathbb{I}}

\def\N {\mathbb{N}}

\def\G {\mathbb{G}}
\def\E {\mathbb{E}}
\def\RE {\mathbb{R}}

\def\X {\mathbb{X}}

\def\Y {\mathbb{Y}}

\def \TG {{G}}
\def \PPhi {{\tilde \varphi}}
\def \XXi {\alpha}

\usepackage{verbatim}

\begin{document}
\title{Beta-Product Poisson-Dirichlet Processes}
\author{Federico Bassetti}
\address{Universit\`a degli Studi di Pavia}
\email{federico.bassetti@unipv.it}

\author{Roberto Casarin}
\address{University Ca' Foscari, Venice}
\email{r.casarin@unive.it}

\author{Fabrizio Leisen}
\address{University Carlos III de Madrid}
\email{leisen@gmail.com}

\begin{abstract}
Time series data may exhibit clustering over time and, in a multiple time series context, the clustering behavior may differ across the series. This paper is motivated by the Bayesian non--parametric modeling of the dependence between the clustering structures and the distributions of different time series. We follow a Dirichlet process mixture approach and introduce a new class of multivariate dependent Dirichlet processes (DDP). The proposed DDP
are represented in terms of vector of stick-breaking processes with dependent weights. The weights are beta random vectors that determine different and dependent clustering effects along the dimension of the DDP vector. We discuss some theoretical properties and provide an efficient Monte Carlo Markov Chain algorithm for posterior computation. The effectiveness of the method is illustrated with a simulation study and an application to the United States and the European Union industrial production indexes.

\medskip


\medskip

\end{abstract}

\date{\today}

\maketitle

\section{Introduction}
This paper is concerned with some multivariate extensions of the Poisson-Dirichlet process. In this paper the class of models
considered  originates from the Ferguson Dirichlet process (DP) (see \cite{Ferguson73,Ferguson74})
that is now  widely used in non-parametric Bayesian statistics.  Our extensions rely upon the so-called Sethuraman's representation of the DP.  \cite{Sethuraman} shows that, given a Polish space $(\X,\CX)$ and a probability measure $H_0$ on $\CX$, a random probability measure $\TG$ on $\CX$ is  a Dirichlet Process of precision
parameter $\alpha>0$ and base measure $H_0$, in short $DP(\a,H_0)$, if and only if
it has the stick-breaking representation:
\begin{equation}\label{Setu}
\TG(\cdot)=\sum_{k\geq1}W_{k} \,\delta_{\PPhi_{k}}(\cdot),
\end{equation}
where $(\PPhi_k)_k$ and $(W_k)_k$ are stochastically independent,
$(\PPhi_k)_k$ is a sequence of independent random variables (atoms) with common distribution $H_0$ (base measure),
and the weights $W_k$s are defined by the stick-breaking construction:
\begin{equation}\label{stickmonodim}
W_k:=S_{k} \prod_{j <k} (1-S_{j}),
\end{equation}
$S_j$ being independent random variables with beta distribution $Beta(1,\alpha)$.

The DP has been extended in many directions. In this paper, we will build on a generalization of the DP
that is called the Poisson--Dirichlet process. A Poisson--DP, $PD(\alpha,l,H_0)$, with parameters $0\leq l<1$ and $\alpha>-l$ and base measure $H_0$, is a random probability measure
that can be represented with a Sethuraman-like construction by taking in \eqref{Setu}-\eqref{stickmonodim}
a sequence of independent random variables $S_k$ with $S_k \sim Beta(1-l,\alpha+lk)$, see \cite{Pitman2006} and \cite{PitmanYor1997}.
Further generalizations based on the stick-breaking construction of the DP can be found in \cite{IsJam2001}.

The DP process and its univariate extensions are now widely used in Bayesian non-parametric statistics. A recent account of Bayesian non-parametric inference can be found in \cite{BNP2010}. The univariate DP is usually employed as a prior for a mixing distribution, resulting in a DP mixture (DPM) model (see for example \cite{Lo1984}). The DPM models incorporate DP priors for
parameters in Bayesian hierarchical models, providing an extremely flexible class of models. More specifically, the DPM models define a random density
by setting:
\begin{equation}\label{Setu2}
f(y)=\int  \CK(y|\theta) \TG(d\theta)   =\sum_{k \geq 1}W_{k} \CK(y|\PPhi_{k}),
\end{equation}
where $\CK$ is a suitable density kernel. Due to the availability of simple and efficient methods for
posterior computation, starting from \cite{Escobar94} and \cite{EscobarWest1995}, DPM models  are now routinely
implemented and used in many fields.

The first aim of this paper is to introduce
new class of vectors of Poisson-Dirichlet processes and of DPM models. Vectors of random probability measures arise naturally in generalizations of the DP and DPM models that accommodate dependence of observations on covariates or time.
Using covariates, data may be divided into different groups and this leads
to consider group-specific random probability measures and,
as we shall see, to assume that the observations are partially exchangeable.

Probably the first paper in this direction, that introduced vectors of priors for partially exchangeable
data, is \cite{CifReg1978}. More recently, \cite{MacEachern1999,MacEachern2001} introduced the
so-called dependent Dirichlet process (DDP) and the DDP mixture models.
His specification of the DDP incorporates dependence
on covariates through the atoms while assuming fixed weights. More specifically, the atoms in the Sethuraman's
representation \eqref{Setu}-\eqref{Setu2} are replaced with stochastic processes $\PPhi_k(z)$, $z$ being a set of covariates.
There exist many applications of this specification of the DDP.
For instance,  \cite{DeIorio2004}
proposed an ANOVA-type dependence structure for the atoms, while \cite{Gelfand2004} considered a spatial dependence structure for the atoms.
Later, DDP with both dependent atoms and weights was introduced in
\cite{GriffinSteel2006}. Other constructions that incorporate a
dependence structure in the weights have been proposed, for instance, in
\cite{Duan2005,ChungDunson,DunsonPeddada2008,DunsonXueCarin2008} and \cite{RodriguezDunsonGelfand2010}.

Other approaches to the definition of dependent vectors of random measures
rely upon either suitable convex combinations of independent DPs (e.g., \cite{Muller2004,PennellDunson2006,HatjispyrosaNicolerisaWalker,Kolo2010})
or hierarchical structures of  stick-breakings (e.g.,  \cite{TehJordan2006,SudderthJordan2009}).

Finally, we should note that it is possible to follow alternative routes other than the Sethuraman's representation to the definition of vectors of dependent random probabilities. For example, \cite{LeisenLijoi2011} used normalized vectors of completely random measures, while \cite{IsZar2009} employed bivariate gamma processes.

In this paper, we introduce a new class of multivariate  Poisson-DP and DP by using a vector of stick-breaking processes with multivariate dependent weights. In the construction of the dependent weights, we consider
the class of multivariate beta distributions introduced by \cite{NadarajahKotz:2005} that have a tractable stochastic
representation and makes  the Bayesian inference procedures easier.
We discuss some properties of the resulting multivariate DP and Poisson-DP and show that our process has
the appealing property that the marginal distributions are DP or Poisson-DP.

The second aim of the paper is to apply the new DP to Bayesian non--parametric inference and to provide a simple and efficient method for posterior computation of DDP mixture models. We follow a data-augmentation
framework and extend to the multivariate context the slice sampling algorithm described in \cite{walker2007} and \cite{walker2011}. The sampling methods for the full
conditional distributions of the resulting Gibbs sampling procedure are detailed and the effectiveness of the proposed algorithm is studied with a set
of simulation experiments.

Another contribution of this paper is to present an application of the proposed multivariate DDP mixture models to multivariate time series modeling. In the recent years, the interest in Bayesian non-parametric models for time series has increased. In this context, DP have been recently employed in different
ways. \cite{RodriguezTerHorst2008} used a Dirichlet process to define an infinite mixture of time
series models. \cite{TaddyKottas2009} proposed a Markov-switching finite mixture of independent
Dirichlet process mixtures. \cite{JensenMaheu2010} considered Dirichlet process mixture of
stochastic volatility models and \cite{Griffin2010} proposed a continuous-time non--parametric model
for volatility. A flexible non--parametric model with a time-varying stick-breaking process
has been recently proposed by \cite{GriffinSteel2011}. In their model, a sequence of dependent
Dirichlet processes is used for capturing time-variations in the clustering
structure of a set of time series. In our paper, we extend the existing Bayesian non--parametric
models for multiple time series by allowing each series to have a possible different clustering structure and by accounting for dependence between the series-specific clusterings. Since we obtain
a dynamic infinite-mixture model and since the number of components with negligible weights can be different in each series, our model represents a non--parametric alternative to multivariate dynamic finite-mixture models (e.g., Markov-switching models) that are usually employed in time series analysis.

The structure of the paper is as follows. Section \ref{Sec2} introduces vectors of dependent stick-breaking processes and defines their properties. Section \ref{Sec4} introduces vectors of Poisson-Dirichlet processes for prior modelling. Section \ref{S:Slice} proposes a Monte Carlo Markov Chain (MCMC) algorithm for approximated inference for vector of DP mixtures. Section \ref{sec_results} provides some applications to both simulated data and to the time series of the industrial production index for the United States and the European Union. Section \ref{Concl} concludes the paper.

\section{Dependent stick-breaking processes}\label{Sec2}
Consider a set of observations, taking values in a space $\mathbb{Y}$,  say a subset of $\RE^d$,
divided in $r$ sub--samples (or group of observations), that is:
\[
Y_{ij} \qquad i=1,\dots r, \quad j=1,\dots,n_i.
\]
Above $Y_{ij}$ is the $j$-th observation within sub--sample $i$. For instance,
$i$ may correspond to a space label  or predictors.
Typically,  one assumes that the  observations of the block $i$
have the same (conditional) density $f_i$ and that
the observations are (conditionally) independent.
Hence,  to perform a non--parametric Bayesian analysis of the data,
one needs to specify
a prior distribution for the vector of densities
\(
(f_1,f_2,\dots,f_{r})\).
Moreover, in assessing a prior for $(f_1,f_2,\dots,f_{r})$, a possible interest is on
borrowing  information across blocks.
To do this, first we introduce  a sequence of density kernels $\CK_i:\mathbb{Y} \times \X \to [0,1]$ ($i=1,\dots,r$)
(where $\CK_i$ is jointly measurable
and $C \mapsto \int_C \CK_i(y|x)\nu(dy)$ defines a probability measure on $\Y$ for any $x$ in $\X$, $\nu$ being a dominating measure on $\Y$). Secondly we define:
\begin{equation}\label{mixturemodelDef-1}
f_i(y):=\int \,\, \CK_i(y|\theta) \TG_i(d\theta) \qquad i=1,\dots,r,
\end{equation}
where $(\TG_1,\dots,\TG_r)$ is a vector of dependent stick breaking
processes that will be defined in the next section.

\subsection{Vectors of stick-breaking processes}\label{SS:defstick}
Following a general definition of dependent stick-breaking  processes, essentially proposed in
\cite{MacEachern1999,MacEachern2001}, we let
\begin{equation}\label{DDP}
\TG_i(\cdot):=\sum_{k \geq 1} W_{i k} \,\, \delta_{\PPhi_{i k}}(\cdot) \qquad i=1,\dots,r.
\end{equation}
where the vectors of weights $W_k=(W_{1k},\ldots,W_{rk})$ and the atoms $\PPhi_{k}=(\PPhi_{1k},\ldots,\PPhi_{rk})$ satisfy the following hypotheses:
\begin{itemize}
\item $(\PPhi_k)_k$ and $(W_k)_k$ are stochastically independent;
\item $(\PPhi_k)_k$ is an i.i.d. sequence taking values in $\X^r$ with common probability distribution $G_0$;
\item  $(W_{k})_{k}$
are determined via the stick breaking construction
\[
W_{ik}:=S_{ik} \prod_{j <k} (1-S_{ij}), \quad i=1,\ldots,r
\]
where $\prod_{i}^j=1$ for $i>j$ and
$S_k=(S_{1k},\ldots,S_{rk})$ are stochastically independent
random vectors taking values in $[0,1]^r$ such that $\sum_{k \geq 1} W_{ik}=1$ a.s. for every $i$.
\end{itemize}

Note that $(\TG_1,\dots,\TG_r)$ is a  vector of dependent random measures whenever $(S_{11},\ldots,S_{r1})$ or $(\PPhi_{11},\ldots,\PPhi_{r1})$ are vectors of dependent random variables. The dependence between the measures affects the dependence structure underlying the densities $f_1,\dots,f_r$, which can be represented as infinite mixtures
\begin{equation}\label{mixturemodelDef}
f_i(y)=\sum_{k \geq 1}W_{i k} \,\, \CK_i(y|\PPhi_{i k}) \qquad i=1,\dots,r
\end{equation}
functions of the atoms $(\PPhi_k)_k$ and the weights $(W_k)_k$.

The above definition of dependent random measures is quite general. For the sake of completeness, we shall notice that our specification of vectors of
stick-breaking  processes can be extended, even if not straightforwardly, up to include more
rich structure such as the matrix of stick-breaking processes proposed in  \cite{DunsonXueCarin2008}.
In the rest of this section we briefly discuss the choice of atoms $\{\PPhi_k\}_{k}$ and analyze some general features of vectors of stick breaking processes.
While in the next section we focus  on the main contribution of this works which is a new specification of the stick-vectors $\{S_k\}_{k}$ based on multivariate beta distribution.

\subsection{Atoms}
The simplest assumption for the atoms is that they are common to all the  measures ${G}_{i}$.
Otherwise stated this means that the base measure of the atoms is
\begin{equation}\label{atoms1}
G_0(A_1 \times A_2 \cdots \times A_r)=H_0(\bigcap_i A_i)
\end{equation}
for every $A_1,\dots,A_r$ measurable subsets of $\X$, $H_0$ being a probability measure on $\X$,
which corresponds to the case
\begin{equation}\label{atoms2}
(\PPhi_{1k},\ldots,\PPhi_{rk})=(\PPhi_{0k},\PPhi_{0k},\ldots,\PPhi_{0k})
\end{equation}
with $\PPhi_{0k}$ distributed according to $H_0$.

Eventually one can choose a more complex structure for the law of the atoms, including
covariates (or exogenous effects) related to the specific block $i$.
For instance one can assume an ANOVA-like scheme of the form
\begin{equation}\label{atoms3}
(\PPhi_{1k},\ldots,\PPhi_{rk})=(\tilde \mu_{0k}+\tilde \mu_{1k},\tilde \mu_{0k}+\tilde \mu_{2k},\ldots,
\tilde \mu_{0k}+\tilde \mu_{rk})
\end{equation}
where $\tilde \mu_{0k}$ represents the overall ''mean'' (of the  $k$-th mixture component)
and $\tilde \mu_{ik}$ the specific  ''mean'' for factor $i$ (of
the $k$-th mixture component). A similar choice has been used in  \cite{DeIorio2004}.

In many situations it is reasonable to assume that the components of the mixture
are essentially the same for all the blocks but that they have different weights.
In addition, the choice \eqref{atoms1}-\eqref{atoms2} yields a simple form of the correlation between the related random measure. This feature,
which may be useful in the parameter elicitation, is discussed in the next subsection.

\subsection{Correlation and Dependence Structure}

In the general definition given in Subsection \ref{SS:defstick}
the vectors $S_k$ of stick variables are assumed to be independent.
If in addition one assumes that they have the same distribution, i.e. that
\begin{equation}\label{indepStick}
\text{ $(S_k)_{k \geq 1}$ is a sequence of i.i.d. random vectors},
\end{equation}
then it is easy to compute the correlation of two elements of the vector
$(\TG_1,\dots,\TG_r)$
as well as the correlation between $f_i(y)$ and $f_j(y)$.

For the shake of simplicity we shall consider only the case $i=1,j=2$ and set
\begin{equation}\label{defCor}
C_{1,2}:=\frac{\E[S_{11}S_{21}]}{1-\E[(1-S_{11})(1-S_{21})]} \sqrt{\frac{(2\E[S_{11}]-\E[S_{11}^2])(2\E[S_{21}]-\E[S_{21}^2])}{ \E[S_{11}^2]\E[S_{21}^2]}}.
\end{equation}
and, for every $y$,
\begin{equation*}
\begin{split}
\k_{G_{01}}(y)&:=\int \mathcal{K}_2(y\mid x)G_{01}(dx), \quad
\k_{G_{02}}(y):=\int \mathcal{K}_2(y\mid x)G_{02}(dx),\\
\k_{G_{0}}(y)&:=\int \mathcal{K}_1(y\mid x_1)\mathcal{K}_2(y\mid x_2)G_{0}(dx_1\times dx_2\times\X^{r-2})\\
\end{split}
\end{equation*}
where $G_{0i}$ denotes the $i$-th marginal of $G_{0}$.

\begin{proposition}\label{Prop1} Assume that \eqref{indepStick} holds true, then for all measurable set $A$ and $B$
\begin{equation}\label{corr}
Cor(\TG_1(A), \TG_2(B))=
C_{1,2}  \times \frac{G_0(A \times B\times\X^{r-2})-G_{01}(A) G_{02}(B)}{\sqrt{G_{01}(A)(1-G_{01}(A))G_{02}(B)(1-G_{02}(B))}}
\end{equation}
and for every $y$ in $\Y$
\begin{equation}
\begin{split}
Cor(f_1(y),f_2(y)) =& C_{1,2} \times \frac{\k_{G_{0}}(y)-\k_{G_{01}}(y) \k_{G_{02}}(y)}{\sqrt{\k_{G_{01}}(y)(1-\k_{G_{01}}(y))\k_{G_{02}}(y)(1-\k_{G_{02}}(y))}}.
\end{split}
\end{equation}
\end{proposition}

The proof of Proposition \ref{Prop1} is in Appendix.

\begin{corollary} Assume that  \eqref{atoms2} and \eqref{indepStick} hold true, then
for every measurable set $A$
\begin{equation}\label{corr2}
Cor(\TG_1(A),\TG_2(A))= C_{1,2}.
\end{equation}
where $C_{1,2}$ is defined in \eqref{defCor}.
\end{corollary}

\subsection{Partial exchangeability}
We conclude this section
by observing that \([Y_{ij}: i=1,\dots r, j\geq 1]\)
is a partially exchangeable random array, indeed the joint law of the infinite process of observation
is characterized by:
\begin{equation}\label{parexc}
P\{ Y_{ij} \in A_{ij} \,\, i=1,\dots r, j=1,\dots,n_i\}=
\E\left[ \prod_{i=1}^r \left[\int_{\X} \prod_{j=1}^{n_i} \int_{A_{ij}} \CK_i(y_{ij}|x_{i}) \nu(dy_{ij}) \TG_i(dx_i)\right] \right],
\end{equation}
where the expectation is respect to the joint law of $(\TG_1,\dots,\TG_r)$.
Recall that an array $[Y_{ij}: i=1,\dots r, j\geq 1]$ is said to be
row-wise partially exchangeable if, for every $n>1$, every measurable sets $A_{ij}$
and any permutations $\rho_1,\dots,\rho_r$ of $\{1,\dots,n\}$,
\[
P\{ Y_{ij} \in A_{ij} \,\, i=1,\dots r, j=1,\dots,n\}
=P\{ Y_{i\rho_i(j)} \in A_{ij} \,\, i=1,\dots r, j=1,\dots,n\}.
\]
In other words,  the joint law is not necessarily invariant to permutations of observations
from different groups. From a practical point of view, the partial exchangeability  represents a
suitable model for sets of data that exhibit sub-samples with some possibly different features.

\section{Beta-Product Poisson-Dirichlet Processes}\label{Sec4}

We propose a new class of vector of dependent probability
measures $(\TG_1,\dots,\TG_r)$ in such a way
that (marginally) $\TG_i$ is a Dirichlet process  for every $i$.
This result follows from the Sethurman's representation \eqref{Setu} if one considers
a  multivariate distribution for $(S_{1k},\dots,S_{rk})$ such that
$S_{ik} \sim Beta(1,\alpha_{i}),\,\forall\,i,k$.

It is worth noticing that there are many possible definitions of  multivariate beta distribution
(see for example \cite{OlkinLiu:2003} and \cite{NadarajahKotz:2005}), but not all of them has a tractable stochastic representation and leads to
simple Bayesian inference procedures. For this reason we follow \cite{NadarajahKotz:2005} and consider a suitable product of
independent beta random variables. More specifically we apply the following result.
\begin{proposition}[\cite{RadhakrishnaRao}]
If $U_1,U_2,\dots,U_p$ are independent beta random variables with shape parameters $(a_i,b_i)$, $i=1,2,\dots,p$ and if $a_{i+1}= a_i + b_i$, $i=1,2,\dots,p-1$, then the product $U_1U_2 \cdots U_p$ is also a beta random variable with  parameters $(a_1,b_1+\dots+b_p)$.
\label{PropKrysicki}
\end{proposition}

\begin{proof}
It is easy to check that the Mellin transform of a beta random variable $Z$ of parameters $(a,b)$ is given by
\[
\CM_Z(s):=\E[Z^s]=\frac{\Gamma(a+b)\Gamma(a+s)}{\Gamma(a)\Gamma(a+b+s)}.
\]
Hence, since $a_2=a_1+b_1$ and using the independence assumption
\[
\CM_{U_1U_2}(s)=\E[U_1^s]\E[U_2^s]=\frac{\Gamma(a_1+b_1)\Gamma(a_1+s)}{\Gamma(a_1)\Gamma(a_1+b_1+s)}
\frac{\Gamma(a_2+b_2)\Gamma(a_2+s)}{\Gamma(a_2)\Gamma(a_2+b_2+s)}
=\frac{\Gamma(a_1+b_1+b_2)\Gamma(a_1+s)}{\Gamma(a_1)\Gamma(a_1+b_1+b_2+s)}.
\]
Which gives the result for $p=2$. The general case follows.
\end{proof}

We obtain two alternative specifications of the multidimensional beta variables. Specifically, if we set
\begin{equation}\label{betastick2B}
(S_{1k},S_{2k},\dots,S_{rk}):=(V_{0k}V_{1k},V_{0k}V_{2k},\dots,V_{0k}V_{rk})
\end{equation}
with $V_{0k},\dots,V_{rk}$ independent, $V_{ik} \sim Beta(\alpha_{0k},\alpha_{1k})$ ($i=1,2,\dots,r$)
and $V_{0k} \sim Beta(\alpha_{0k}+\alpha_{1k},\alpha_{2k})$, then $S_{ik} \sim Beta(\alpha_{0k},\alpha_{1k}+\alpha_{2k})$.

As an alternative we consider
\begin{equation}\label{betastick3B}
(S_{1k},S_{2k},\dots,S_{rk}):=(V_{0k}V_{1k} \dots V_{r-1 k},V_{0k}V_{1k}\dots V_{r-2 k}, \dots, V_{0 k})
\end{equation}
with $V_{0k},\dots,V_{r-1 k}$  independent and $V_{ik} \sim Beta(\alpha_{0k}+\dots+\alpha_{i k},\alpha_{i+1 k})$,
$i=0,\dots,r-1$, that gives $S_{ik} \sim Beta(\alpha_{0k},\alpha_{1k}+\dots +\alpha_{r+1-i, k})$.

It should be noted that \eqref{betastick2B} resembles the specification of the matrix stick-breaking process in \cite{DunsonXueCarin2008}. In that paper all the components of the stick-breaking process are products of two independent beta variables with fixed parameters $(1,\alpha)$ and $(1,\beta)$, that precludes obtaining Poisson-Dirichlet marginals. For this reason we propose the specifications in  \eqref{betastick2B} and \eqref{betastick3B} that, for a special choice of the parameters $\alpha_{i k}$, allow for Poisson-Dirichlet marginals. In the first construction we obtain a random vector with identical marginal distributions, while in the second construction the vector has different marginals. Moreover, in the second case $S_{1k} \leq S_{2k} \leq \dots \leq S_{rk}$, which induces an ordering on the concentration parameters of the Dirichlet marginals. These aspects will be further discussed in the following sections.

\subsection{Dirichlet Process marginal.}\label{SS:dirichletmarginal}
For the sake of clarity, we start with $r=2$.  We assume \eqref{indepStick}  and we discuss how to choose the parameters in \eqref{betastick2B}-\eqref{betastick3B} in order to get
DP marginals. Note that since  \eqref{indepStick}  holds true,
then $S_k=(S_{1k},S_{2k})\sim (S_{11},S_{21})$.
According to the construction schemes given in \eqref{betastick2B} and \eqref{betastick3B},
with $\alpha_{0k}=1$, $\alpha_{1k}=\a_1$, $\alpha_{2k}=\a_2$, two possible and alternative specifications of $(S_{11},S_{21})$ are:
\begin{itemize}
\item[(H1)] {\it $(S_{11},S_{21}):=(V_{0}V_{1},V_0V_2)$, with $V_0,V_1,V_2$ independent,
$V_0 \sim Beta(1+\a_1,\a_2)$  and
$V_i \sim Beta(1,\a_1)$, $i=1,2$, where $\a_1>0$ and $\a_2>0$};
\item[(H2)] {\it  $(S_{11},S_{21}):=(V_{0}V_{1},V_0)$, with $V_0,V_1$ independent,
$V_0 \sim Beta(1,\a_1)$ and  $V_1 \sim Beta(1+\a_1,\a_2)$ with $\a_1>0$ and $\a_2>0$}.
\end{itemize}
Thanks to Proposition \ref{PropKrysicki}, if (H1) holds, then $S_{11}  \sim Beta(1,\a_1+\a_2)$  and $S_{21} \sim Beta(1,\a_1+\a_2)$, while if (H2) holds, then $S_{11} \sim Beta(1,\a_1+\a_2)$ and $S_{21}=V_0 \sim   Beta(1,\a_1)$. Hence, we have the following result.

\begin{proposition}\label{P:dirichmarg}
Under \eqref{indepStick}, if {\rm(H1)} holds true,  $\TG_1$ and $\TG_2$ are (marginally) Dirichlet processes with the same precision parameter
$\a_1+\a_2$ and base measures  $G_{01}$, $G_{02}$ respectively. If {\rm(H2)} holds true, then the first component of $(\TG_1,\TG_2)$ is a Dirichlet process with precision parameter $\a_1+\a_2$ and base measure $G_{01}$ and the second component is a Dirichlet process with precision parameter $\a_1$ and base measure $G_{02}$.
\end{proposition}

Since with this construction the $\TG_i$'s are Dirichlet processes, we call $(\TG_1,\TG_2)$
Beta-Product Dependent Dirichlet Process of parameters $\XXi=(\a_1,\a_2)$ and base measure $G_0$, in short $\beta_{i}\!-\!\hbox{DDP}(\XXi,G_0)$, where $i=1$ for (H1) and $i=2$ for (H2). In addition, when we assume \eqref{atoms1}, we denote the resulting process with $\beta_{i}\!-\!\hbox{DDP}(\XXi,H_0)$.

It should be noted that the two processes have different marginal behaviors. The $\beta_{1}\!-\!\hbox{DDP}(\XXi,G_0)$ process has marginals with the same precision
parameter and should be used as a prior when the clustering along the different vector
dimension is expected to be similar. In the $\beta_{2}\!-\!\hbox{DDP}(\XXi,G_0)$ process, the precision parameter decreases along the vector dimension. This process should be used as a prior when a priori one suspects that the clustering features are different in the subgroups of observations.

\begin{figure}[t]
\begin{centering}
\begin{tabular}{cc}
\multicolumn{2}{c}{\textbf{Correlation under H1}}\\
\includegraphics[height=5cm,width=7cm, angle=0, clip=false]{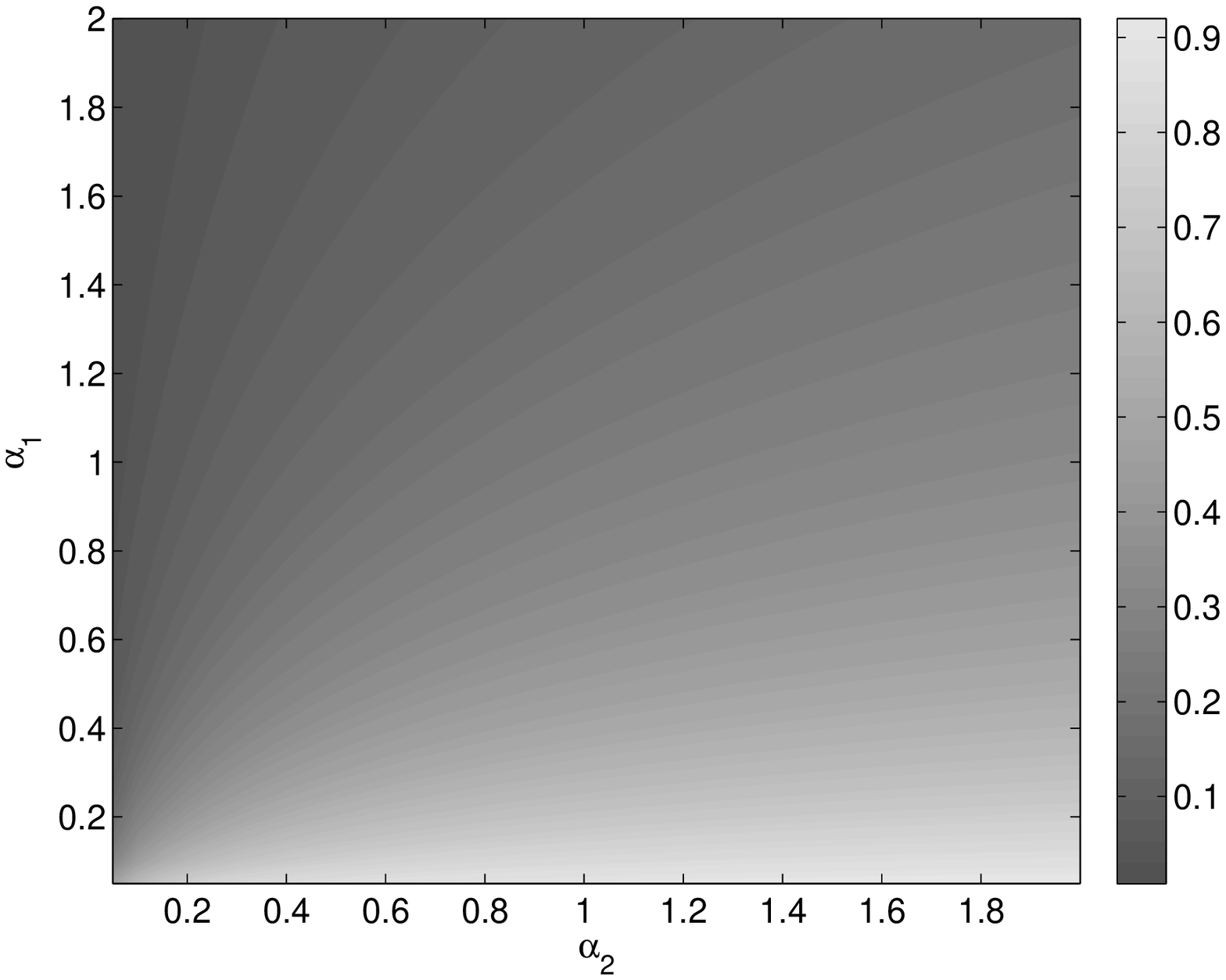}&\includegraphics[height=5cm,width=7cm, angle=0, clip=false]{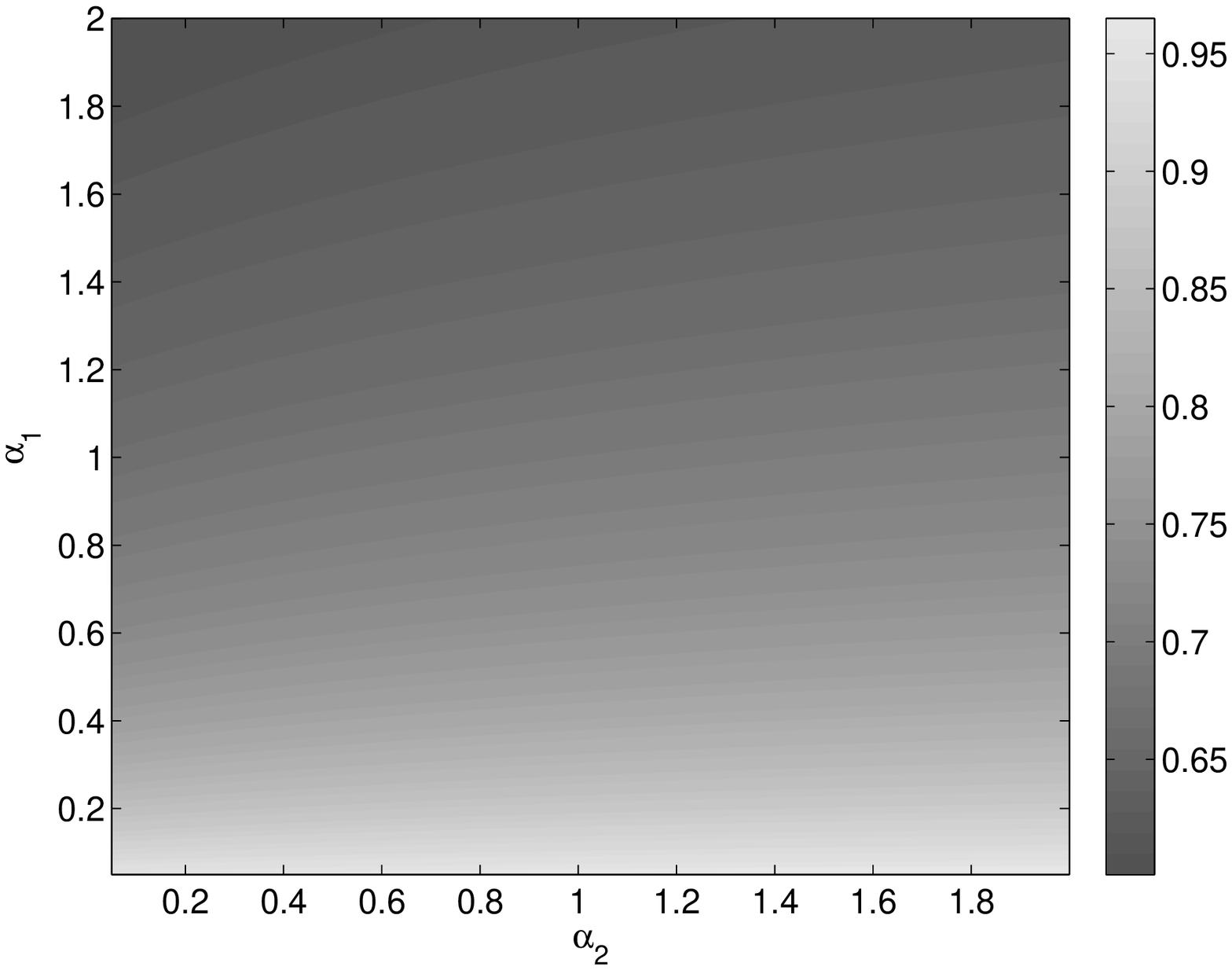}\\
\multicolumn{2}{c}{\textbf{Correlation under H2}}\\
\includegraphics[height=5cm,width=7cm, angle=0, clip=false]{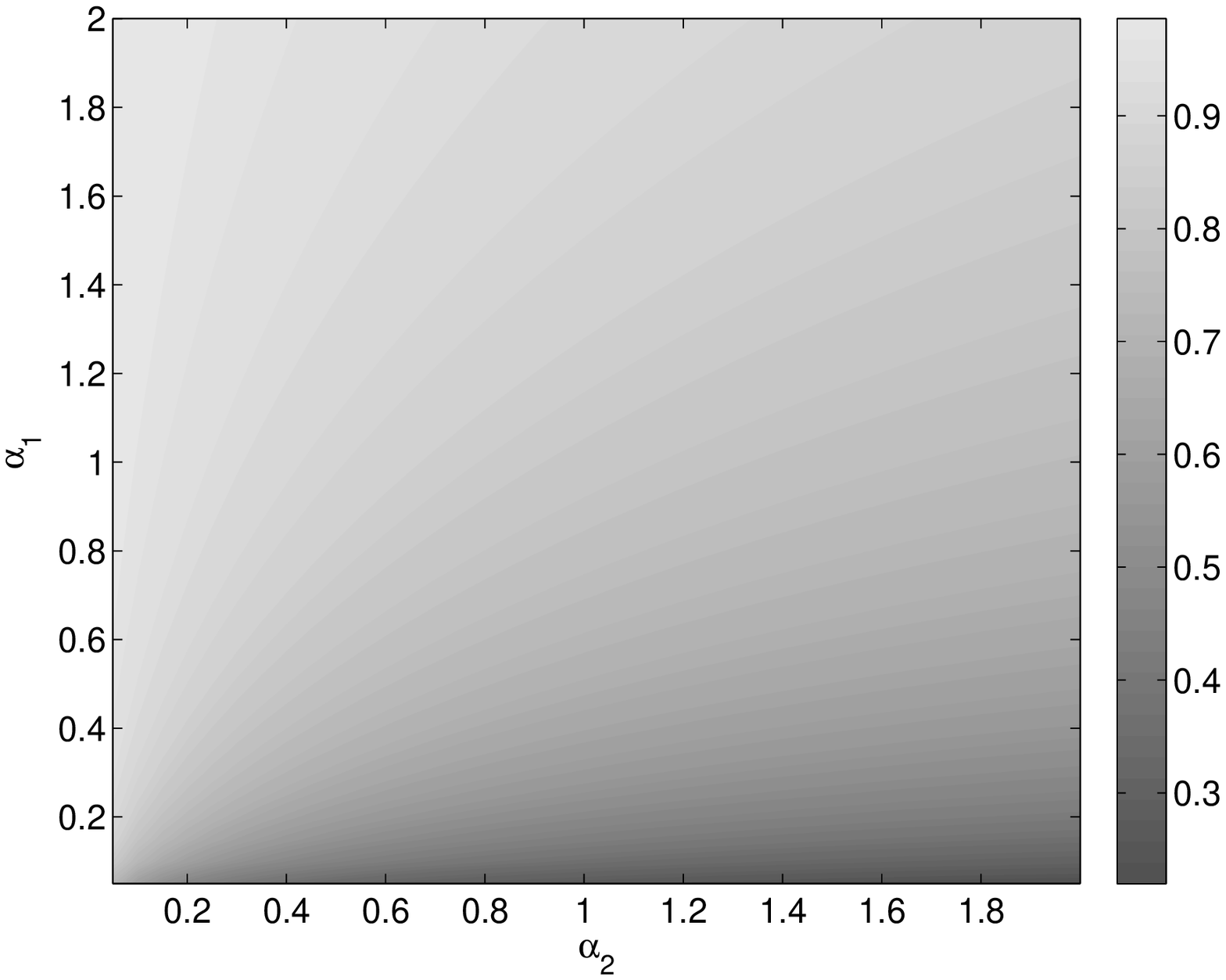}&\includegraphics[height=5cm,width=7cm, angle=0, clip=false]{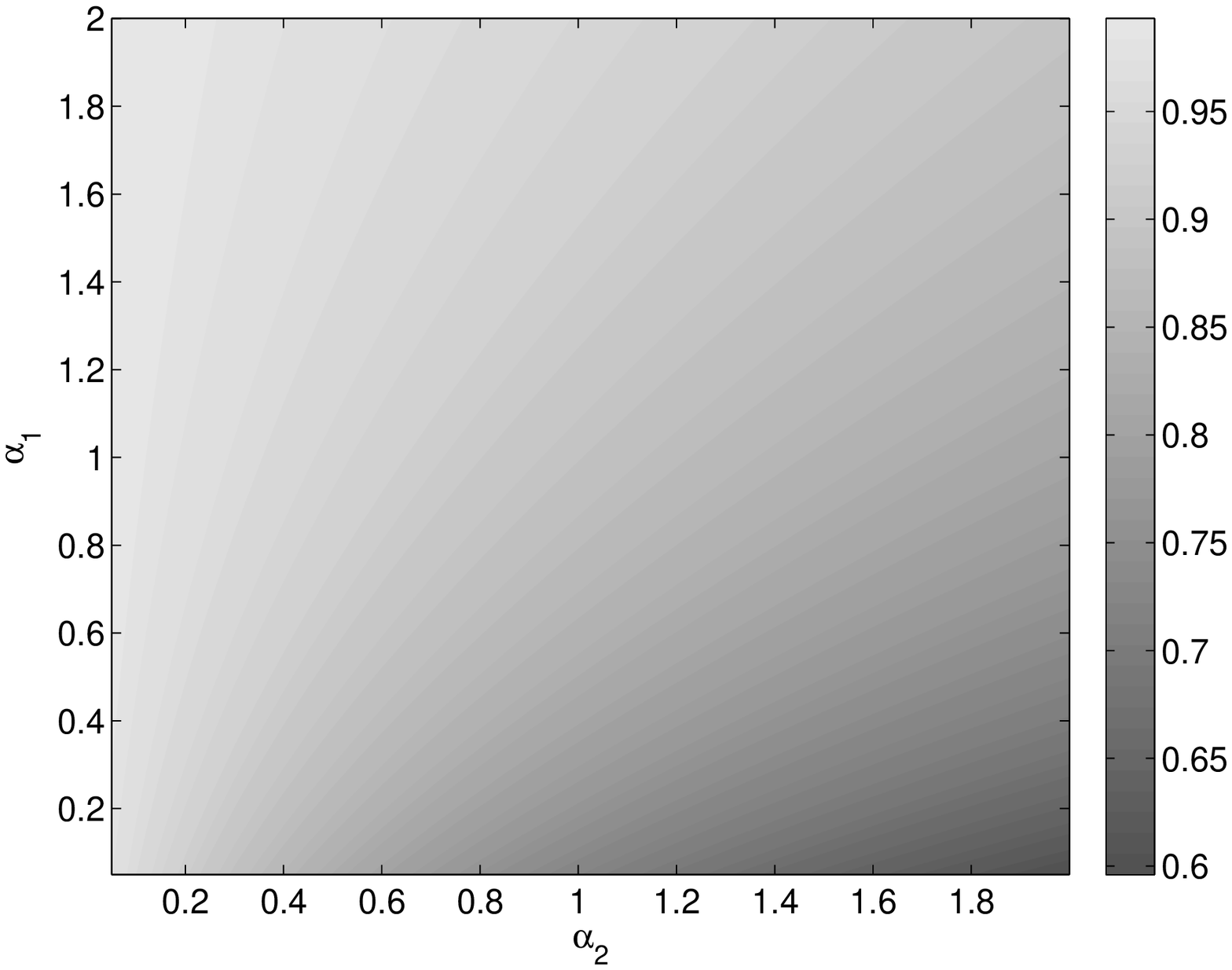}\\
\end{tabular}
\caption{Left column: correlation between $S_{11}$ and $S_{21}$ under (H1) (first row) and (H2) (second row). Right column: correlation between
${G}_1$ and ${G}_2$,  assuming \eqref{atoms1}, under (H1) (first row) and (H2) (second row)}\label{cor}
\end{centering}
\end{figure}

For parameter elicitation purposes, it is useful to analyze how the choice of $(\alpha_1,\alpha_2)$ affects the correlation
between ${\TG}_1$ and ${\TG}_2$. Let us start by considering the correlation between the stick variables.
From Theorem 4 and 6 in \cite{NadarajahKotz:2005}, one obtains the following correlation between the components
$S_{1h}$ and $S_{2h}$ in the  cases of (H1) and (H2)
$$
Cor(S_{11},S_{21})=\left\{%
\begin{array}{ll}
\frac{\alpha_2}{(\a_1+1)(\a_1+\a_2)}                          &\hbox{for}\quad (H1)\\
\sqrt {\frac{\a_1(2+\a_1+\a_2)}{(\a_1+\a_2)(2+\a_1)}}            &\hbox{for}\quad (H2).\\
\end{array}
\right.
$$

Fig. \ref{cor} shows the correlation level between the stick-breaking components (left column) and the random measures (right column) for different values of $\alpha_{1}$ and $\alpha_{2}$. In these graphs, the white color is used for correlation values equal to one and the
black is used for a correlation value equal to zero. The gray areas represent correlation values in the $(0,1)$ interval under both (H1) and (H2) beta models. According to the graph at the top-left of Fig. \ref{cor}, one can conclude that the parametrization used in this paper allows for covering the whole range of possible correlation values in the (0,1) interval. For instance, under (H1), a low correlation between the components of the stick-breaking corresponds to low values of the parameter $\a_1$, say between 0 and 0.1, for any choice of the parameter $\a_2$.

\begin{corollary}\label{corrH} Under the same assumptions of {\rm Proposition \ref{P:dirichmarg}},
\[
C_{1,2}=\left\{%
\begin{array}{ll}
\left(\frac{(\a_1+\a_2+1)(\a_1+2)}{2(\a_1+1)(\a_1+\a_2+1)-(\a_1+2)}\right)&\hbox{for}\quad (H1)\\
 \frac{2(\a_1+1)}{(\a_1+2)(2\a_1+\a_2+1)-\a_1} \sqrt{(\a_1+1)(\a_1+\a_2+1)} &\hbox{for}\quad (H2).\\
\end{array}
\right .
\]
\end{corollary}

Recall that if
\eqref{atoms1} holds true,
then for any measurable set $A$:
$$
Cor(\TG_1(A),\TG_2(A))=C_{1,2}.
$$
The graphs at the bottom of Fig. \ref{cor} show how the parameters $\a_1$ and $\a_2$ affect the correlation between the components $S_{1h}$ and $S_{2h}$
of the bivariate beta used in the stick breaking process and the correlation between the random measures
$\TG_{1}$ and $\TG_{2}$ -- when assuming \eqref{atoms1}.

It is worth noticing that, under (H1), $(S_{11},S_{21})$ converges (in distribution)
to $(V_1,V_2)$ as $\a_2 \to 0$, where $V_1$ and $V_2$ independent random variables with  distribution $Beta(1,\a_1)$.
While, under (H2), $(S_{11},S_{21})$ converges to $(V_0,V_0)$ as $\a_2 \to 0$, where $V_0$ is a $Beta(1,\a_1)$ random variable.
In particular, if one assumes \eqref{atoms1} and (H2), when
$\a_2 \to 0$, one gets the limit situation in which all the observations are sampled from
a common mixture of Dirichlet processes. In other words, in this limit case, as can be seen by \eqref{parexc}
for $\TG_1=\TG_2$,  one considers
the observations (globally) exchangeable, so no distinction between the two blocks are allowed.
The other limiting case is when one assumes (H1) and takes
\((\PPhi_{1k},\PPhi_{2k})\) to be independent random elements with probability distribution $G_{01}$ and $G_{02}$. In the limit
for $\a_2 \to 0$, one obtains two independent Dirichlet processes $\TG_1$ and $\TG_2$
with base measures $G_{01}$ and $G_{02}$. In other words,
with this choice, one considers the blocks of observations as two independent blocks of exchangeable variables and no sharing of information is allowed.

The (H1) construction for the case $r>2$ follows immediately from \eqref{betastick2B}
with $\alpha_{0k}=1$, $\alpha_{1k}=\a_1$, $\alpha_{2k}=\a_2$, that is
 by assuming $V_ {0k},V_{1k}$, $V_{2k}$, $\dots,V_{rk}$ to be independent with $V_{0k}\sim Beta(1+\a_1,\a_2)$ and $V_{ik} \sim Beta(1,\a_1)$, $i=1,2,\dots,r$, where $\a_1>0$ and $\a_2>0$. In this case, $S_{ik}$ has $Beta(1,\a_1+\a_2)$ distribution for $i=1,\dots,r$ and hence
$\TG_i$ is marginally $DP(\a_1+\a_2,G_{0i})$. Also the formula for the correlation between two measures easily extends to the case $r>2$ under \eqref{atoms1}:
\[
Corr(\TG_i(A),\TG_j(A))=
\frac{(\a_1+\a_2+1)(\a_1+2)}{2(\a_1+1)(\a_1+\a_2+1)-(\a_1+2)}.
\]
The (H2) construction extends to $r>2$ by setting in \eqref{betastick3B}
$\alpha_{0k}=1$ and $\alpha_{ik}=\alpha_i$ ($i \geq 1$),
that is by taking $V_{0k},\dots,V_{r-1\,k}$ to be independent random variables with $V_{0k} \sim Beta(1,\alpha_1)$, $V_{1k} \sim Beta(1+\alpha_1,\alpha_2)$,
$\dots$,  $V_{r-1\,k} \sim Beta(1+\alpha_1+\dots+\alpha_{r-1},\alpha_{r})$, where $\alpha_i>0$ for all $i=1,\dots,r$. In this last case  $S_{ik} \sim Beta(1,\alpha_1+\dots+\alpha_{r-i+1})$ for every $i=1,\dots,r$. Hence $\TG_i$ is marginally $DP(\alpha_1+\dots+\alpha_{r-i+1},G_{0i})$.
Under \eqref{atoms1} the correlation between $\TG_{i}(A)$ and $\TG_{j}(A)$ with $1 \leq i <j \leq r$ is given
by
\begin{equation}\label{corr-r>2H2}
\begin{split}
& Corr(\TG_{i}(A),\TG_{j}(A)) \\
&\quad =\frac{2\sqrt{(1+\alpha_1+\dots+\alpha_{r-i+1})} (1+\alpha_1+\dots+\alpha_{r-j+1})^{\frac{3}{2}}}{2(1+\alpha_1+\dots+\alpha_{r-j+1})^2+(2+\alpha_1+\dots+\alpha_{r-j+1})(\alpha_{r-j+2}+\dots+\alpha_{r-i+1})}. \\
\end{split}
\end{equation}
The proof of this last result is given in the Appendix.

\subsection{Poisson-Dirichlet process marginal}
{
Recall that a Poisson--Dirichlet process, $PD(\alpha,l,H_0)$, with parameters $0\leq l<1$ and $\alpha>-l$, and base measure $H_0$, is obtained by taking in \eqref{Setu}-\eqref{stickmonodim} a sequence
of independent random variables $S_k$ with $S_k \sim Beta(1-l,\alpha+lk)$.
In this section we show that by a suitable  choice of the parameters in \eqref{betastick2B}-\eqref{betastick3B} we obtain a vectors of dependent random measures with Poisson-Dirichlet marginals.

In the first case use \eqref{betastick2B} with $\alpha_{0k}=1-l$, $\alpha_{1k}=\a_1$ and $\alpha_{2k}=\a_2+lk$, that is
take  $V_{0k},\dots,V_{rk}$ to be independent random variables such that
\begin{equation}\label{H1PD}
V_{0k} \sim Beta(1-l+\a_1,\a_2+lk), \quad V_{ik} \sim Beta(1-l,\a_1) \quad i=1,\dots,r
\end{equation}
where $\a_1>0$, $\a_2>0$, and $0\leq l<1$.
Proposition \ref{PropKrysicki} yields
that $S_{ik}  \sim Beta(1-l,\a_1+\a_2+lk)$.
In the second case use \eqref{betastick3B} with $\alpha_{0k}=1-l$, $\alpha_{1k}=\alpha_1+lk$, and $\alpha_{ik}=\alpha_i$ for $i\geq 2$,
that is take
$V_{0k},\dots,V_{r-1 k}$ to be independent random variables such that
\begin{equation}\label{H2PD}
\begin{split}
& V_{0k} \sim Beta(1-l,\alpha_1+lk), V_{1k} \sim Beta(1+\alpha_1+l(k-1),\alpha_2), \dots, \\
& \dots V_{r-1 k} \sim Beta(1+\alpha_1+\dots+\alpha_{r-1}+l(k-1),\alpha_{r})\\
\end{split}
\end{equation}
 with $\alpha_i>0$ $i=0,\dots,r$} and $0\leq l<1$.
In this last case $S_{ik}$ has $Beta(1-l,\alpha_1+\dots+\alpha_{r-i+1}+lk)$ for every $i=1,\dots,r$.

Summarizing we have proved the following

\begin{proposition}
If \eqref{betastick2B} and \eqref{H1PD}
hold true  $\TG_i$ is a $PD(\a_1+\a_2,l,G_{0i})$ for every $i=1,\dots,r$,
while if \eqref{betastick3B} and \eqref{H2PD}
hold true  $\TG_i$ is a $PD(\alpha_1+\dots+\alpha_{r-i+1},l,G_{0i})$ for every $i=1,\dots,r$.
\end{proposition}

\section{Slice Sampling Algorithm for Posterior Simulation}\label{S:Slice}

For posterior computation, we propose an extension of the slice sampling algorithm
introduced in \cite{walker2007,walker2011}.  For the sake of simplicity
we shall describe the sampling strategy for a  vector of  {Beta-Product DDP} with $r=2$ ($\beta_{i}\!-\!\hbox{DDP}(\XXi,G_0)$),
see Subsection \ref{SS:dirichletmarginal}.
The proposed algorithm can be easily extend to the case $r>2$ and to the Beta-Product dependent Poisson-DP.

Recall that in $\beta_{i}\!-\!\hbox{DDP}(\XXi,G_0)$
the stick variables are defined by
\[
(S_{1k},S_{2k}):=(V_{0k}V_{1k},V_{0k}V_{2k})
\]
for a sequence of independent vectors $V_k:=(V_{0k},V_{1k},V_{2k})$ with the same distribution
of $(V_0,V_1,V_2)$ and the convention $V_{2k}=1$
and $V_k:=(V_{0,k},V_{1,k})$ under (H2).

Starting from \eqref{mixturemodelDef}, the key idea of the slice sampling is to find a finite number
of  variables to be sampled. First we introduce a latent variable $u$ in such a way that
$f_i(y)$ is the marginal density of
\[
 f_i(y,u)=\sum_{k \geq 1} \J\{ u \leq W_{i k} \} \,\, \CK_i(y|\PPhi_{i k}).
\]
It is clear that given $u$, the number of components is finite. In addition
we introduce a further latent variable $d$ which indicates which of these finite number of components provides
the observation, that is
\begin{equation}\label{densitacond}
f_{i}(y,u,d) := \J\{u < W_{i,d}  \}  \CK_i(y|\PPhi_{d}).
\end{equation}
Hence, the likelihood function for the augmented variables $(y,u,d)$
is available as a simple product of terms and crucially $d$ is finite.

To be more precise we introduce the allocation variables $D_{ij}$ ($i=1,2;j=1,\dots,n_{i}$)
taking values in $\N$ and the slice variables $U_{ij}$ ($i=1,2;j=1,\dots,n_{i}$)
taking values on $[0,1]$. We shall use the notation
\[
Y_i^{(n_i)}:=(Y_{i1},\dots,Y_{in_i}), \quad  D_i^{(n_i)}:=(D_{i1},\dots,D_{in_i}), \quad
U_i^{(n_i)}:=(U_{i1},\dots,U_{in_i})
\]
and we write:
$\PPhi$ for $(\PPhi_k)_{k}$, $V$ for $(V_k)_{k}$,
$U$ for $[U_1^{(n_1)},U_2^{(n_2)}]$, $D^{(n)}$ for $[D_1^{(n_1)},D_2^{(n_2)}]$
and $Y$ for $[Y_{1}^{(n_1)},Y_2^{(n_2)}]$.

The random elements $(\PPhi,V)$ have the law already described.  While conditionally on
$(\PPhi,V)$, the random vectors
\((Y_{ij},U_{ij},D_{ij})\), $i=1,2;j=1,\dots,n_i$,
are stochastically independent with the joint density
\eqref{densitacond}.

We conclude by observing that it can be useful to put a prior distribution
even on the hyperparameters $(\a_1,\a_2)$.
The law of $\PPhi$ is assumed independent of the random vector of hyperparameters
$ \tilde \XXi:=(\tilde \a_1,\tilde \a_2)$, while the distribution
of $V_j$ depends on $ \tilde \XXi$ through (H1) or (H2), so we shall write $P\{V_j \in dv_j|\tilde \XXi\}$.

Our target is the exploration of the posterior distribution (given $Y$) of
$[\PPhi, V, U, \tilde \XXi, D]$
 by Markov
Chain Monte Carlo sampling.

 Essentially we shall use a block Gibbs sampler
which iteratively simulates $\PPhi$ given $[V, U, D,\tilde \XXi,Y]$,
$[V, U, \tilde \XXi]$ given $[D,\PPhi,Y]$ and $D$ given $[V, U, \PPhi,\tilde \XXi,Y]$.

For the one dimensional, this blocking structure  case has been introduced in \cite{Papaspiliopoulos2008} and \cite{walker2011} as an alternative and more efficient version of the original algorithm of \cite{walker2007}.
Our algorithm extends the one dimensional slice sampling of \cite{walker2011} to the multidimensional
case. This extension is not trivial as it involves generation of random samples from vectors of allocation
and slice variables of a multivariate stick-breaking process. We present an efficient Gibbs sampling
algorithm by elaborating further on the blocking strategy of \cite{walker2011}.

In order to describe in  details the full-conditionals of the above sketched block Gibbs sampler, we need
some more notation.
Define
for $i=1,2$ and $k \geq 1$,
\[
\begin{split}
& \mathcal{D}_{i,k}:=\{j \in \{1,\dots,n_i\}: D_{i,j}=k\},   \\
& A_{i,k}:=\sum_{j=1}^{n_i} \J\{D_{i,j}=k\}=card(\mathcal{D}_{i,k}), \quad
 B_{i,k}:=\sum_{j=1}^{n_i} \J\{D_{i,j}>k\} . \\
\end{split}
\]
Moreover, let
\begin{equation}\label{mstar}
D^*:=\max_{i=1,2}\max_{1 \leq j \leq n_i} D_{i,j}.
\end{equation}
In our MCMC algorithm we shall treat $V$ as three blocks of random length:
$V=(V^*,V^{**},V^{***})$,
where
\[
V^*=\{V_{k} :  k \in \mathcal{D}^* \}, \qquad V^{**}=\{V_{k} :  k \not \in \mathcal{D}^*,  k \leq  D^*\},
 \qquad V^{***}=\{V_{k} :  k >  D^*\}
\]
and
\( \mathcal{D}^*=\{k:\mathcal{D}_{1,k}\cup\mathcal{D}_{2,k}\neq \emptyset\}\).
Note that $D^*=\max\{ k \in \mathcal{D}^*\}$ and  $ |\mathcal{D}^*| \leq D^*$ almost surely.
In the following subsections we give the details of the full conditionals of the blocking Gibbs sampler,
further details on the algorithm are given in Appendix.

\subsection{The full conditional of $\PPhi$}
The atoms $\PPhi$ given $[V,D,U,\tilde \XXi,Y]$  are conditionally independent and the full conditionals are:
\begin{equation}\label{F1}
\begin{split}
P\{ \PPhi_k \in d \varphi_k & |D,U,V,\tilde \XXi,Y  \}=
P\{ \PPhi_k \in d \varphi_k  |D,Y  \}
\\& \propto G_0(d\varphi_k) \prod_{ j \in \CD_{1,k}} \CK_1 (Y_{1,j}|\varphi_{1k})
 \prod_{ j \in \CD_{2,k}} \CK_2 (Y_{2,j}|\varphi_{2k});\\
 \end{split}
\end{equation}
where $\varphi_k=(\varphi_{1k},\varphi_{2k})$. The strategy for sampling  from this full conditional
depends on the specific form of $\CK_i$ and $G_0$. In the next section we will discuss
a possible strategy for Gaussian kernels.

\subsection{The full conditional of $[V, U, \tilde \XXi]$}\label{step2}
In order to sample from the conditional distribution of $[V, U, \tilde \XXi]$ given $[D,\PPhi,Y]$  a further blocking is used:

\begin{itemize}
\item
{$[V^*,\tilde \XXi]$ given $[D,\PPhi,Y]$.} The joint conditional distribution of $[V^*,\tilde \XXi]$ given $[D,\PPhi,Y]$ is
\begin{equation}\label{F2-A}
\begin{split}
&P\{  V^*\in dv^*, \tilde \XXi \in (d \a_1,d \a_2)| Y,\PPhi,D\}=P\{ V^* \in dv^*, \tilde \XXi \in (d \a_1,
d \a_2) | D\}\\
&\hspace{2cm}\propto \prod_{k  \in \mathcal{D}^*} \frac{Q_k(v_k |D,\a_1,\a_2)}{B(\a_1+1,\a_2)B^2(1,\a_1)} dv_k \pi(d\a_1,d\a_2)  \\
\end{split}
\end{equation}
where $\pi(d\a_1,d\a_2)=P\{\tilde \XXi \in (d\a_1,d\a_2)\}$ is the prior on the concentration parameters and
\begin{equation}\label{densQ}
Q_k(v_k|D,\XXi) :=
\left \{
\begin{split}
& 
\!\! v_{0k}^{\a_1+A_{1k}+A_{2k}} (1-v_{0k})^{\a_2-1} \!\!
\prod_{i=1,2} v_{ik}^{A_{ik}}(1-v_{ik})^{\a_1-1}
(1-v_{0k}v_{ik})^{B_{ik}} \\
& \qquad \qquad \qquad
\text{under (H1) with $v_k=(v_{0k},v_{1k},v_{2k})$} \\
& 
\!\!   v_{0k}^{A_{1k}+A_{2k}}\! (1-v_{0k})^{\a_1+B_{2k}-1} v_{1k}^{\a_1+A_{1k}}
 (1-v_{1k})^{\a_2-1} (1-v_{0k}v_{1k})^{B_{1k}}  \\
& \qquad \qquad \qquad \text{under (H2) with $v_k=(v_{0k},v_{1k})$} \\
\end{split}
\right .
\end{equation}
To  sample from \eqref{F2-A}, we  iterate a two-step Metropolis-Hastings (M.-H.) within Gibbs with full conditionals
\begin{equation}\label{V*givenparameters}
P\{V^{*} \in dv^{*} |\tilde \XXi, D\}\propto \prod_{k  \in \mathcal{D}^*} Q_k(v_k |D, \tilde \XXi)dv_k
\end{equation}
and
\begin{equation}\label{parametersgivenV*}
\begin{split}
P\{ \tilde \XXi \in (d\a_1,d\a_2) | V^{*},D\}\propto &\prod_{k\in \mathcal{D}^*}
\frac{1}{B(\a_1+1,\a_2)}v_{0k}^{\a_1} (1-V_{0k})^{\a_2-1} \\ & \cdot \prod_{i=1,2}
\frac{1}{B(1,\a_1)^2}(1-V_{ik})^{\a_1-1}\pi(d\a_1,d\a_2).\\
\end{split}
\end{equation}
For the each element $(V_{0k},V_{1k},V_{2k})$ of $V^{*}$ we consider a multivariate Gaussian
random walk proposal with diagonal scale matrix $\tau^{2}I_{3}$, with $\tau^{2}=0.05$ in order to have acceptance
rates between 0.3 and 0.5 for the elements of $V^{*}$.

\item  $[V^{**},V^{***}]$ given $[D,V^*,\PPhi,\tilde \XXi,Y]$. The $V_k$ (with $k \not \in \mathcal{D}^*$)
are conditionally independent given $[D,V^*,\PPhi,\tilde \XXi,Y]$ with
$P\{V_k \in dv_k |\tilde \XXi, D,V^*\}\propto  Q_k(v_k |D, \tilde \XXi)dv_k $
if $k \leq D^*$
and $P\{ V_k \in dv | V^*,\PPhi,D,\tilde \XXi,Y\}  = P\{V_k \in dv |\tilde \XXi \}$ if $k > D^*$.
Note that if $k \not \in \mathcal{D}^*$ and $k \leq D^*$, then
$A_{i,k}=0$ in the definition of $Q_k(v_k |D, \tilde \XXi)$. In order to sample
from $Q_k(v_k |D, \tilde \XXi)$
 the same M.-H. step, used for the full conditional in \eqref{V*givenparameters},
 is employed.
\item {$U$ given $[V,D,\PPhi,\tilde \XXi,Y]$.} The slice variables $U$ are conditionally independent given $[V,D,\PPhi,\tilde \XXi,Y]$ with
\begin{equation}\label{F3}
P\{ U_{i,j} \in  d u | V,Y, \PPhi, D\}=P\{U_{i,j}\in du |V,D\}
 =   \frac{\J\{ u \leq W_{i,D_{i,j}}\}}{W_{i,D_{i,j}}} du.
 \end{equation}
\end{itemize}

\subsection{The full conditional of $D$}
The $D$'s are conditionally independent given $[V,U,\PPhi,\tilde \XXi,Y]$
with
\begin{equation}\label{F4}
P\{D_{i,j}=d |\PPhi,V,U,\tilde \XXi,Y \}
\propto \CK_i(Y_{i,j}|\PPhi_{id}) \,\,
 \J\{U_{i,j} \leq W_{i,d}  \}.
\end{equation}
Here an important remark is in order.
As in the slice sampling proposed in \cite{walker2007,walker2011}, the full conditional
\eqref{F4} samples, almost surely, from a finite number of terms. So again it is easy to sample from this full conditional.
More precisely, following \cite{walker2007}, we note that $d> N_{i,j}^*$
 ensures that $W_{i,d} < U_{i,j}$ where $N_{i,j}^*$ $(i=1,2;j=1,\dots,n_i)$ is the smallest integer such that
\begin{equation}\label{condN}
\sum_{k=1}^{N_{i,j}^*} W_{i,k}>1-U_{i,j}.
\end{equation}

\section{Illustrations}\label{sec_results}
We apply our new Beta-product dependent Dirichlet process to make inference for mixture of normals and mixture of vector autoregressive processes. The resulting  model and inference procedure have been applied on both simulated data and real data on the industrial production in the United States and the European Union.

\subsection{ $\beta_{1}\!-\!\hbox{DDP}(\XXi,H_0)$  mixtures of Gaussian distributions}\label{S.S51} $\quad$
In this section, we apply our $\beta_{1}\!-\!\hbox{DDP}(\XXi,H_0)$ Gaussian mixture model for inference
on synthetic data generated from finite Gaussian mixtures. More precisely
we assume \eqref{atoms2}-\eqref{indepStick}, (H1) and Gaussian
kernels $\CK_i( y|\varphi)$ ($i=1,2$) with means $ \mu$ and variance
$\s^2$ for $\varphi=( \mu,\s^{2})$, i.e.
\[
\CK_i( y|\varphi)
=\frac{1}{\sqrt{2 \pi} \s}
\exp \Big \{ -\frac{1}{2\s^2}(y-\mu)^2\Big \}.
\]
As base measure $H_0(d \varphi)$ we take the product of a normal
$\mathcal{N}(0,s^{-2})$ and inverse gamma $\mathcal{I}\mathcal{G}(\lambda,\lambda)$,
which are conjugate distributions for the bivariate kernel at hand.
For the vector $\XXi=(\alpha_1,\alpha_2)$ of the precision parameters of
the bivariate Dirichlet process we consider independent
gamma priors $\mathcal{G}(\zeta_{11},\zeta_{21})\mathcal{G}(\zeta_{12},\zeta_{22})$.
In summary the Bayesian non--parametric model is
\[
\begin{split}
& Y_{ij} | ( \tilde \mu_{ij},\tilde \s_{ij}^{2})  \stackrel{ind}{\sim} \mathcal{N}(\tilde \mu_{ij},\tilde \s_{ij}^2) \qquad i=1,2, j \geq 1 \\
& (\tilde \mu_{ij},\tilde \s_{ij}^2) | \TG_1,\TG_2 \stackrel{iid}{\sim} \TG_i   \quad i=1,2 \\
&  (\TG_1,\TG_2)|\tilde \XXi \sim \beta_{1}\!-\!\hbox{DDP}(\tilde \XXi,H_0) \\
&  \tilde \XXi \sim \mathcal{G}(\zeta_{12},\zeta_{22})\mathcal{G}(\zeta_{12},\zeta_{22}) \\
\end{split}
\]

The sampling procedure for $U$ and $D$ given in the previous section applies straightforwardly to this example.
We shall describe here in more details the sampling strategy for the other unknown quantities of the model.
For the sake of simplicity we will omit indicating the dependence of the full conditional on the hyperparameters.

In order to sample from the full-conditional $P\{ \PPhi_k \in d \varphi_k  | V, D,Y^{(n)},U\}$, for $k \geq 1$
we consider a two-step Gibbs sampler with full conditional distributions
\begin{equation}\label{F1.2Example}
\begin{split}
&P\{\tilde \mu_k \in d \mu_k  |\tilde \s^{-2}_{k}, D,Y^{(n)}\}\propto\\
&\quad\propto \exp\left\{-\frac{1}{2s^{-2}}\mu_{k}^{2}\right\}\prod_{m=1,2}\prod_{i\in\mathcal{D}_{m,k}} \exp\left\{-\frac{1}{2 \tilde \sigma_{k}^{2}}(Y_{mi}-\mu_{k})^{2}\right\}d\mu_{k}\\
&\quad\propto \exp\left\{-\frac{1}{2}\mu_{k}^{2}\left(s^{2}+\tilde \sigma_{k}^{-2}(A_{1,k}+A_{2,k})\right)+\mu_{k}\frac{1}{\tilde \sigma_{k}^{2}}(\eta_{1,k}^{(1)}+\eta_{2,k}^{(1)})\right\}d\mu_{k}
\end{split}
\end{equation}
and
\begin{equation}\label{F1.1Example}
\begin{split}
&P\{\tilde \s^{-2}_k \in d \s^{-2}_k  |\tilde \mu_{k}, D,Y^{(n)}\}\propto\\
&\quad\propto\exp\left\{-\lambda\s_{k}^{-2}\right\}(\s_{k}^{-2})^{\lambda+1}\prod_{m=1,2}\prod_{i\in\mathcal{D}_{m,k}} (\s^{-2}_{k})^{1/2}\exp\left\{-\frac{1}{2\sigma_{k}^{2}}(Y_{mi}-\tilde \mu_{k})^{2}\right\}d\s^{-2}_{k}\\
&\quad\propto\exp\left\{-(\lambda+\frac{1}{2}(\eta_{1,k}^{(2)}+\eta_{2,k}^{(2)}))\s_{k}^{-2}\right\}(\s_{k}^{-2})^{\lambda+\frac{1}{2}(A_{1,k}+A_{2,k})-1}d\s^{-2}_{k}
\end{split}
\end{equation}
which are proportional to the density function of a normal
\begin{equation}
\mathcal{N}\left(\frac{\tilde \sigma_{k}^{-2}(\eta_{1,k}^{(1)}+\eta_{2,k}^{(1)})}{s^{2}+\tilde \sigma_{k}^{-2}(A_{1,k}+A_{2,k})},\frac{1}{s^{2}+\tilde \sigma_{k}^{-2}(A_{1,k}+A_{2,k})}\right)\\
\end{equation}
and an inverse gamma
\begin{equation}
\mathcal{I}\mathcal{G}\left(\lambda+\frac{1}{2}(A_{1,k}+A_{2,k}),\lambda+\frac{1}{2}(\eta_{1,k}^{(2)}+\eta_{2,k}^{(2)})\right)
\end{equation}
respectively, where $A_{m,k}$, $m=1,2$ have been defined in the previous section and
\begin{equation}
\eta_{m,k}^{(1)}=\sum_{i\in\mathcal{D}_{m,k}}Y_{im}\quad\hbox{and}\quad\eta_{m,k}^{(2)}=\sum_{i\in\mathcal{D}_{m,k}}(Y_{im}-\tilde \mu_{k})^{2}.
\end{equation}

A sample from the conditional joint distribution of the precision parameters and the stick breaking elements can be obtained
following the blocking scheme described in Subsection \ref{step2}.
Since we assume gamma priors, $\mathcal{G}(\zeta_{11},\zeta_{21})$ and $\mathcal{G}(\zeta_{12},\zeta_{22})$ for
$\alpha_1$ and $\alpha_2$ respectively, \eqref{parametersgivenV*} becomes
\begin{equation}
\begin{split}
P\{ &\tilde \XXi \in (d\a_1,d\a_2)  | V^{*},D\} \propto \\ & \frac{1}{B(\a_1+1,\a_2)^{d_1}}\frac{1}{B(1,\a_1)^{d_2}}\a_1^{\zeta_{11}-1}\exp\left\{-\a_1\bar{\zeta}_{21}\right\}
\a_2^{\zeta_{12}-1}\exp\left\{-\a_2\bar{\zeta}_{22}\right\}d\a_1d\a_2 \\
\end{split}
\end{equation}
where $\bar{\zeta}_{21}=\zeta_{21}-\sum_{k\in\mathcal{D}}(\log(V_{0k})+\log(1-V_{1k})+\log(1-V_{2k}))$
and $\bar{\zeta}_{22}=\zeta_{22}-\sum_{k\in\mathcal{D}}\log(1-V_{0k})$ and $d_1=card(\mathcal{D}^*)$ and
$d_2=2d_1$. We simulate from the full conditional by a M.-H. step. We considered two alternative proposals.
First we assume independent proposals. At the $j$-th iteration, given $\XXi^{(j-1)}$, we simulate
\begin{equation}
\a_1^{(*)}\sim\mathcal{G}a(\zeta_{11},\zeta_{21}),\quad \a_2^{(*)}\sim\mathcal{G}a(\zeta_{12},\zeta_{22})
\end{equation}
and accept with probability
\begin{equation}
\min\left\{1,\frac{B(\a_1^{(j-1)}+1,\a_2^{(j-1)})^{d_1}B(1,\a_1^{(j-1)})^{d_2}}{B(\a_1^{(*)}+1,\a_2^{(*)})^{d_1}B(1,\a_1^{(*)})^{d_2}}\right\}.
\end{equation}
In our experiments this kind of proposal turns out to be highly inefficient and the M.-H. exhibits low acceptance rates, thus
we consider a gamma random walk proposal. At the $j$-th iteration of the algorithm, given the previous value $(\a_1^{(j-1)},\a_2^{(j-1)})$ of the chain,
we simulate
\begin{equation}
\a_1^{(*)}\sim\mathcal{G}a(\kappa(\a_1^{(j-1)})^{2},\kappa \a_1^{(j-1)}),\quad \a_2^{(*)}\sim\mathcal{G}a(\kappa(\a_2^{(j-1)})^{2},\kappa \a_2^{(j-1)})
\end{equation}
where $\kappa$ represents the scale of random walk.  The proposal is accepted with probability
\begin{equation}
\min\left\{1,\frac{P\{ (\a_1^{(*)},\a_2^{(*)})| V^{*},D\}}{P\{\a_1^{(j-1)},\a_2^{(j-1)}| V^{*},D\}}\frac{g(\a_1^{(j-1)}|\a_1^{(*)})}{g(\a_1^{(*)}|\a_1^{(j-1)})}
\frac{g(\a_2^{(j-1)}|\a_2^{(*)})}{g(\a_2^{(*)}|\a_2^{(j-1)})}\right\}
\end{equation}
where $$g(y|x)=\frac{1}{\Gamma(\kappa x^{2})}(\kappa x)^{\kappa x^{2}}y^{\kappa x^{2}-1}e^{-\kappa x y}$$
is the conditional density of the gamma random walk proposals. We set the scale parameter $\kappa$ in order to have
acceptance rates close to 0.5.

We simulate $n=50$ independent vectors, $(Y_{1,j},Y_{2,j})$ with $j=1,\ldots,n$, of observations.
The component of the vectors $(Y_{1,j},Y_{2,j})$ are independent and alternatively follow  one of these models.
\begin{itemize}
\item The same three-component mixture of normals (model Mix1)
\begin{eqnarray*}
&&\frac{1}{3}\mathcal{N}(-10,1)+\frac{1}{3}\mathcal{N}(0,1)+\frac{1}{3}\mathcal{N}(10,1)\\
&&\frac{1}{3}\mathcal{N}(-10,1)+\frac{1}{3}\mathcal{N}(0,1)+\frac{1}{3}\mathcal{N}(10,1)
\end{eqnarray*}

\item Two different mixtures with two common components (model Mix2)
\begin{eqnarray*}
&&\frac{1}{4}\mathcal{N}(0,0.5)+\frac{1}{4}\mathcal{N}(3,0.25)+\frac{1}{4}\mathcal{N}(2,0.25)+\frac{1}{4}\mathcal{N}(5,0.5)\\
&&\frac{1}{4}\mathcal{N}(0,0.5)+\frac{1}{4}\mathcal{N}(3,0.25)+\frac{1}{4}\mathcal{N}(-3,0.25)+\frac{1}{4}\mathcal{N}(7,0.5)
\end{eqnarray*}

\item The same three-component mixture of normals with different component probabilities (model Mix3)
\begin{eqnarray*}
&&\frac{1}{3}\mathcal{N}(-10,1)+\frac{1}{3}\mathcal{N}(0,1)+\frac{1}{3}\mathcal{N}(10,1)\\
&&\frac{1}{6}\mathcal{N}(-10,1)+\frac{4}{6}\mathcal{N}(0,1)+\frac{1}{6}\mathcal{N}(10,1)
\end{eqnarray*}
\end{itemize}
The simulated set of data considered in the experiments are given in Fig. \ref{data}. In the left column
is the histogram of the first component of the set of data and in the right  column is the histogram of the second component.
\begin{figure}[t]
\begin{centering}
\begin{tabular}{c}
\includegraphics[height=4cm,width=7cm, angle=0, clip=false]{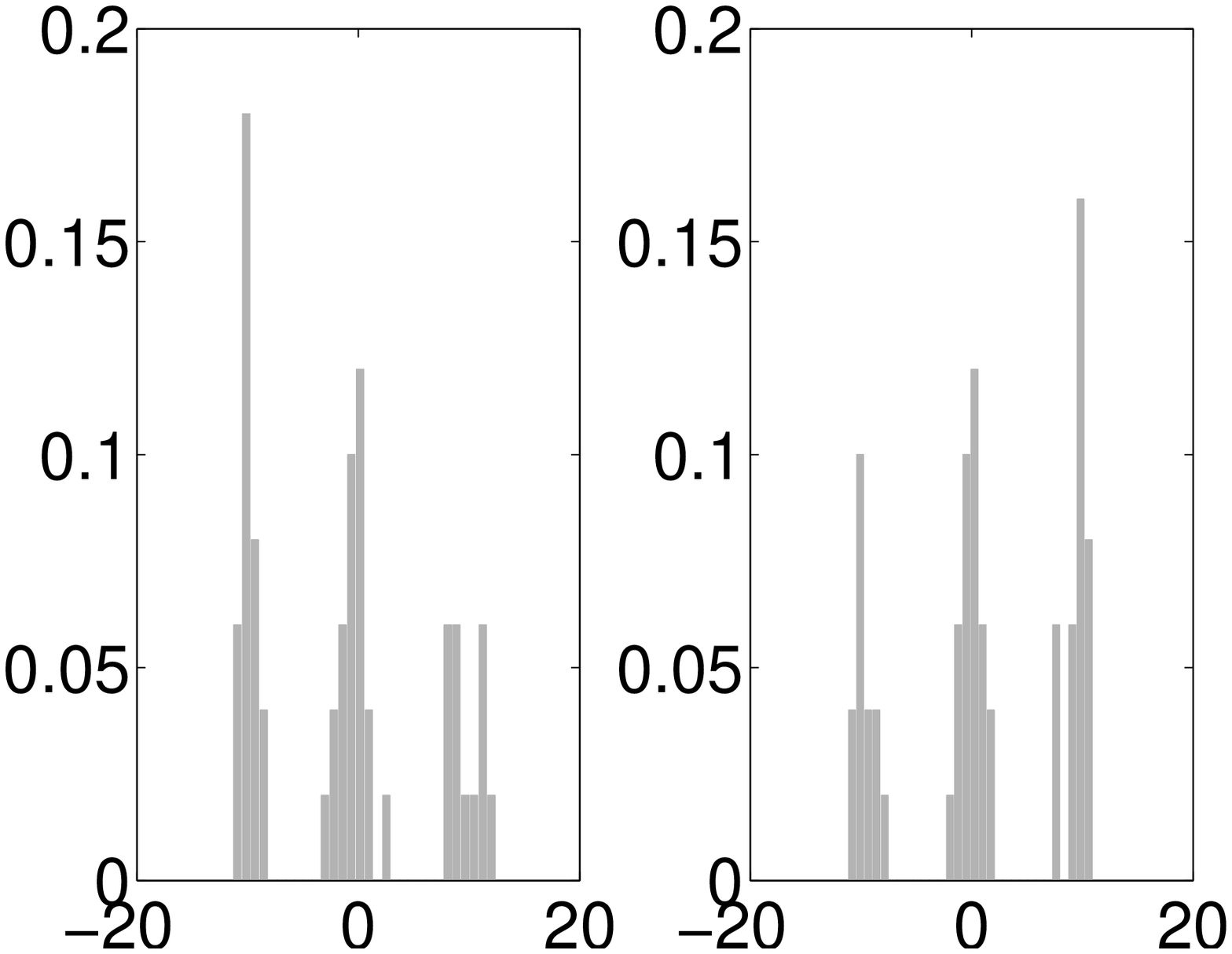}\\
\includegraphics[height=4cm,width=7cm, angle=0, clip=false]{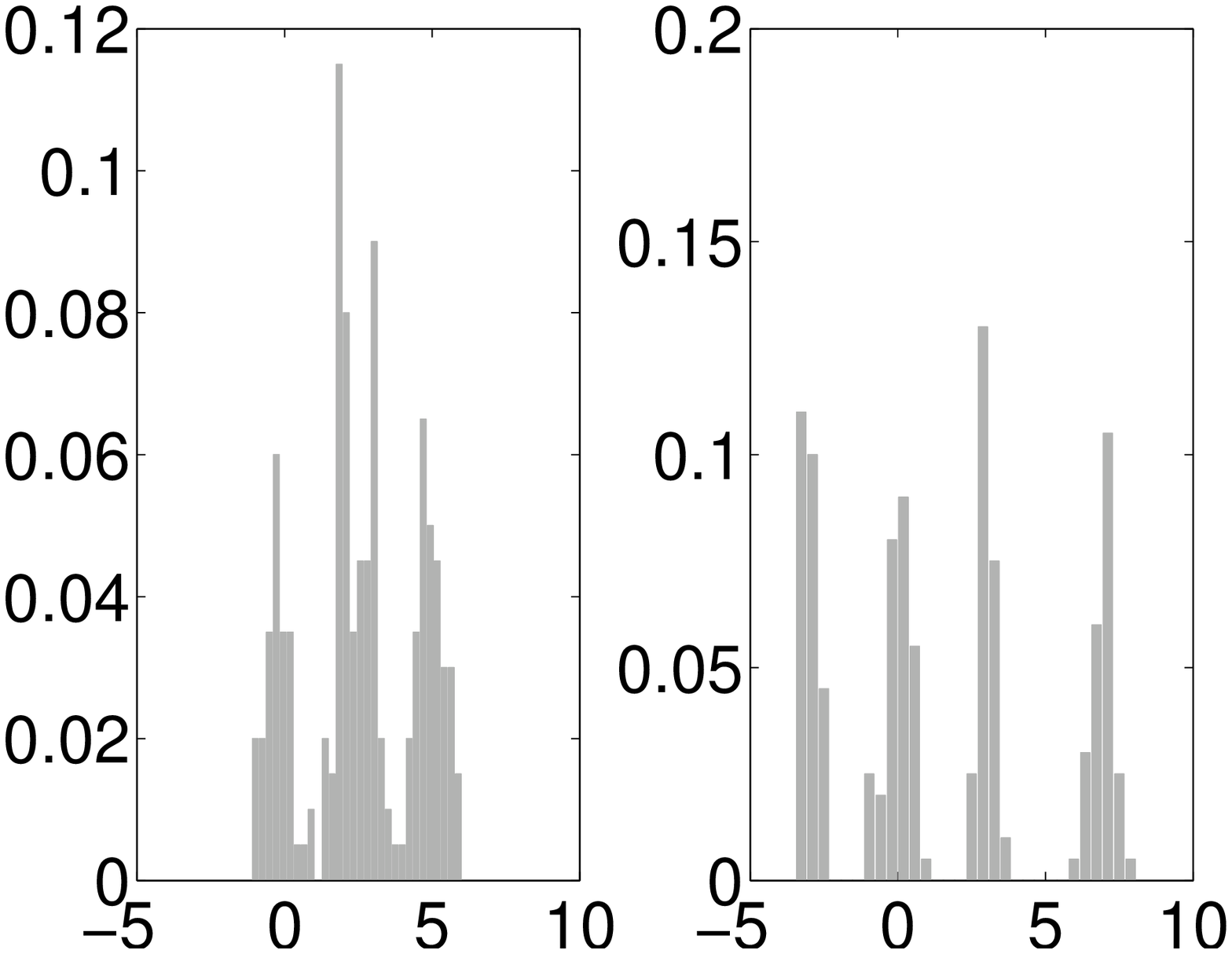}\\
\includegraphics[height=4cm,width=7cm, angle=0, clip=false]{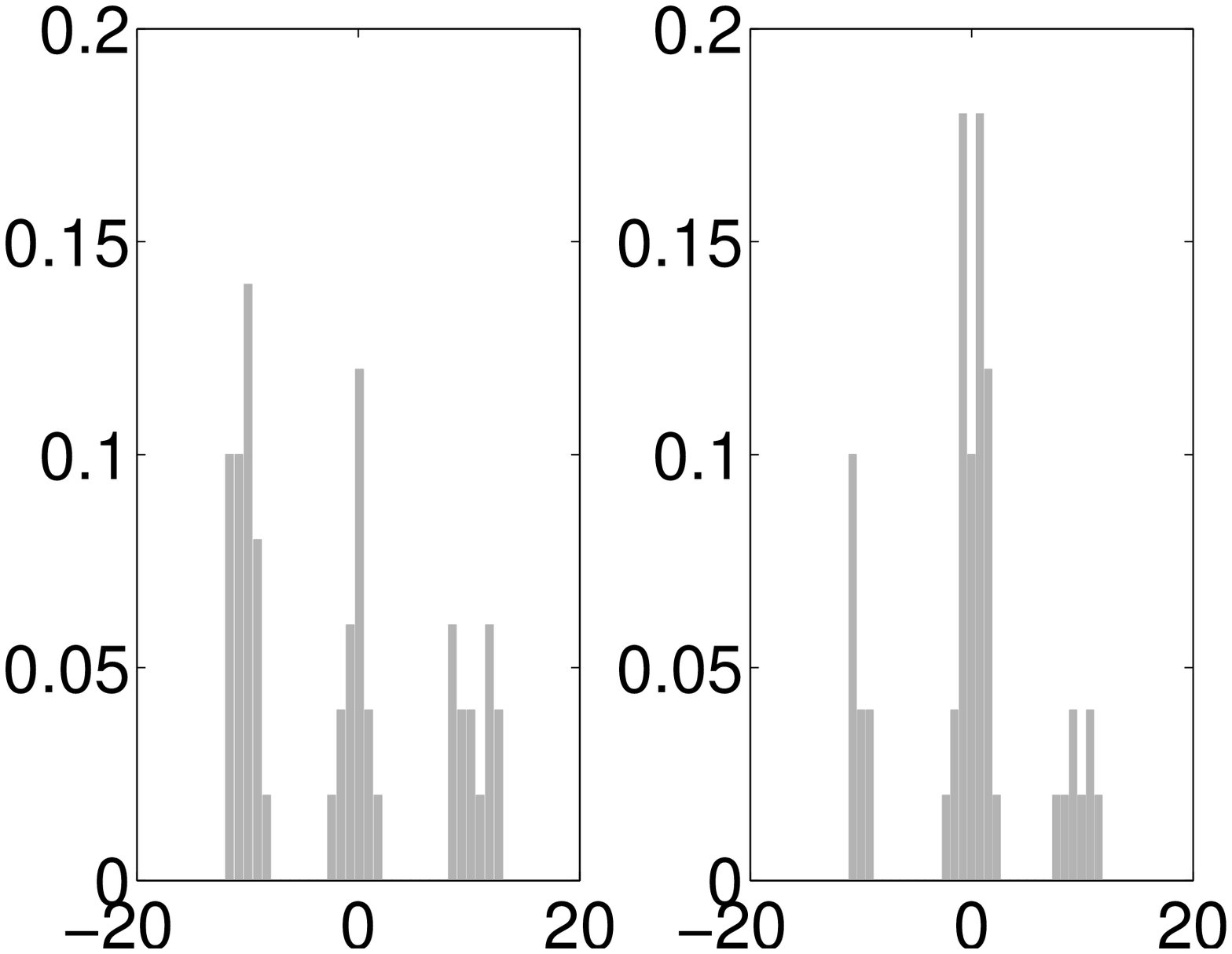}
\end{tabular}
\caption{Sample of data for the first (left) and second (right) component generated under different model assumptions: Mix1 (first row), Mix2 (second row), and Mix3 (third row).}\label{data}
\end{centering}
\end{figure}
Then we estimate the  $\beta_{1}\!-\!\hbox{DDP}(\XXi,H_0)$ mixture model on the different set of data. In the inference exercise, we choose a
non--informative prior specification for the mean and precision parameters of the base measure and set $s^{2}=0.1$, $\lambda=0.5$ (see for example \cite{walker2007}). For the concentration parameters $(\alpha_1,\alpha_2)$ of the stick-breaking components, we follow \cite{walker2011} and consider two alternative specifications of the priors: weakly informative (WI) prior and strongly informative (SI) prior.
For the WI case, the hyperparameters setting is $(\zeta_{1j}=0.01,\zeta_{2j}=0.01)$, for $j=1,2$, in all the Mix1, Mix2 and Mix3 experiments. This setting corresponds to diffuse priors on the concentration parameters, with prior means $\mathbb{E}(\tilde \a_1)=\mathbb{E}(\tilde \a_2)=1$ and variances $\mathbb{V}(\tilde \a_1)=\mathbb{V}(\tilde \a_2)=100$, and to a low prior dependence level ($Cor(\TG_{1},\TG_{2})=0.6902$) between the random marginal densities. In our WI setup, a small amount of information exchange is allowed a priori, between the two marginal densities and the posterior level of information exchanged is heavily affected by the empirical evidence.
For the SI case, a large amount of information exchange is desired instead between the two sets of data. In the SI, we set $(\zeta_{11}=100,\zeta_{21}=400)$ and $(\zeta_{12}=100,\zeta_{22}=200)$ in the Mix1 and Mix3 examples and $(\zeta_{11}=10,\zeta_{21}=100)$ and $(\zeta_{12}=200,\zeta_{22}=100)$ in the Mix2 example. These settings correspond to a very concentrated prior and a high prior dependence level between the two marginal densities $Cor(\TG_{1},\TG_{2})=0.7609$ and $Cor(\TG_{1},\TG_{2})=0.9343$ respectively.

For both the WI and SI settings, the Gibbs sampler, presented in the previous section, was run for 20,000 iterations.
The raw output of the MCMC chain  for the number of clusters is given in Fig. \ref{clust}. For the estimation of the number of clusters, a burn-in period of 10,000 samples was discarded.
At each Gibbs iteration from 10,000 onwards, a sample $(Y_{1,n+1},Y_{2,n+1})$ from the predictive was taken. The solid lines in Figure \ref{pred} show the estimated predictive
distributions using 10,000 samples from the Gibbs and the original sets of data.

\begin{figure}[p]
\begin{centering}
\begin{tabular}{cc}
\textbf{Weakly Informative Prior (WI)}&\textbf{Strongly Informative Prior (SI)}\\
\includegraphics[height=2.5cm,width=5.5cm, angle=0, clip=false]{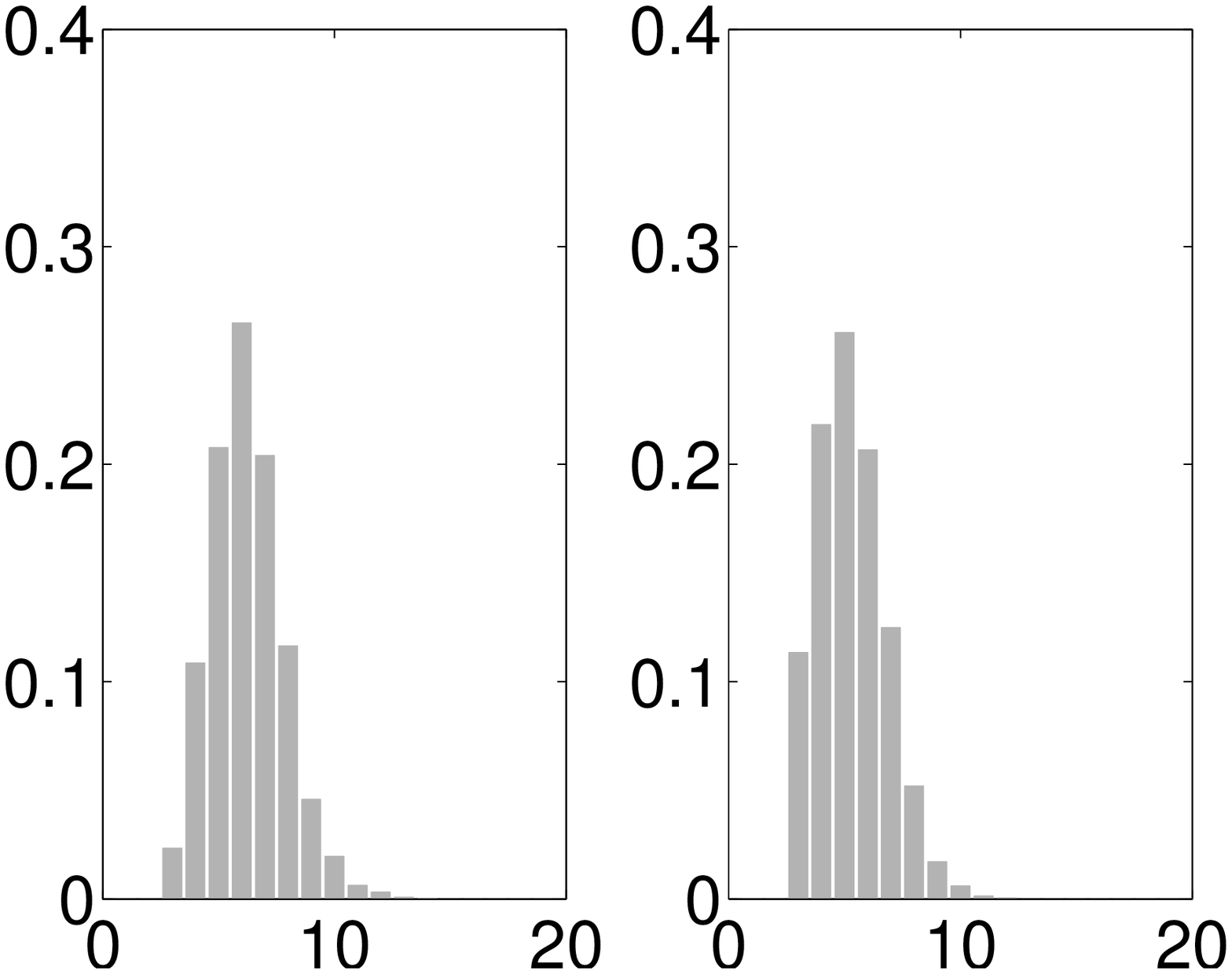}&\includegraphics[height=2.5cm,width=5.5cm, angle=0, clip=false]{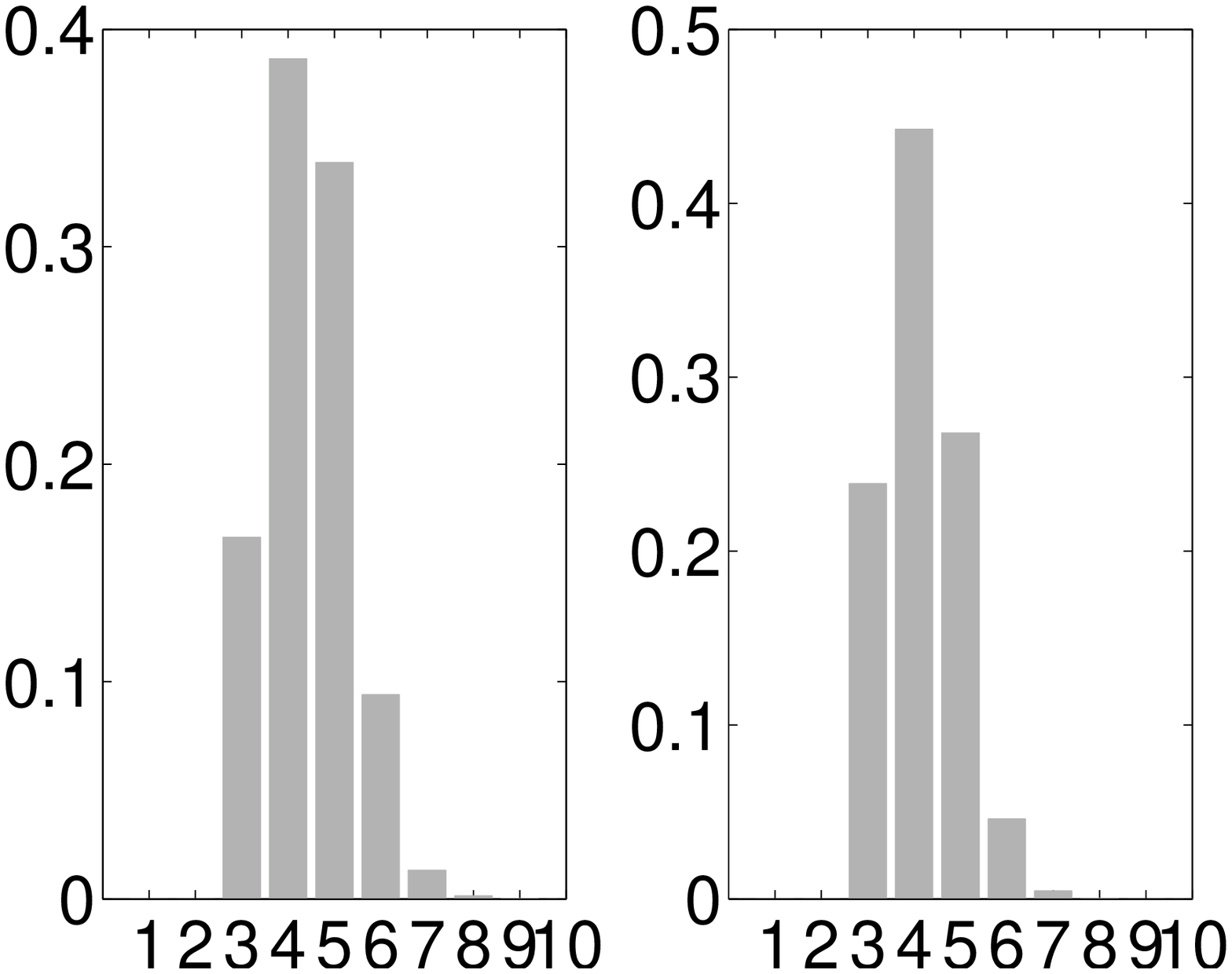}\\
\includegraphics[height=2.5cm,width=5.5cm, angle=0, clip=false]{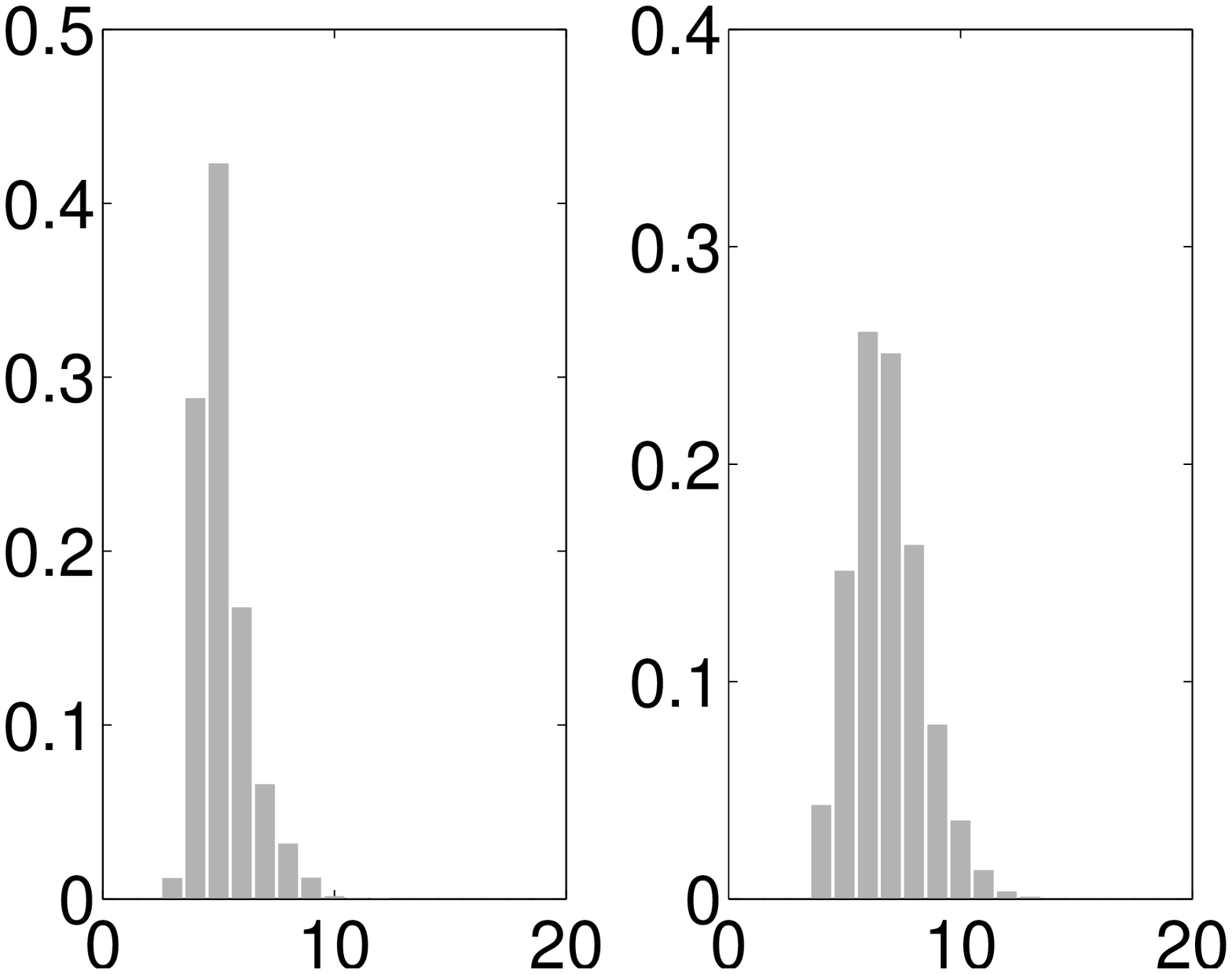}&\includegraphics[height=2.5cm,width=5.5cm, angle=0, clip=false]{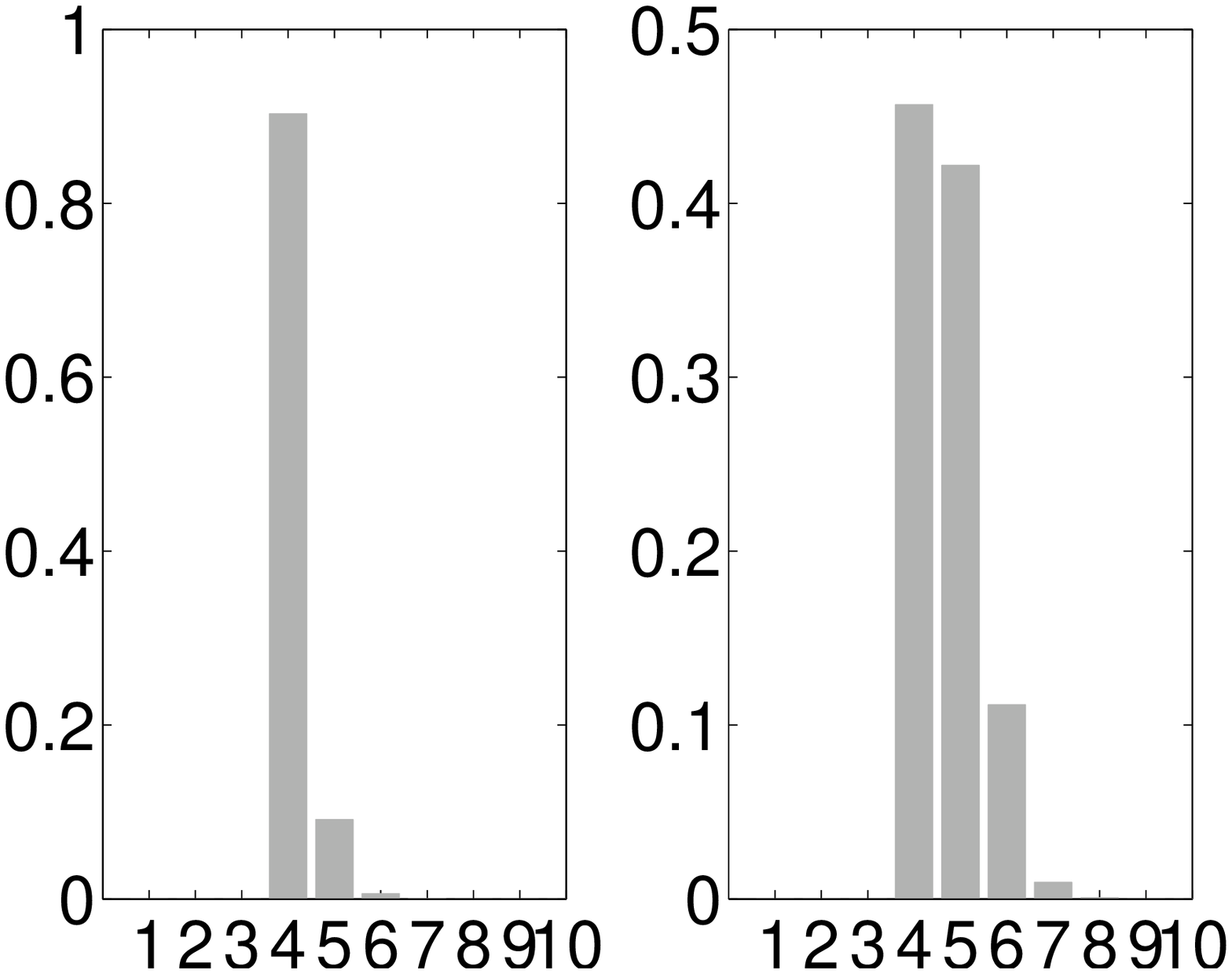}\\
\includegraphics[height=2.5cm,width=5.5cm, angle=0, clip=false]{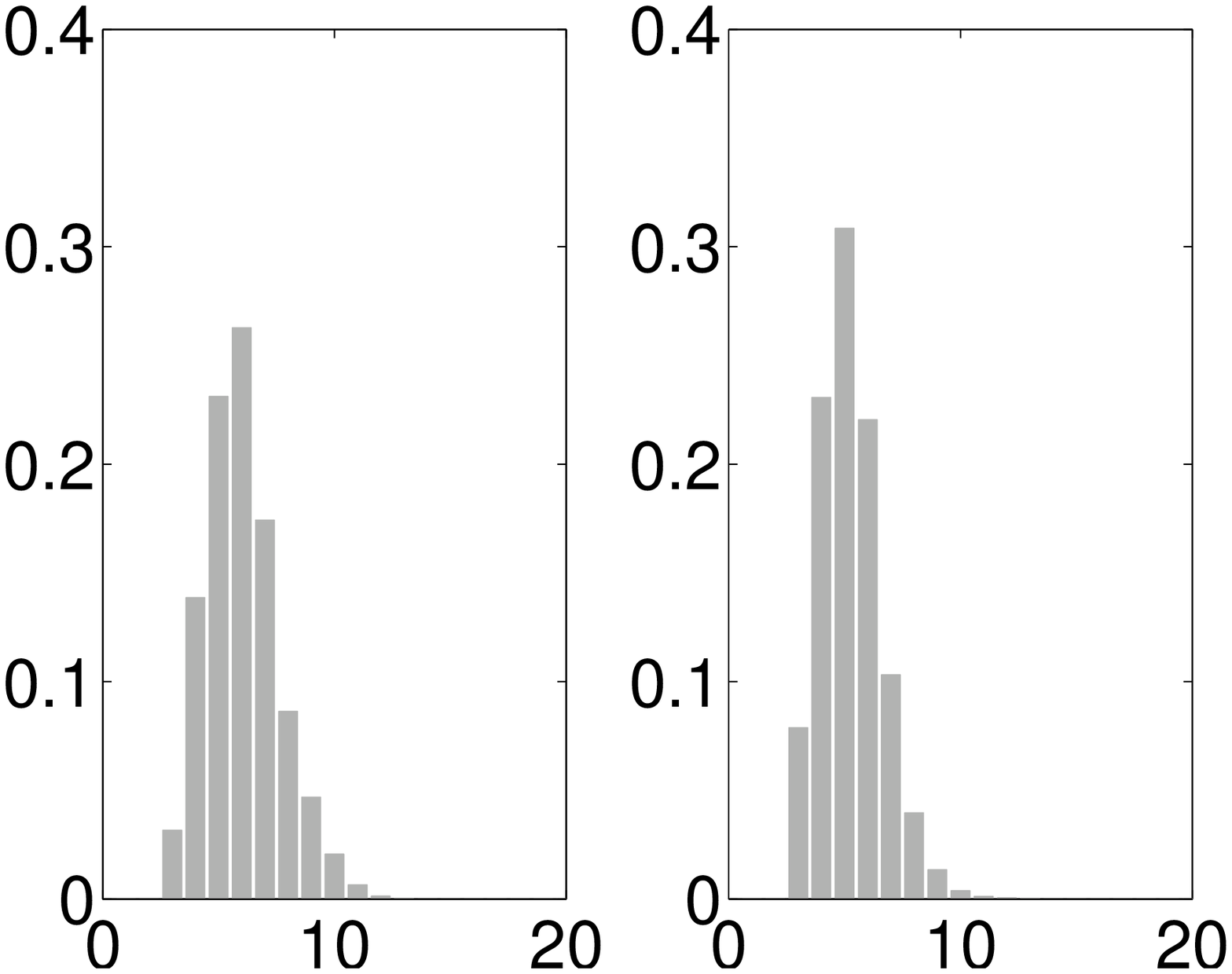}&\includegraphics[height=2.5cm,width=5.5cm, angle=0, clip=false]{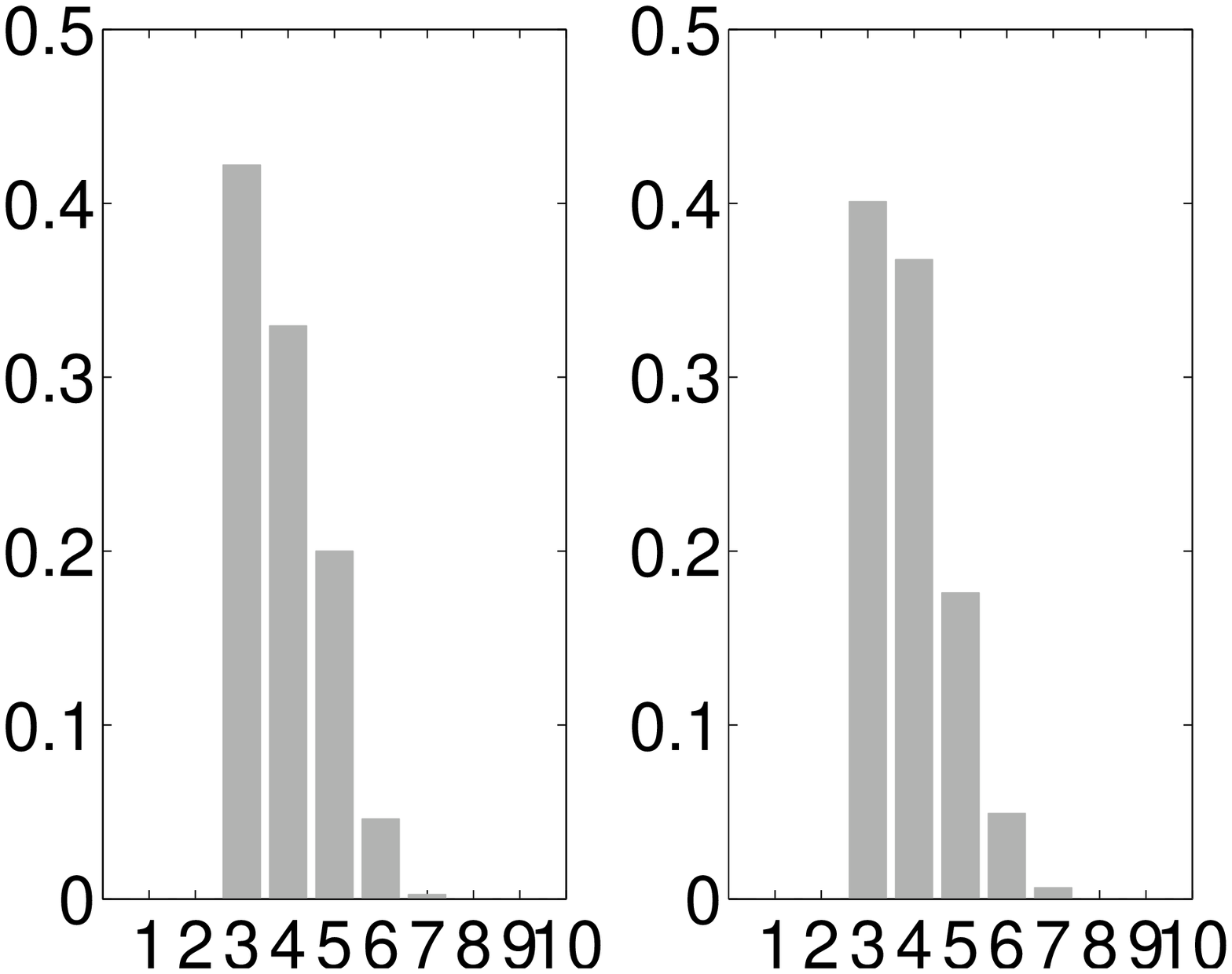}\\
\end{tabular}
\caption{Number of clusters for different prior settings (panels WI and SI), for the first (left column) and
second component (right column) and for the different models: Mix1 (first row), Mix2 (second row)
and Mix3 (third row).}\label{clust}
\end{centering}
\end{figure}

\begin{figure}[p]
\begin{centering}
\begin{tabular}{cc}
\textbf{Weakly Informative Prior (WI)}&\textbf{Strongly Informative Prior (SI)}\\
\includegraphics[height=2.5cm,width=5.5cm, angle=0, clip=false]{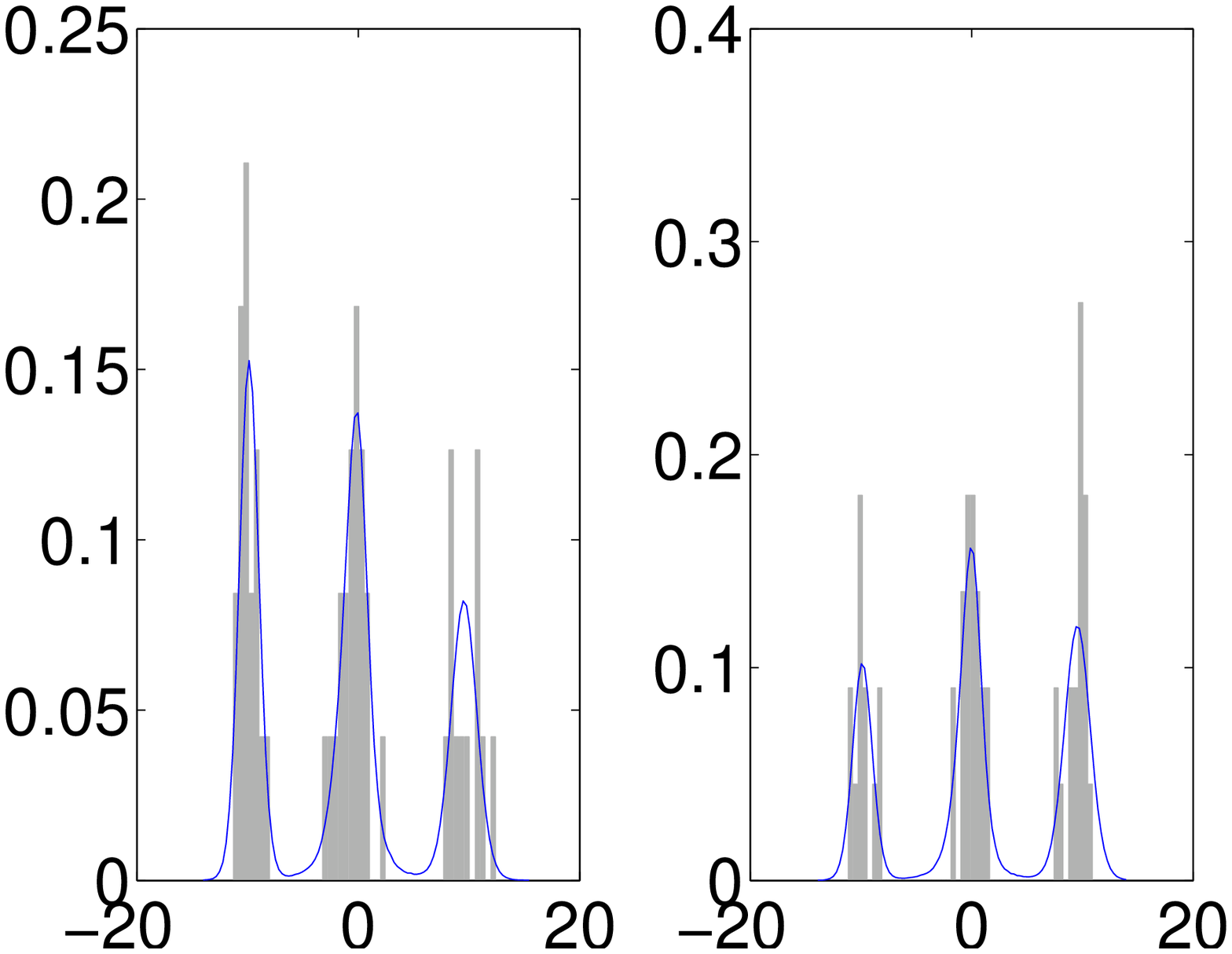}&\includegraphics[height=2.5cm,width=5.5cm, angle=0, clip=false]{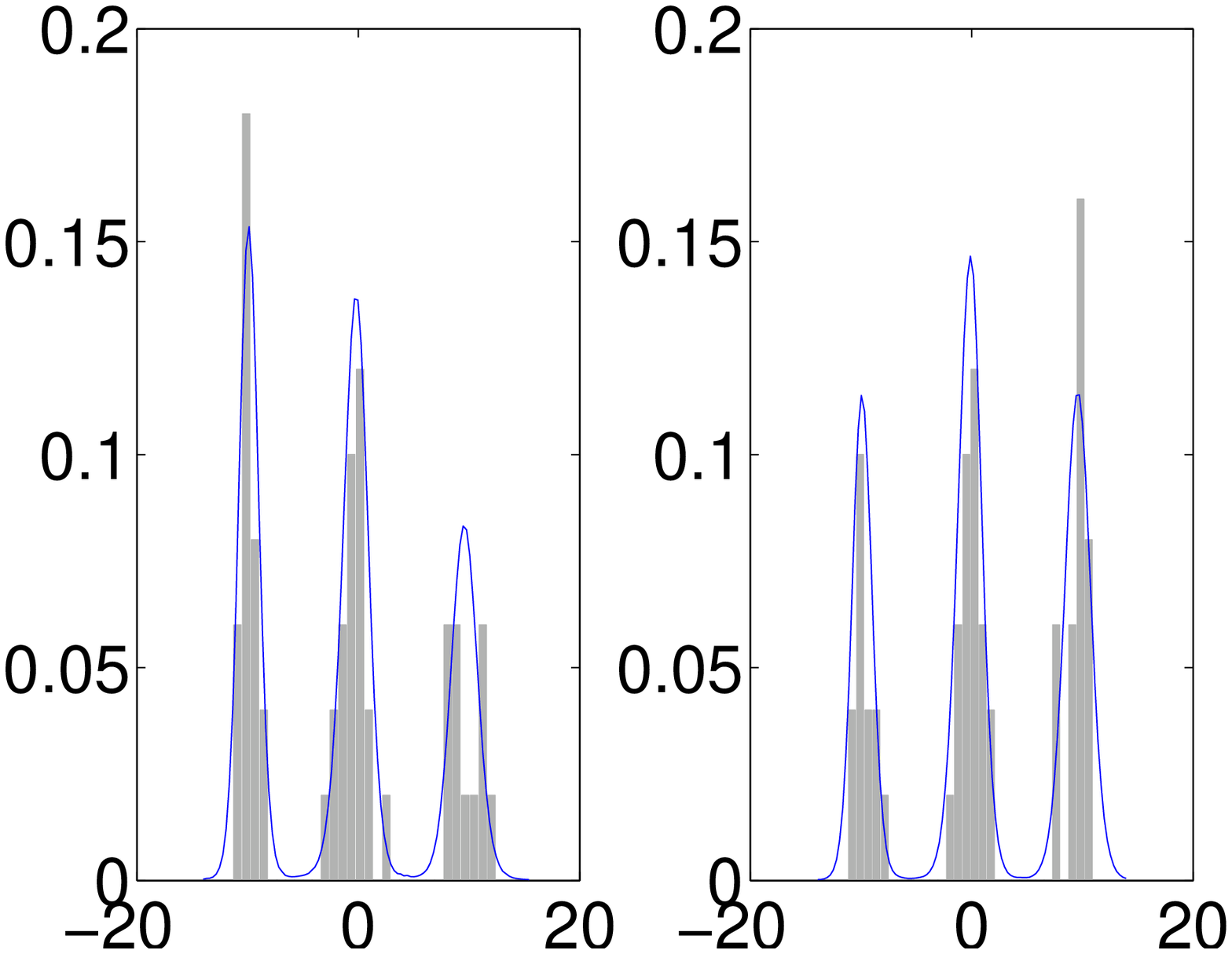}\\
\includegraphics[height=2.5cm,width=5.5cm, angle=0, clip=false]{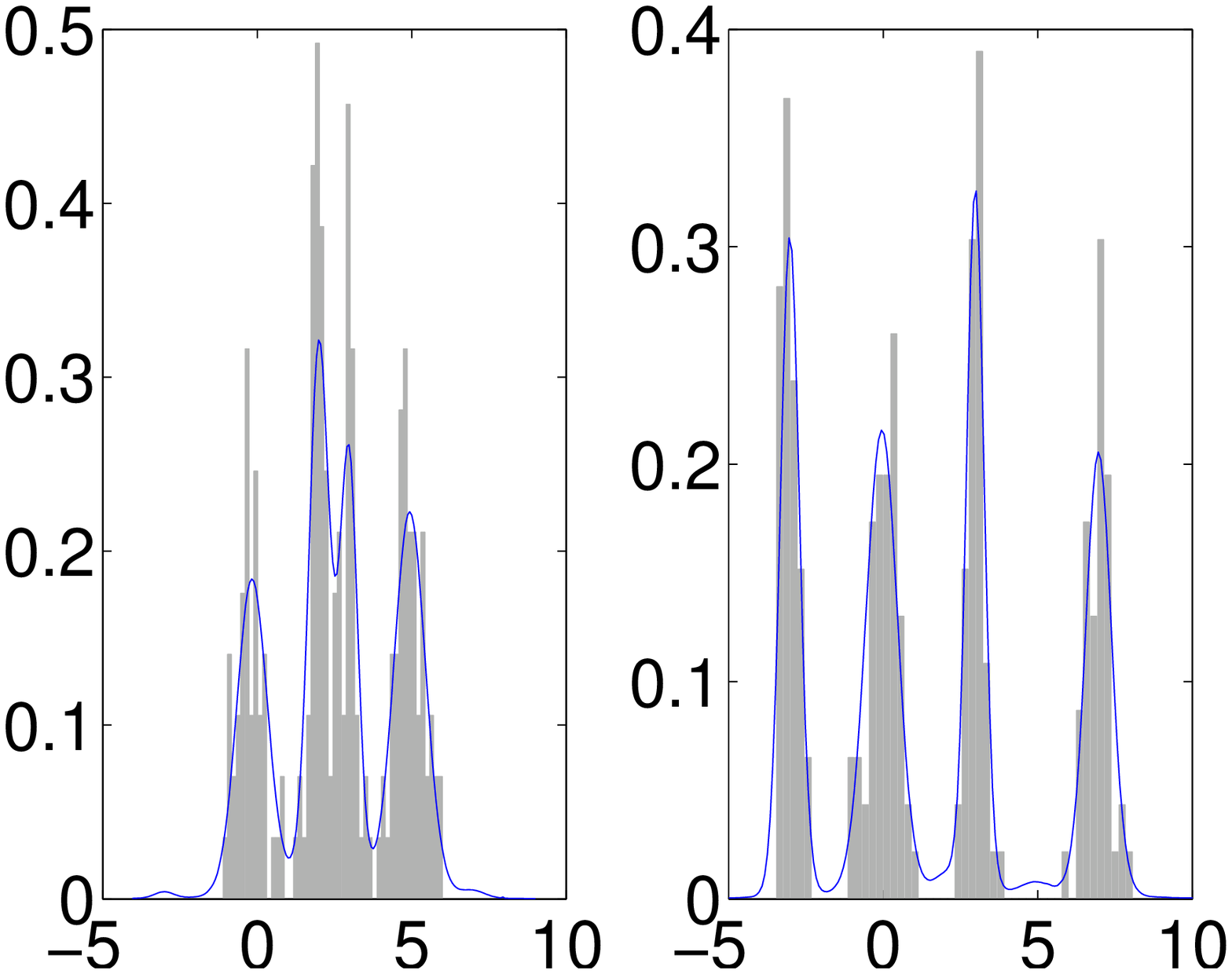}&\includegraphics[height=2.5cm,width=5.5cm, angle=0, clip=false]{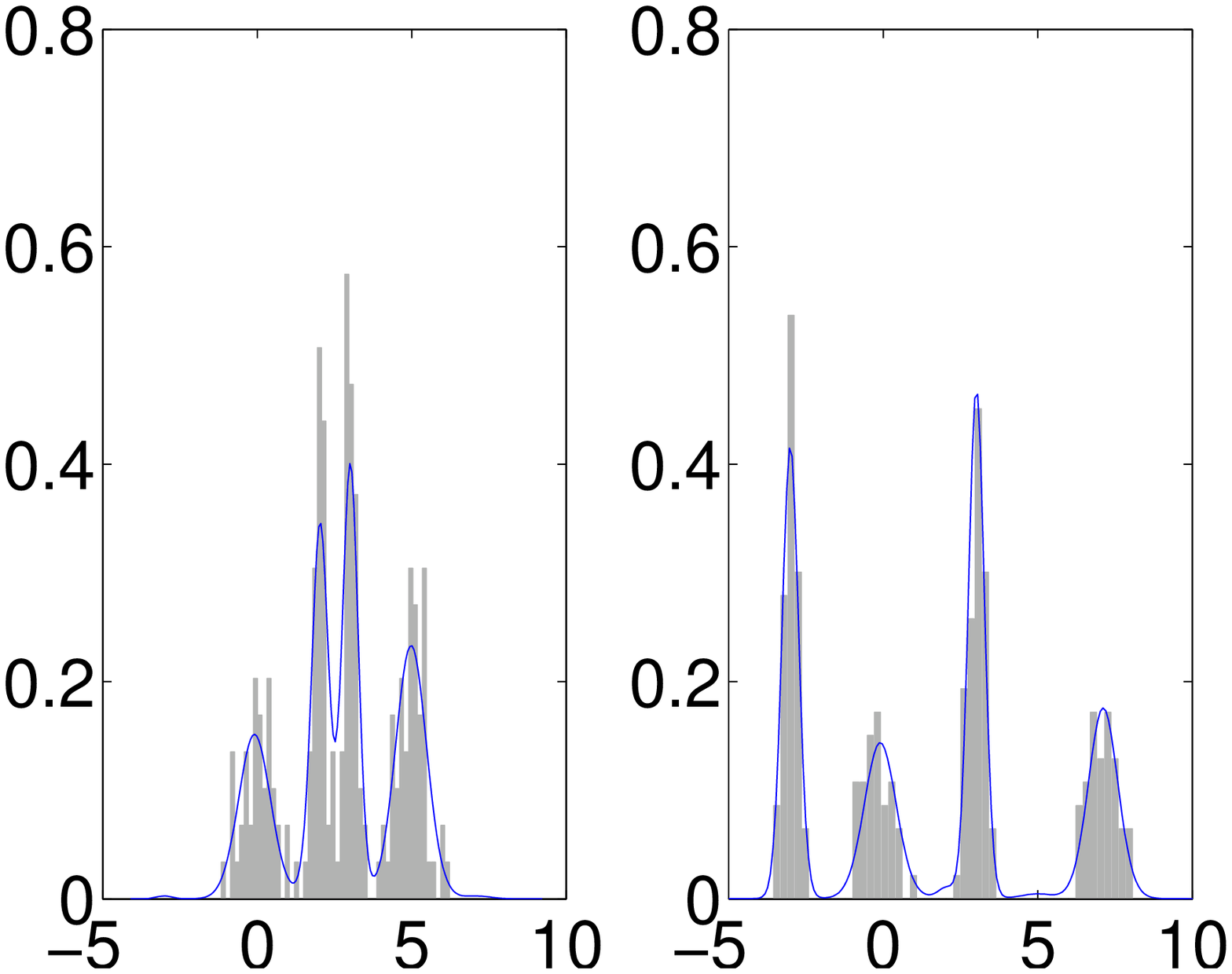}\\
\includegraphics[height=2.5cm,width=5.5cm, angle=0, clip=false]{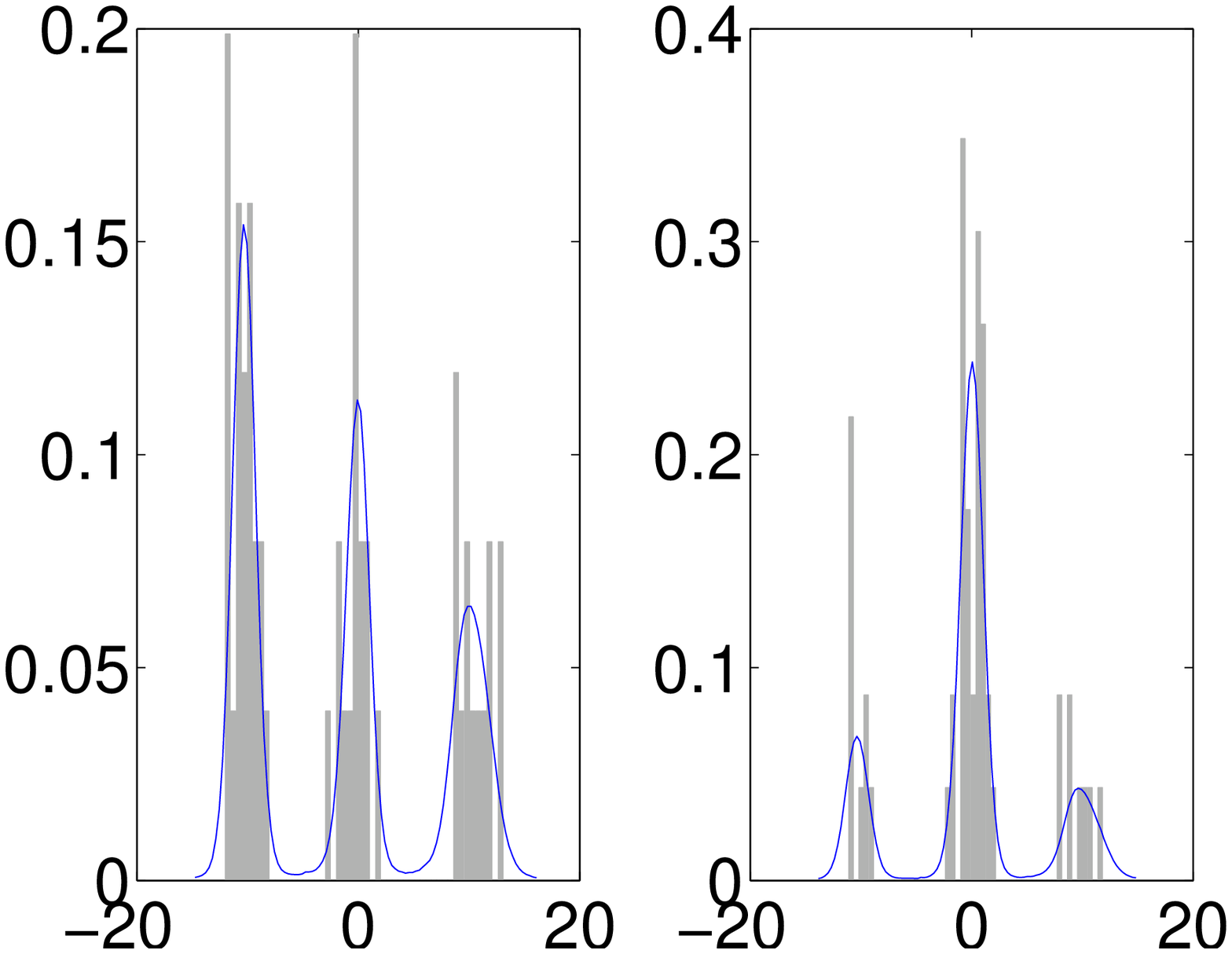}&\includegraphics[height=2.5cm,width=5.5cm, angle=0, clip=false]{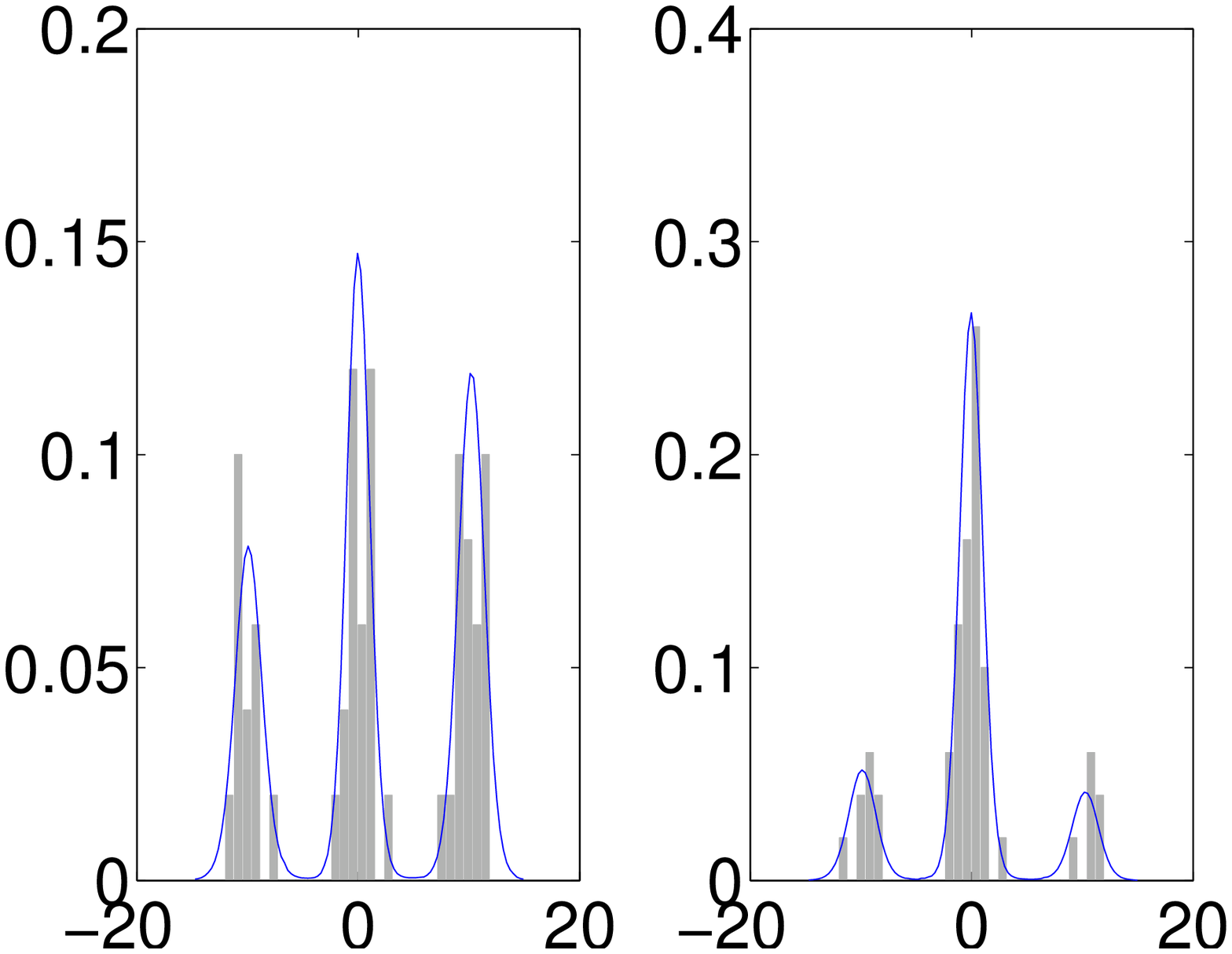}\\
\end{tabular}
\caption{Approximated predictive density (solid lines) for the two settings (panels WI and SI) and for the first (left) and second (right) component of the
data from models Mix1 (first row), Mix2 (second row), and Mix3 (third row).}\label{pred}
\end{centering}
\end{figure}

Fig. \ref{param} shows the raw output and the ergodic average of the MCMC chain for the
parameters $\alpha_1$ and $\alpha_2$ in the WI and SI prior settings.

\begin{figure}[t]
\begin{centering}
\begin{tabular}{cc}
\textbf{Weakly Informative Prior (WI)}&\textbf{Strongly Informative Prior (SI)}\\
\includegraphics[height=3cm,width=6cm, angle=0, clip=false]{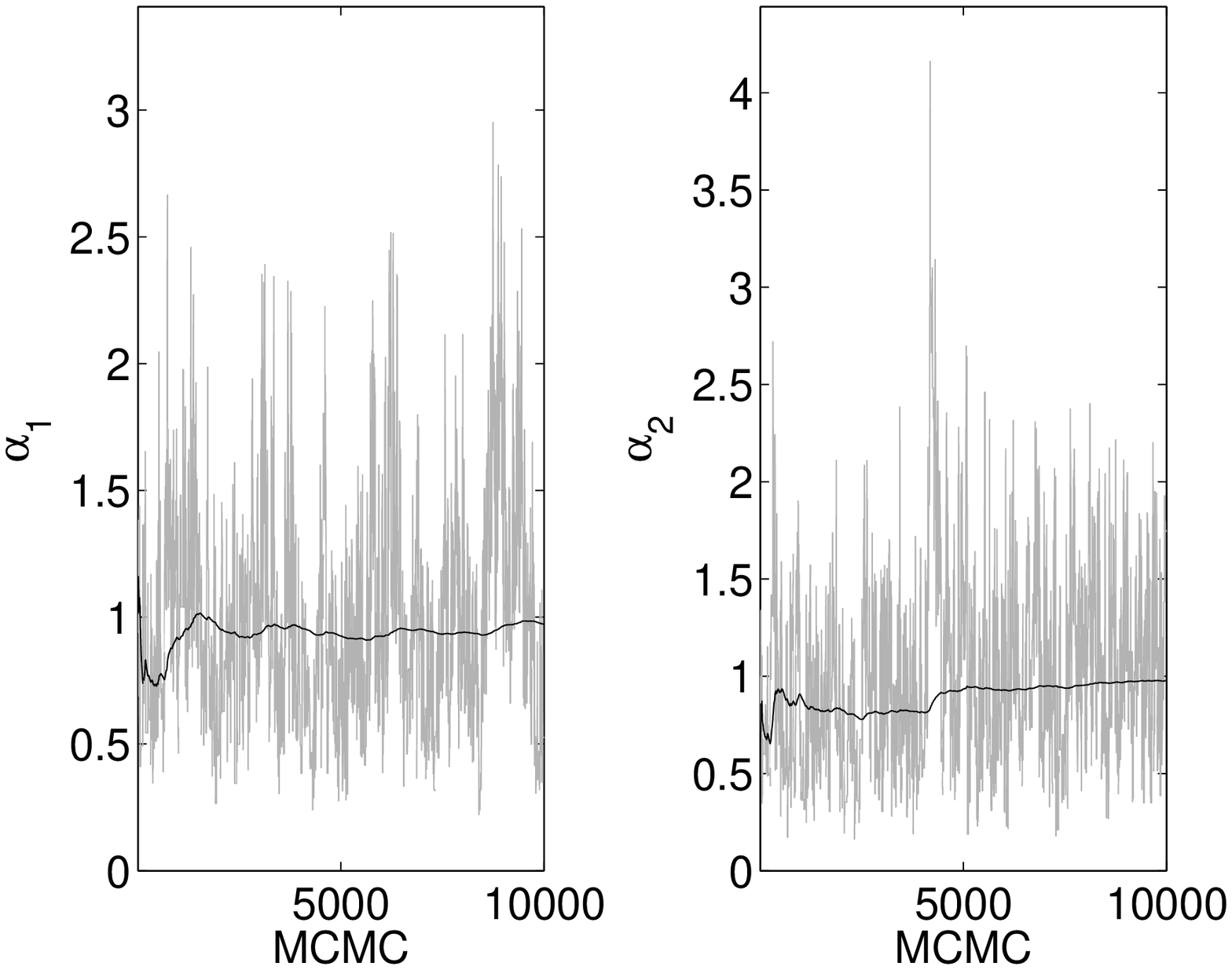}&\includegraphics[height=3cm,width=6cm, angle=0, clip=false]{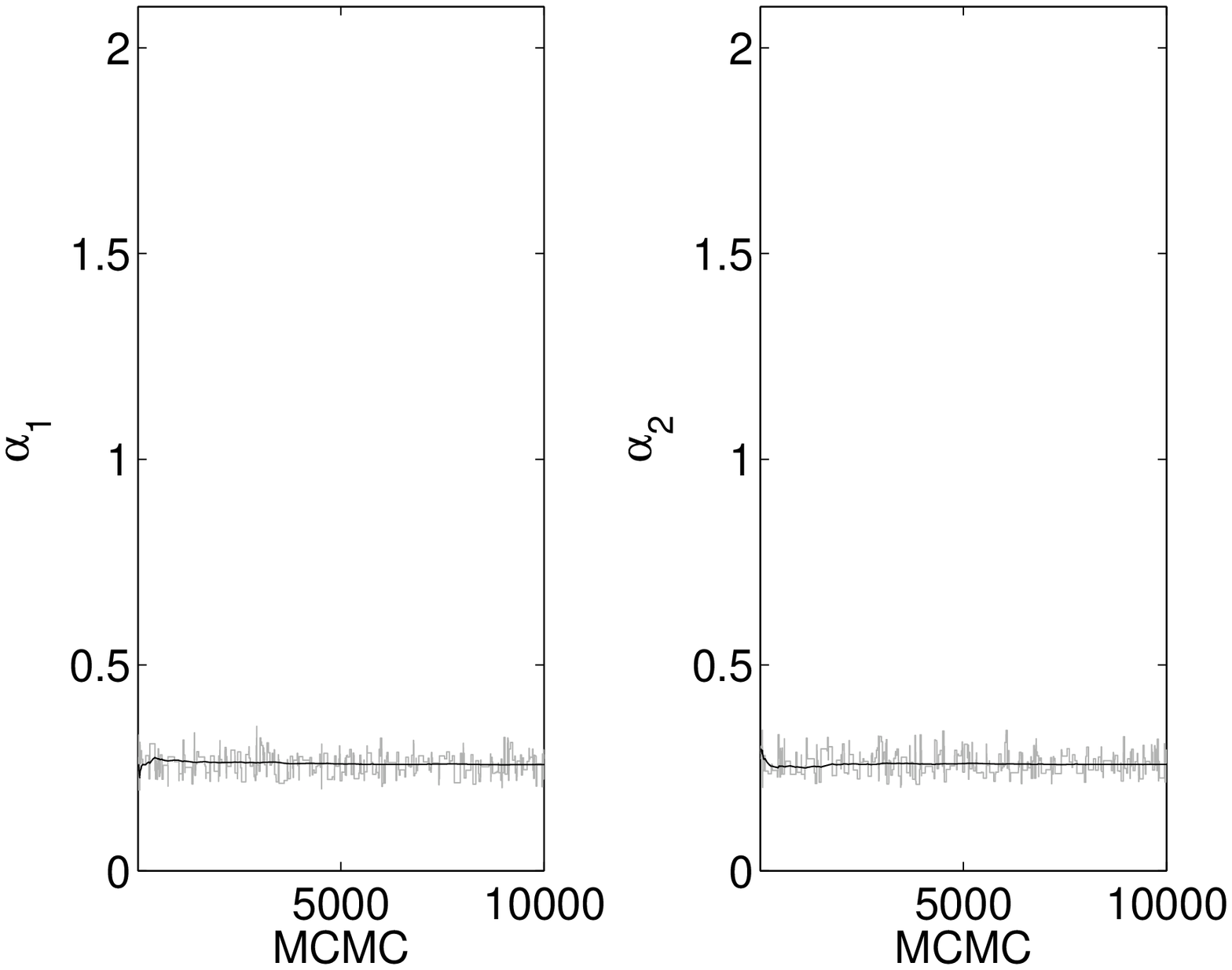}\\
$\,$&\\
\includegraphics[height=3cm,width=6cm, angle=0, clip=false]{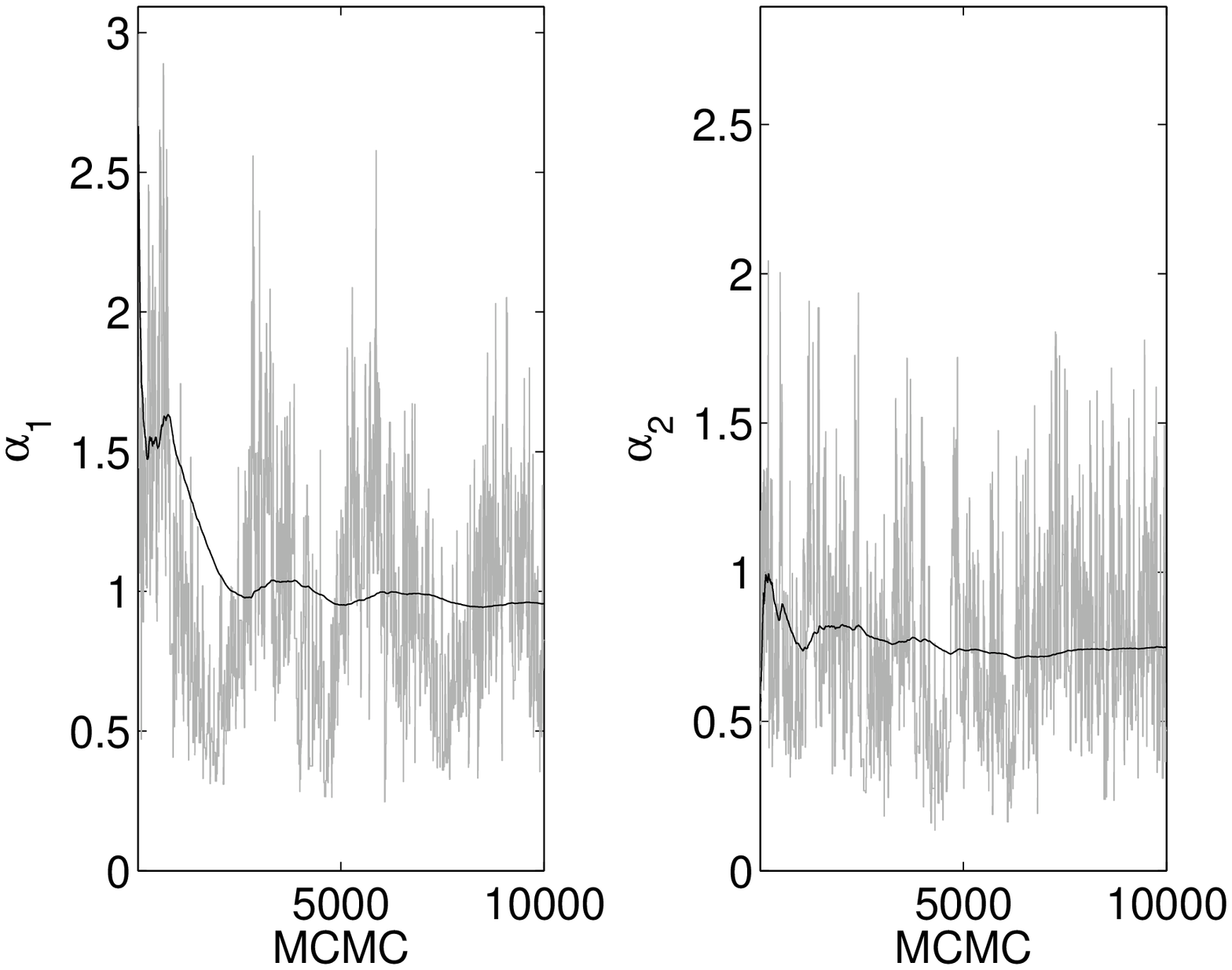}&\includegraphics[height=3cm,width=6cm, angle=0, clip=false]{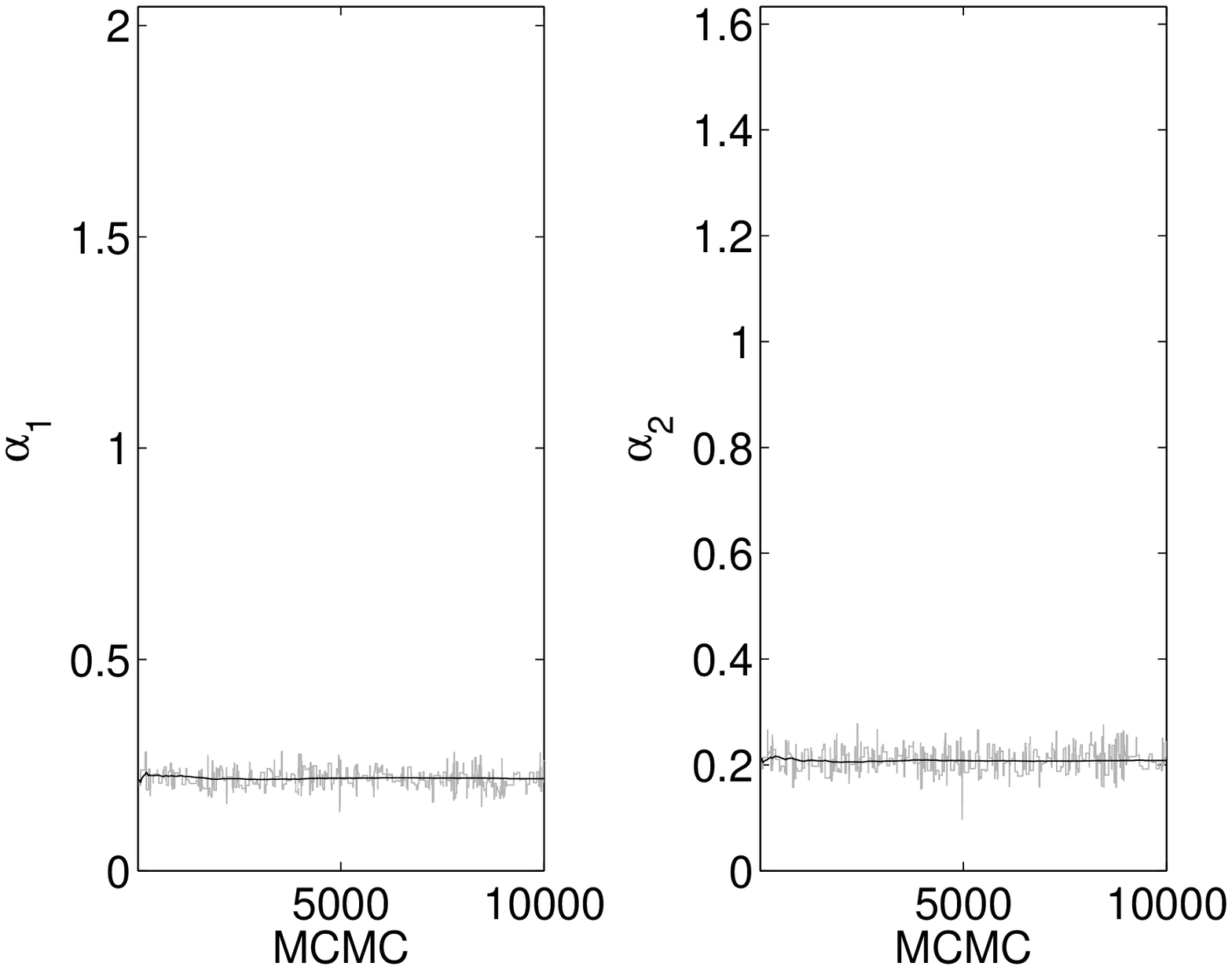}\\
$\,$&\\
\includegraphics[height=3cm,width=6cm, angle=0, clip=false]{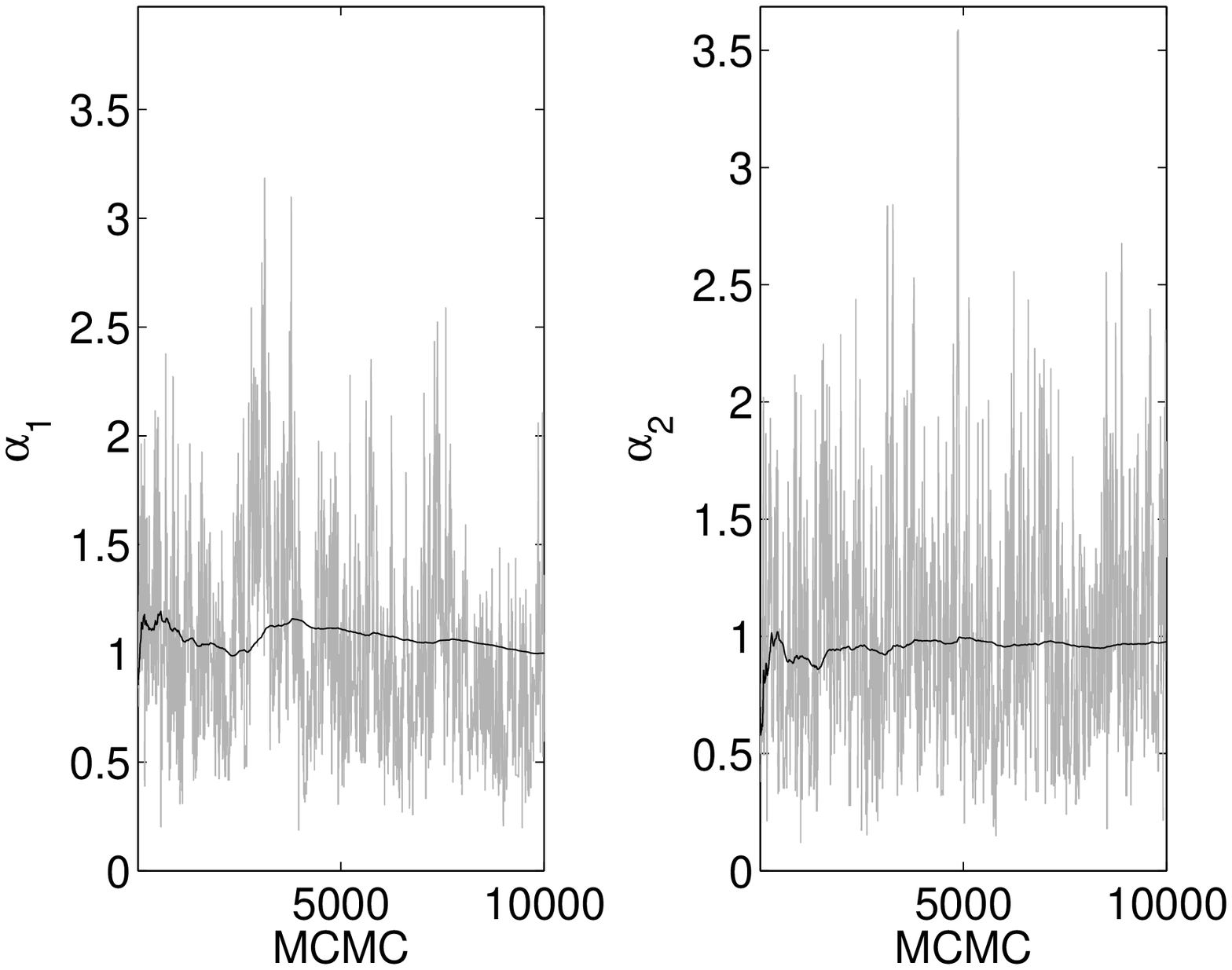}&\includegraphics[height=3cm,width=6cm, angle=0, clip=false]{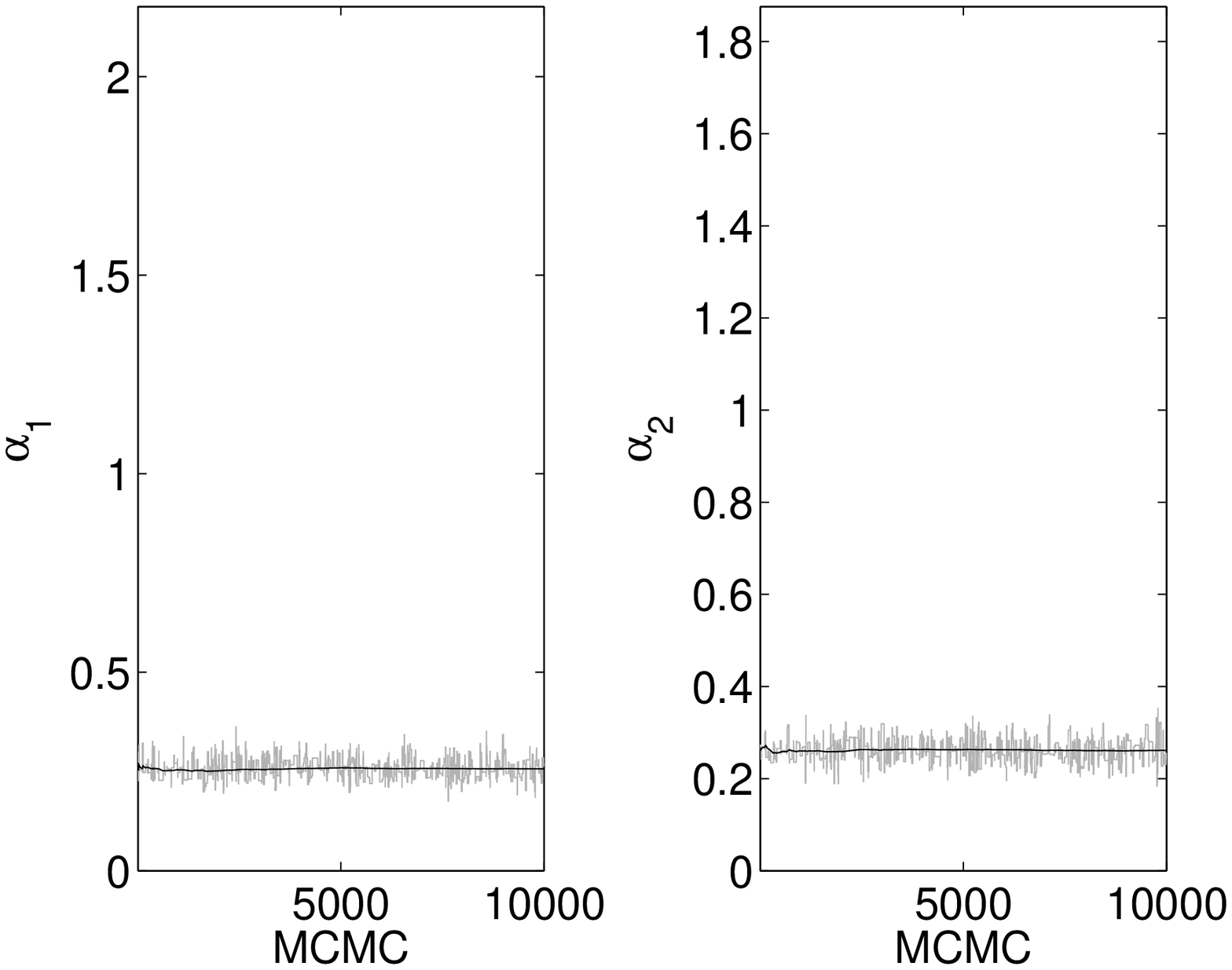}\\
\end{tabular}
\caption{Output of the M.-H. within Gibbs for the two settings (panels WI and SI) for the parameters $\alpha_{1}$ and $\alpha_{2}$ of the models Mix1 (first row), Mix2 (second row), and
Mix3 (third row).}\label{param}
\end{centering}
\end{figure}

\subsection{ $\beta_{1}\!-\!\hbox{DDP}(\XXi,H_0)$  mixtures of vector autoregressive processes}$\quad$
In parametric models for the growth rate of the industrial production (business cycle), great advances
have been made by allowing for separate parameter values in periods (called regimes) of recession and
expansion. The seminal paper of \cite{Ham89} proposes to use a dynamic mixture model with
two components for capturing clustering of observations during the recession and expansion phases
in a business cycle. This simple model has been successfully extended in many directions. In particular, the
estimation of the number of regimes is an important issue studied in many papers (e.g., \cite{KimMur02},
\cite{KimPig00} and \cite{Kro00}). The estimation of the number of regimes is  still an open issue in the
analysis of the business cycle. Moreover, specifically it is interesting to verify
whether the strong contraction in 2009 calls for the use of a higher number of regime than three or four in business cycle models.

The above cited papers consider parametric models with a regime-switching mechanism and use some model selection
criteria to estimate the number of regimes. Conversely, in this paper, we propose a non--parametric approach to the
joint estimation of the number of regimes in multiple time series. We assume our Dirichlet mixture process
$\beta_{1}\!-\!\hbox{DDP}(\XXi,H_0)$ as a prior for the parameters of a vector autoregressive model (VAR) for time series
data. We consider two well studied cycles of the international economic system: the United States (US) and the
European Union (EU) cycles. Even if the features of the regimes (or clusters) in the US and the EU
growth rates are different, one could expect that the regimes in the two cycles also exhibit some dependence.
For this reason, we apply a dependent multivariate Dirichlet process to account for the similarity between the clustering
of the two series. In this sense, our model extends the existing literature on the use of univariate Dirichlet process
prior for time series models. In this literature, the same clustering process is usually assumed for all
the parameters of a multivariate model.

We consider seasonally and working day adjusted industrial production indexes (IPI), at a monthly frequency
from the time of April 1971 to January 2011, for the US and the EU, $\{X_{1t}\}_{t=1}^{T}$ and $\{X_{2t}\}_{t=1}^{T}$ respectively (see first row
in Fig. \ref{dataindp}). We take the quarterly growth rate: $Y_{it}=\log X_{it}-\log X_{it-4}$ (second row in Fig. \ref{dataindp}).
The histograms of these time
series (see histograms Fig. \ref{dataindppost}) exhibit many modes that are the results of different regimes in the series.
We consider the
following specification for the VAR model
\begin{equation*}
{\,\,Y_{1t}\, \choose \,\,Y_{2t}\,}%
={\,\, \tilde \mu_{1t}\, \choose \,\, \tilde \mu_{2t}\,}
+
{ \,\,Z_{t}'  \quad O_{2p}'\,\choose \,\,O_{2p}'  \quad Z_{t}'\,}
{\,\, \Upsilon_{1} \,\choose \,\, \Upsilon_{2}\,}
+
{\,\varepsilon_{1t}\,\choose \,\varepsilon_{2t}\,}
\end{equation*}
for $t=1,\ldots,T$, where $O_{2p}=(0,\ldots,0)'\in\mathbb{R}^{2p}$, $\Upsilon_{i}=(\upsilon_{1,1,i},\ldots,\upsilon_{1,2p,i})'$ and $Z_{t}=(Y_{1t-1},\ldots,Y_{1t-p}$, $Y_{2t-1},\ldots,Y_{2t-p})'$ and $\varepsilon_{it}\sim\mathcal{N}(0,\tilde \sigma^{2}_{it})$ with $\varepsilon_{1t}$ and $\varepsilon_{2s}$
independent $\forall s,t$.

In this paper we consider four lags (i.e. $p=4$) as for example in \cite{Ham89} and \cite{Kro00}. Moreover, as most of the forecast errors are due to shifts to
the deterministic factors (see \cite{Kro00} and \cite{CleKro98}), we propose a model with shifts in the intercept and in the volatility and
assume a vector of Dirichlet processes as a prior for $(\mu_{it},\sigma_{it}^{2})$, $i=1,2$
\begin{equation}
\begin{split}
& (\tilde \mu_{1t},\tilde \sigma_{1t}^{2})|  \TG_1, \TG_2  \stackrel{i.i.d.}{\sim}   \TG_1   \\
& (\tilde \mu_{2t},\tilde \sigma_{2t}^{2})|  \TG_1,\TG_2  \stackrel{i.i.d.}{\sim}    \TG_2   \\
& ( \TG_1, \TG_2) \sim  \beta_{1}\!-\!\hbox{DDP}(\tilde \XXi,H_0)   \\
&  \tilde \XXi \sim \mathcal{G}(\zeta_{11},\zeta_{21})\mathcal{G}(\zeta_{12},\zeta_{22}) \\
\end{split}
\end{equation}
where the base measure $H_{0}$ is a product of normal $\mathcal{N}(0,s^{-2}I_{2})$ and inverse gamma $\mathcal{I}\mathcal{G}(\lambda,\lambda)$.

Following the standard practice in Bayesian VAR modelling, for the parameters $\Upsilon_{1}$ and $\Upsilon_{2}$ we consider improper uniform prior on $\RE^{4p}$ and obtain a multivariate normal as full conditional posterior distribution to be used in the Gibbs sampler.
\begin{figure}[p]
\begin{centering}
\begin{tabular}{cc}
\includegraphics[height=4cm,width=5.5cm, angle=0, clip=false]{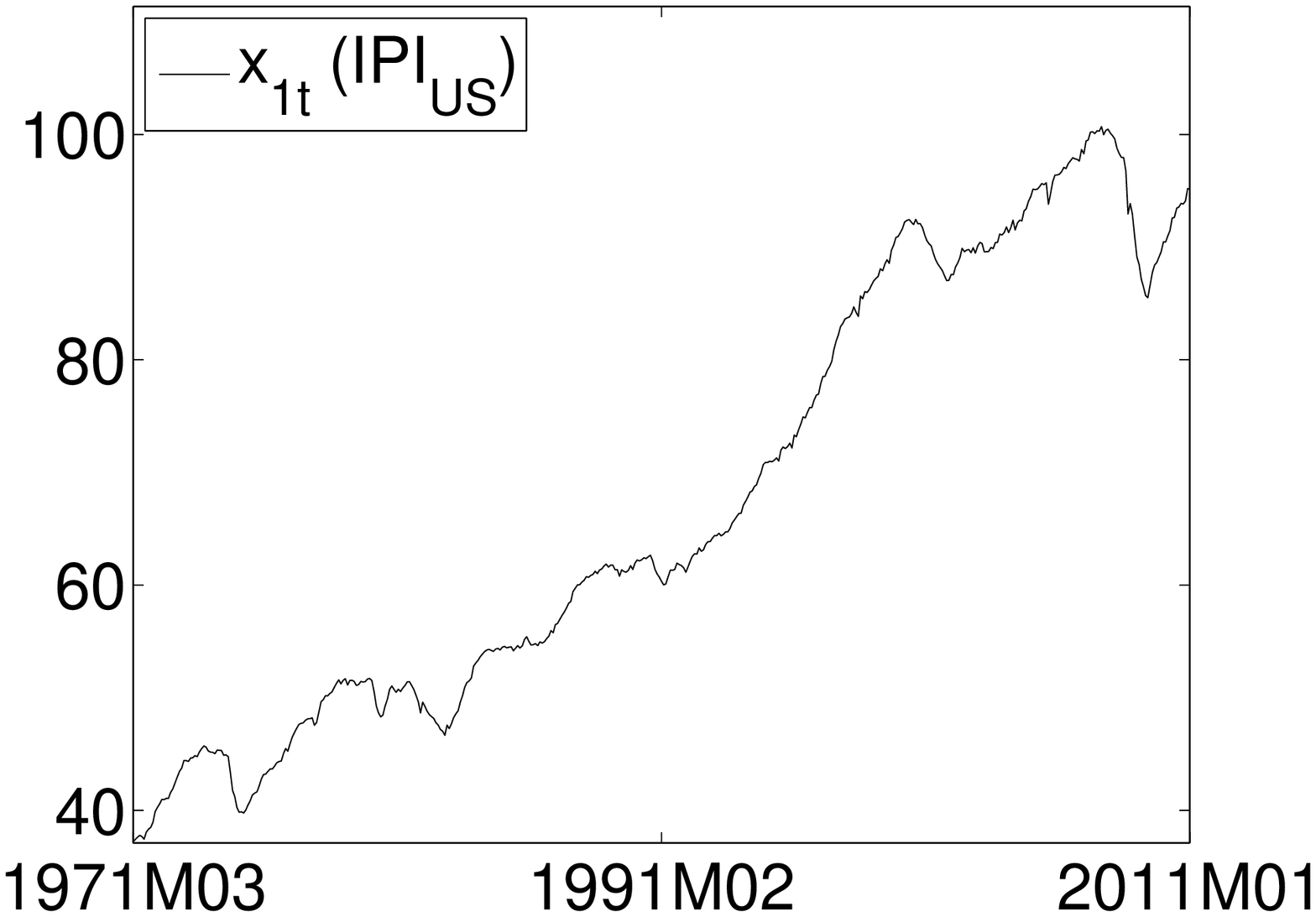}&\includegraphics[height=4cm,width=5.5cm, angle=0, clip=false]{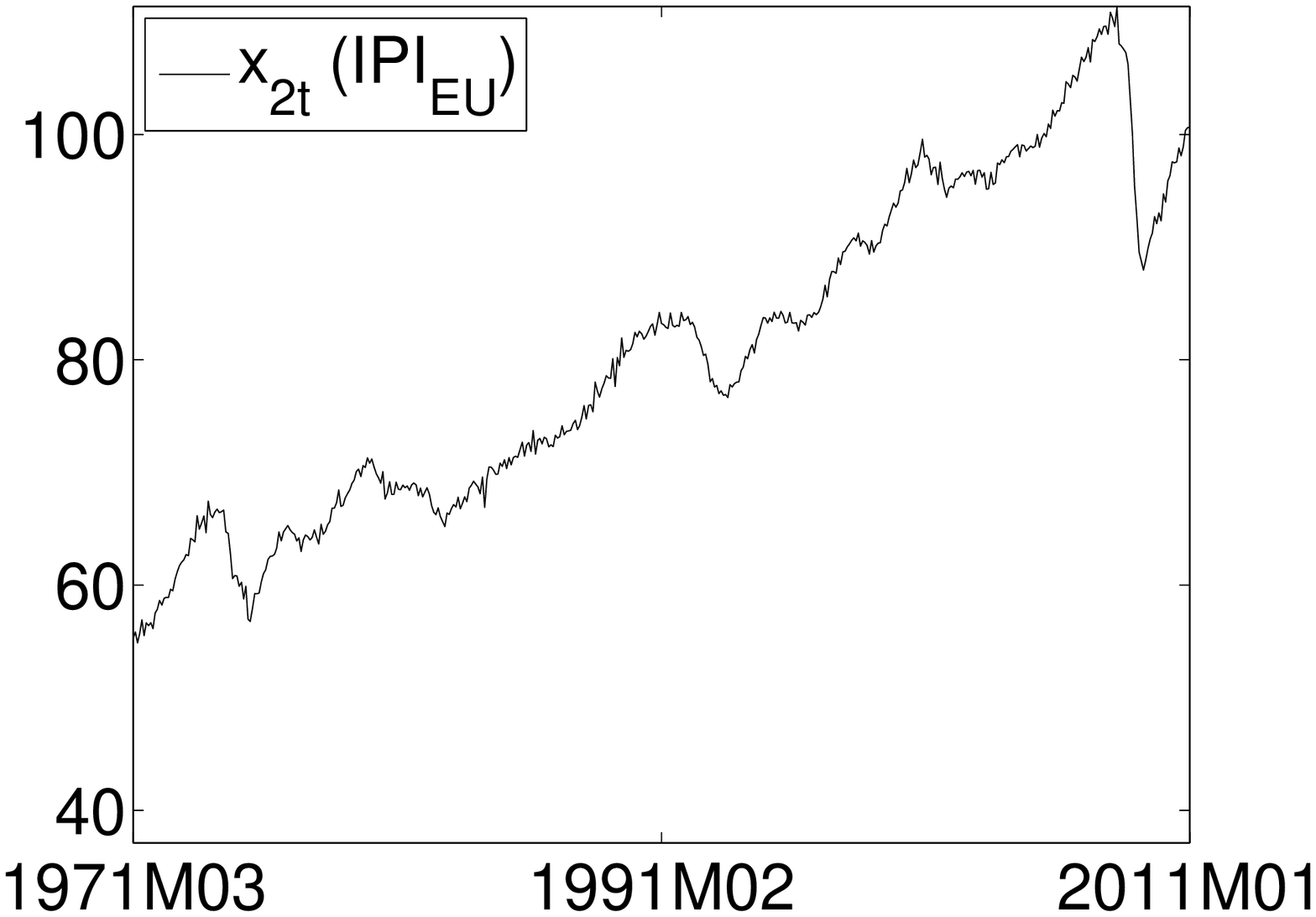}\\
\includegraphics[height=4cm,width=5.5cm, angle=0, clip=false]{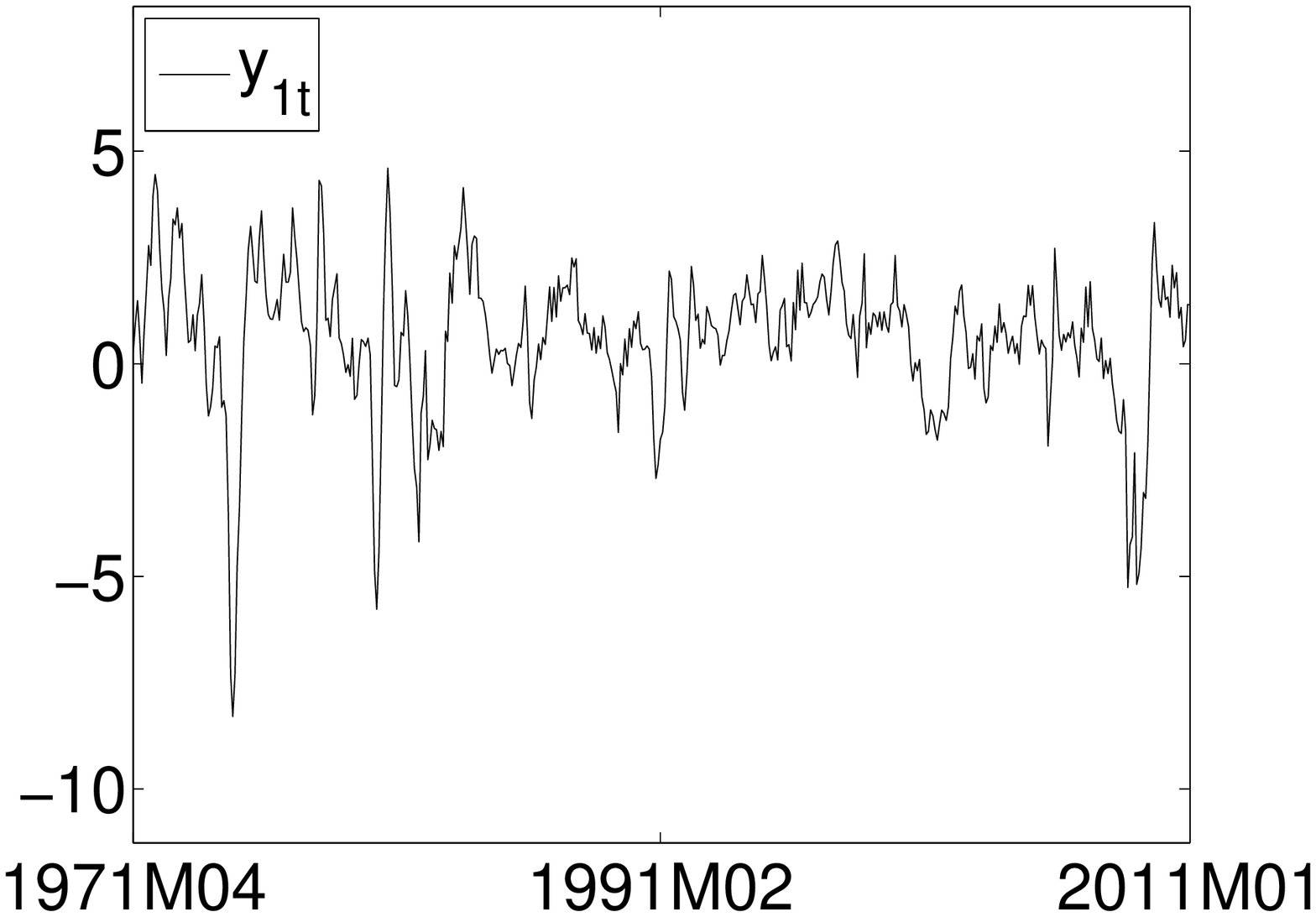}&\includegraphics[height=4cm,width=5.5cm, angle=0, clip=false]{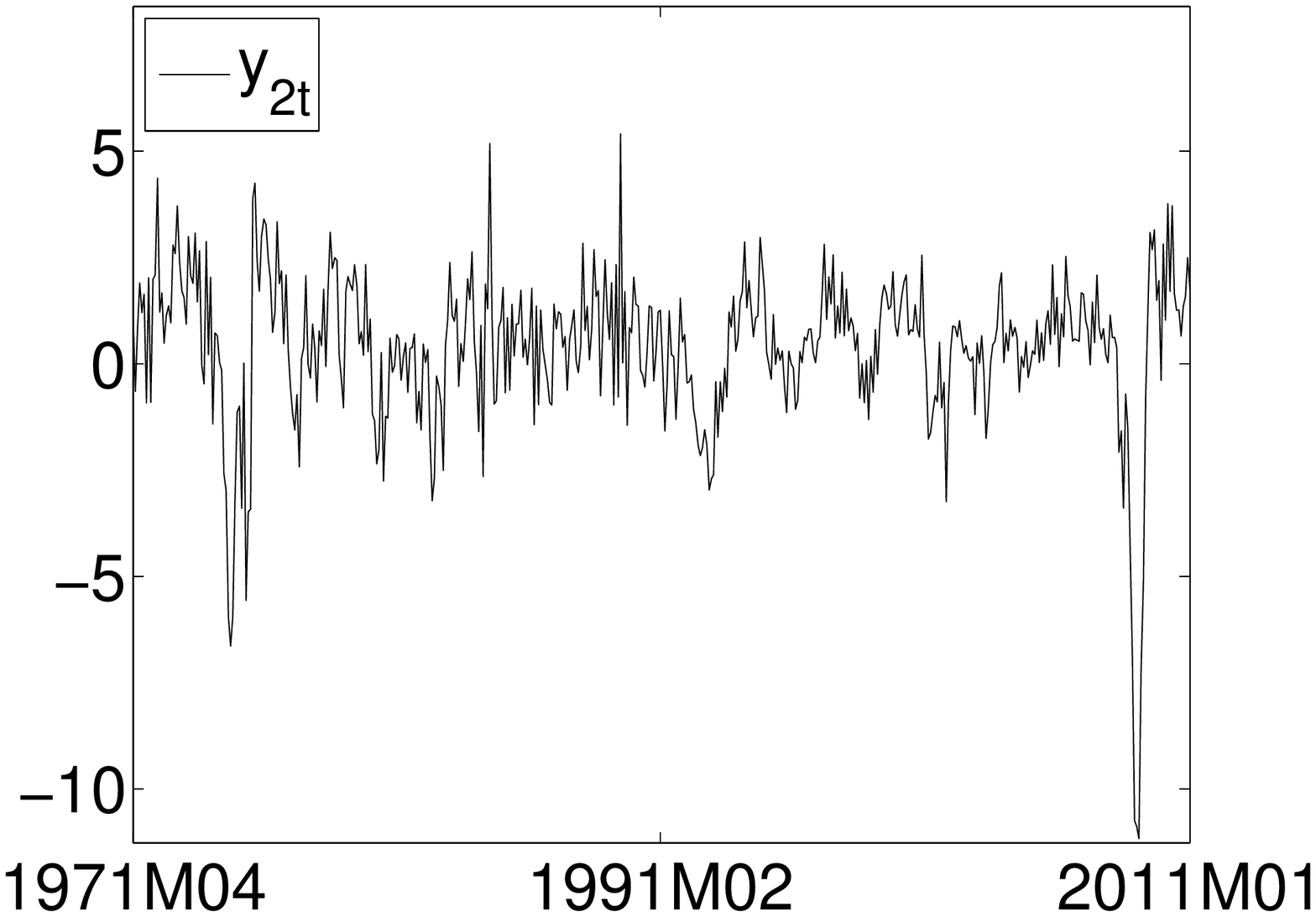}\\
\end{tabular}
\caption{First row: Industrial Production Index (IPI) for the US ($x_{1t}$) and the EU ($x_{2t}$) at monthly frequency for the period: March 1971 to
January 2011. Second row: logarithmic quarterly changes in the US IPI ($y_{1t}$) and the EU IPI ($y_{2t}$) variables.}\label{dataindp}
\end{centering}
\end{figure}
\begin{figure}[p]
\begin{centering}
\begin{tabular}{cc}
\includegraphics[height=4cm,width=5.5cm, angle=0, clip=false]{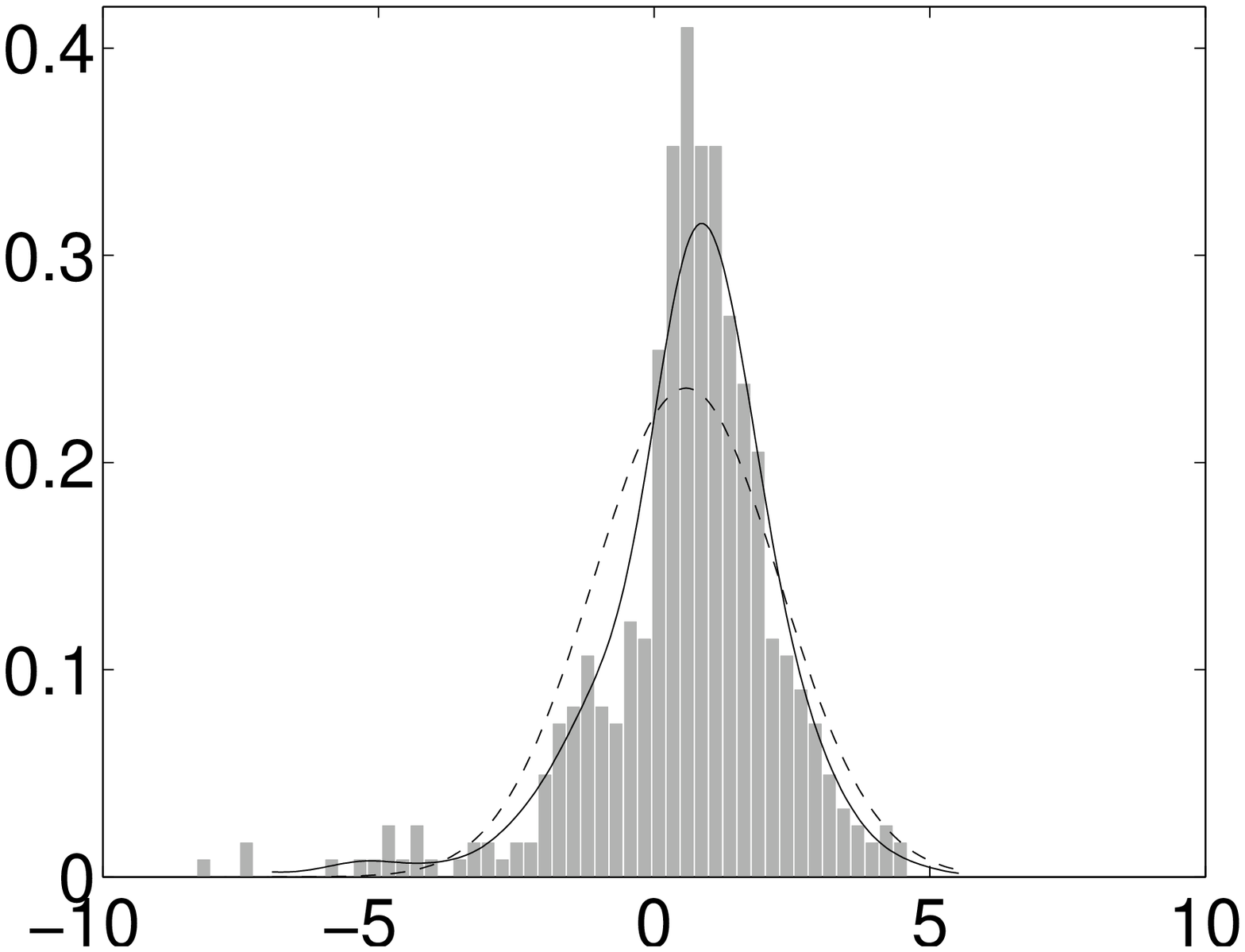}&\includegraphics[height=4cm,width=5.5cm, angle=0, clip=false]{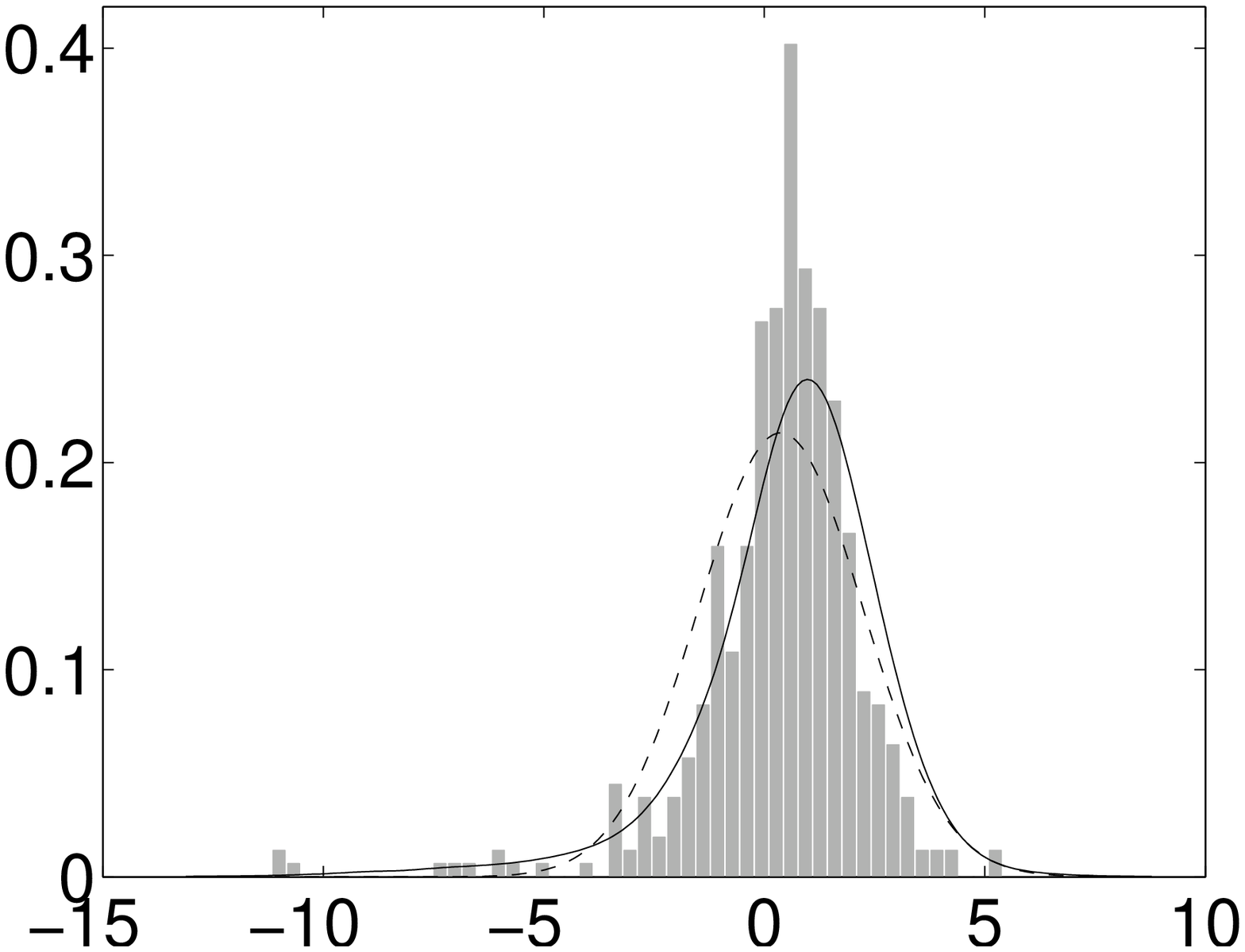}\\
\includegraphics[height=4cm,width=5.5cm, angle=0, clip=false]{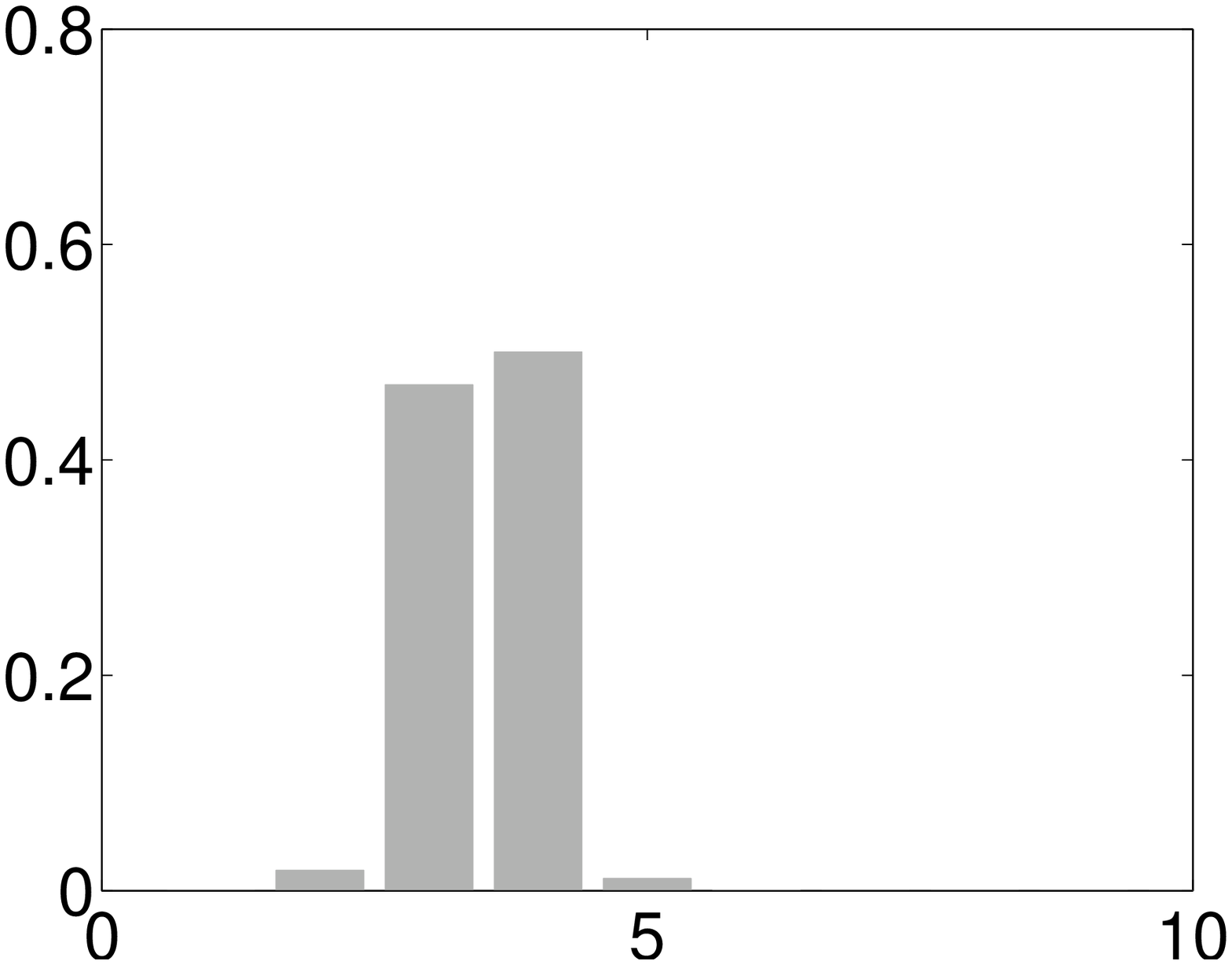}&\includegraphics[height=4cm,width=5.5cm, angle=0, clip=false]{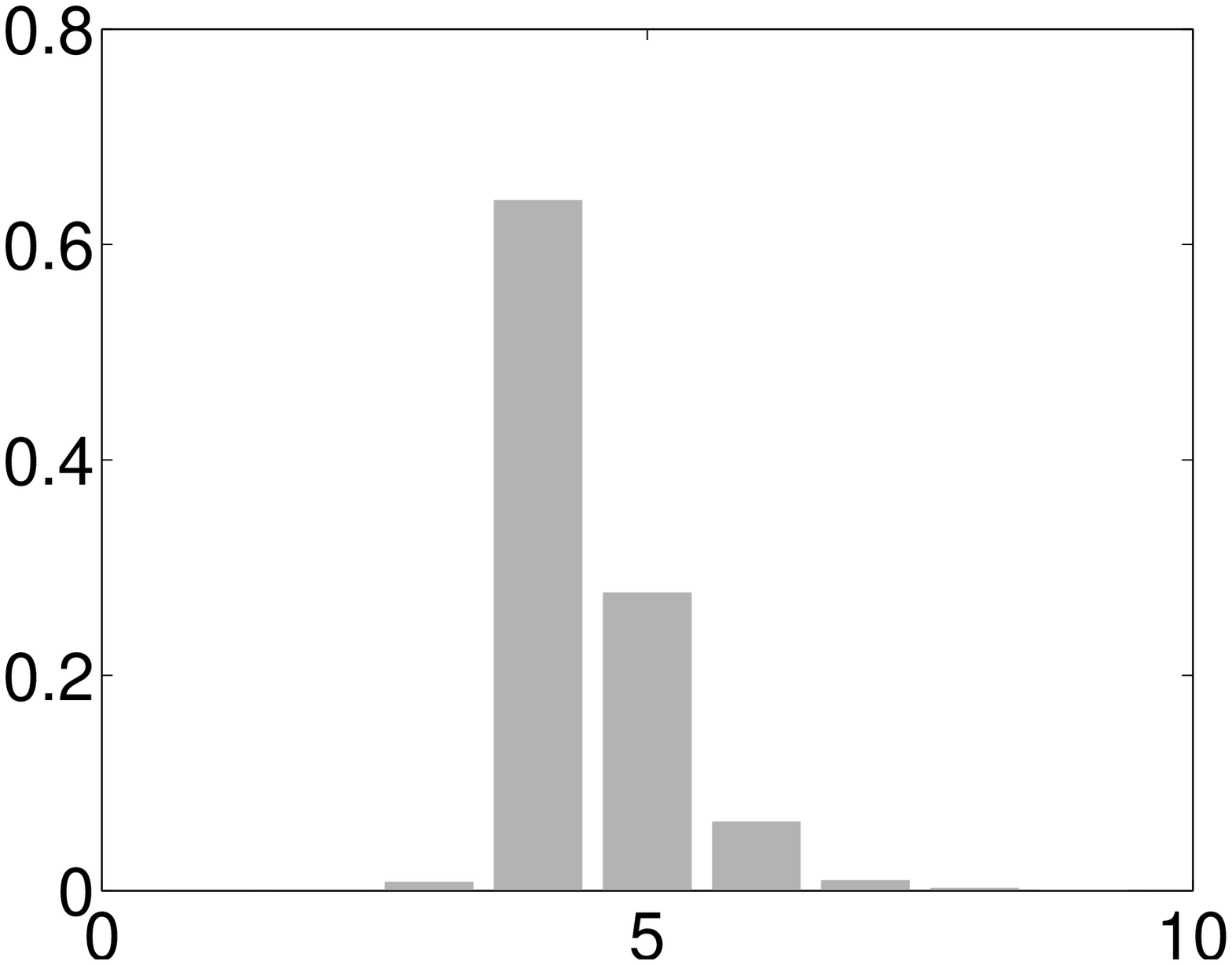}\\
\end{tabular}
\caption{First row: IPI log-changes (histogram), predictive distribution (solid line), best normal (dashed line). Second row: number of clusters.}\label{dataindppost}
\end{centering}
\end{figure}
The charts in the first row of Fig. \ref{dataindppost} show the predictive distributions (solid lines) 
generated by the non--parametric approach conditioning
on all values of $y_{it}$, for $i=1,2$ and $t=1,\ldots,T$ and the best normal fits (dashed lines) for the empirical distributions of the two series.

From a comparison with the empirical distribution, we note that the non--parametric approach, as opposed to the normal model, is able to capture asymmetry, excess of kurtosis, and multimodality in the data. The results from our non--parametric approach are in line with the practice of using of time-varying parameter models (e.g., Markov-switching models) to capture asymmetry and non-linearity in both the US and the EU business cycles.

The posterior distribution of the number of clusters is given in the second row of Fig. \ref{dataindppost}. The location of the posterior mode of the histograms allows us to conclude that the non--parametric approach detects three clusters for the US cycle and four clusters for the EU cycle. The result for the US data is coherent with the results available in the literature where three-regime Markov-switching models (see for example \cite{Kro00}) are usually considered. Moreover, we observe that the inclusion in the sample of the 2009 negative-growth (recession) period extends the validity of many past empirical findings that do not include the 2009 slowdown in the economic activity. An inspection of the posterior mean of the atoms and of the marginal clustering (see below in this section) allows us to conclude that the three clusters can have the economic interpretation of business cycle phases associated to substantially different levels of IPI growth-rates.

\begin{figure}[t]
\begin{centering}
\begin{tabular}{c}
\includegraphics[height=9.5cm,width=13.5cm, angle=0, clip=false]{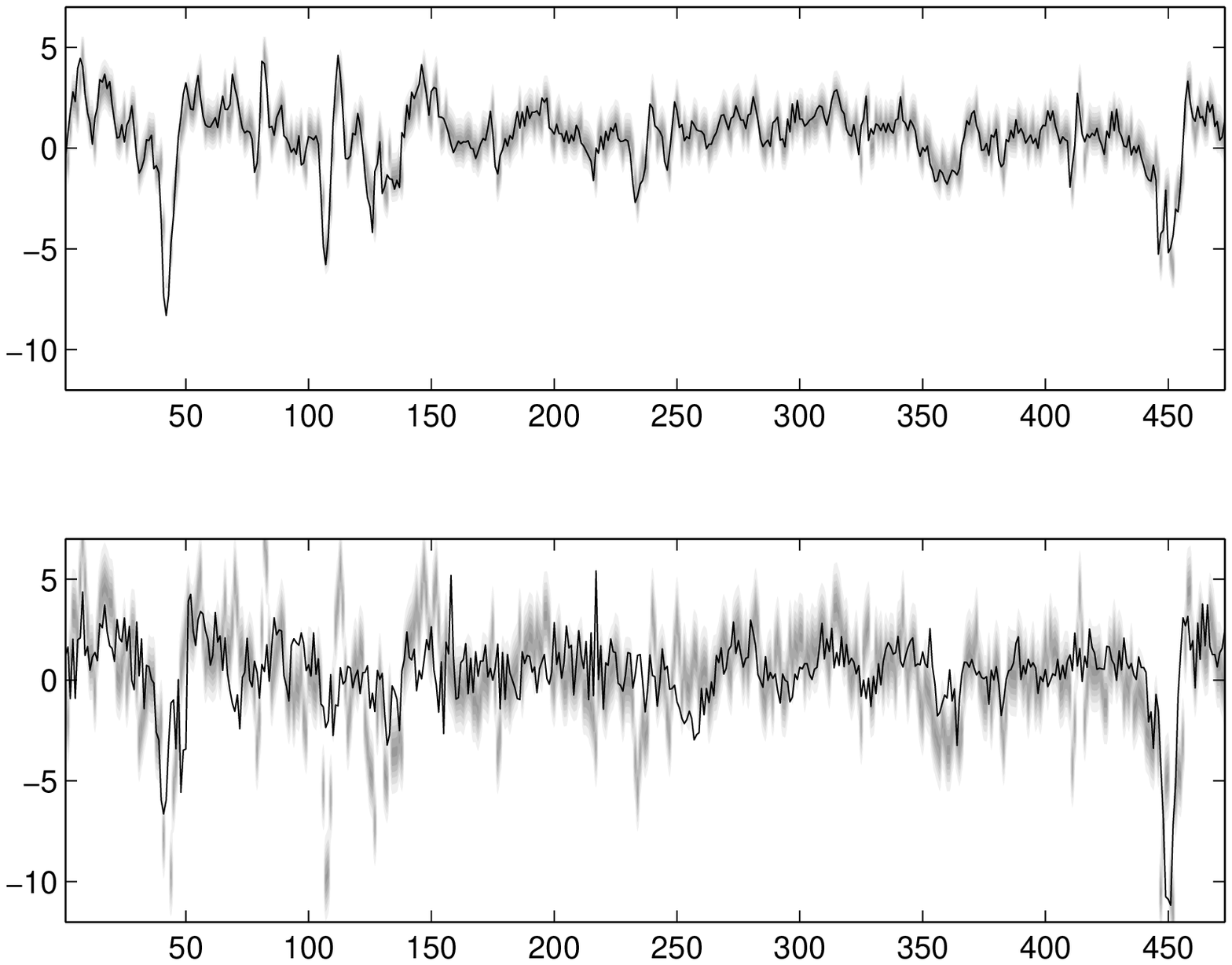}\\
\end{tabular}
\caption{US and EU IPI growth rates (black lines) and predictive densities (gray areas) evaluated sequentially for
$t=1,\ldots,T$, at the values of the predictors $y_{it-1},\ldots,y_{it-p}$, for $i=1,2$.}\label{dataindpSeq}
\end{centering}
\end{figure}

The results for the US and the EU cycles are, in a certain way, coherent with the output of parametric studies which suggest to consider at least three regimes. Nevertheless, the effects of the 2009 recession on the past empirical findings is an open issue and a matter of research. The result from our non--parametric approach is an interesting one because it suggests that four components are needed in order to capture the effects of the 2009 recession phase (see Fig. \ref{dataindppost}). As a consequence of the 2009 recession, a long left tail  present in the predictive (solid line in Fig. \ref{dataindppost}) is fatter than the tail of the best normal (dashed line in the same figure).

\begin{figure}[t]
\begin{centering}
\begin{tabular}{cc}
\multicolumn{2}{c}{\textbf{t=430 (1st of July 2007)}}\\
\includegraphics[height=5.5cm,width=6.5cm, angle=0, clip=false]{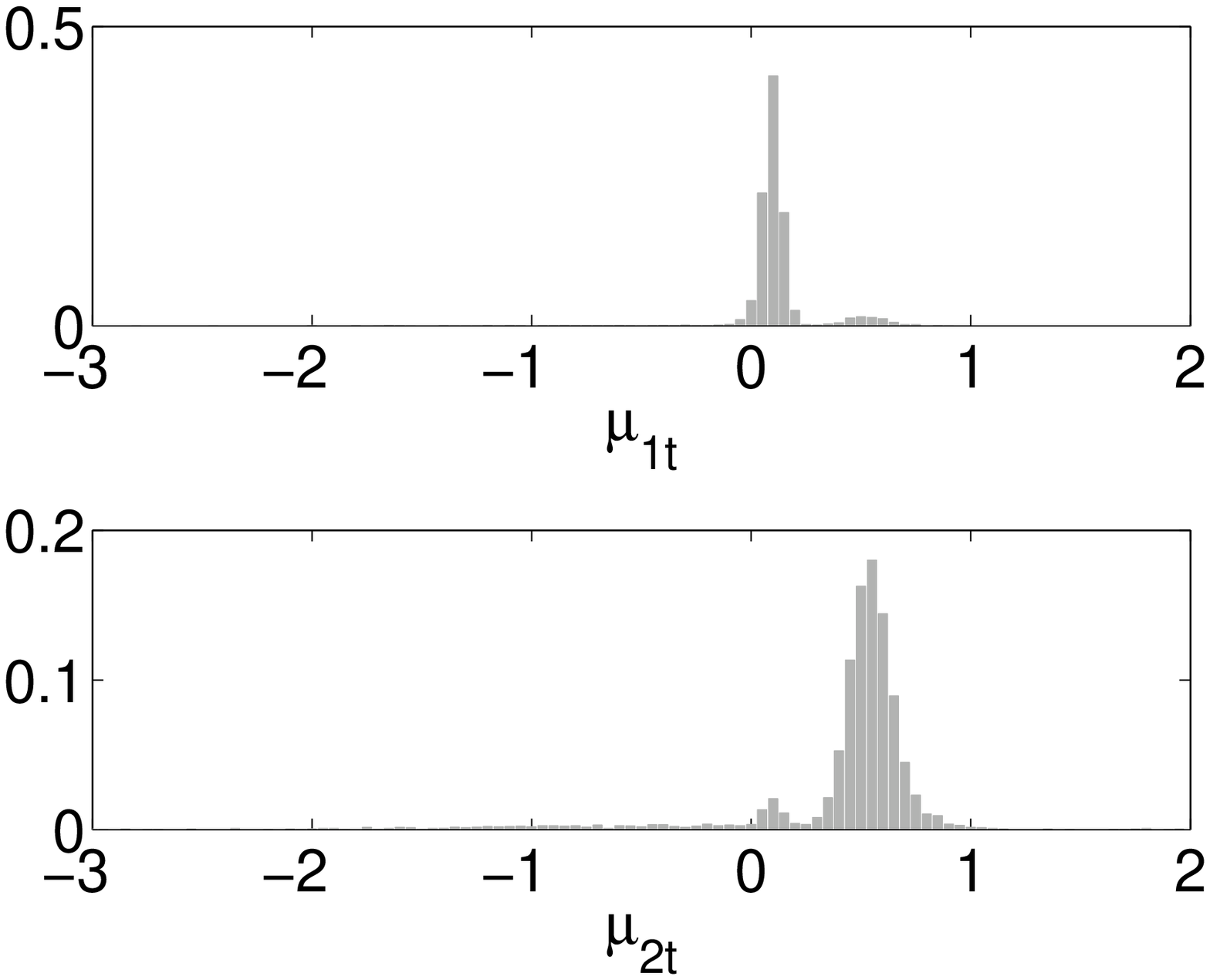}&
\includegraphics[height=5.5cm,width=6.5cm, angle=0, clip=false]{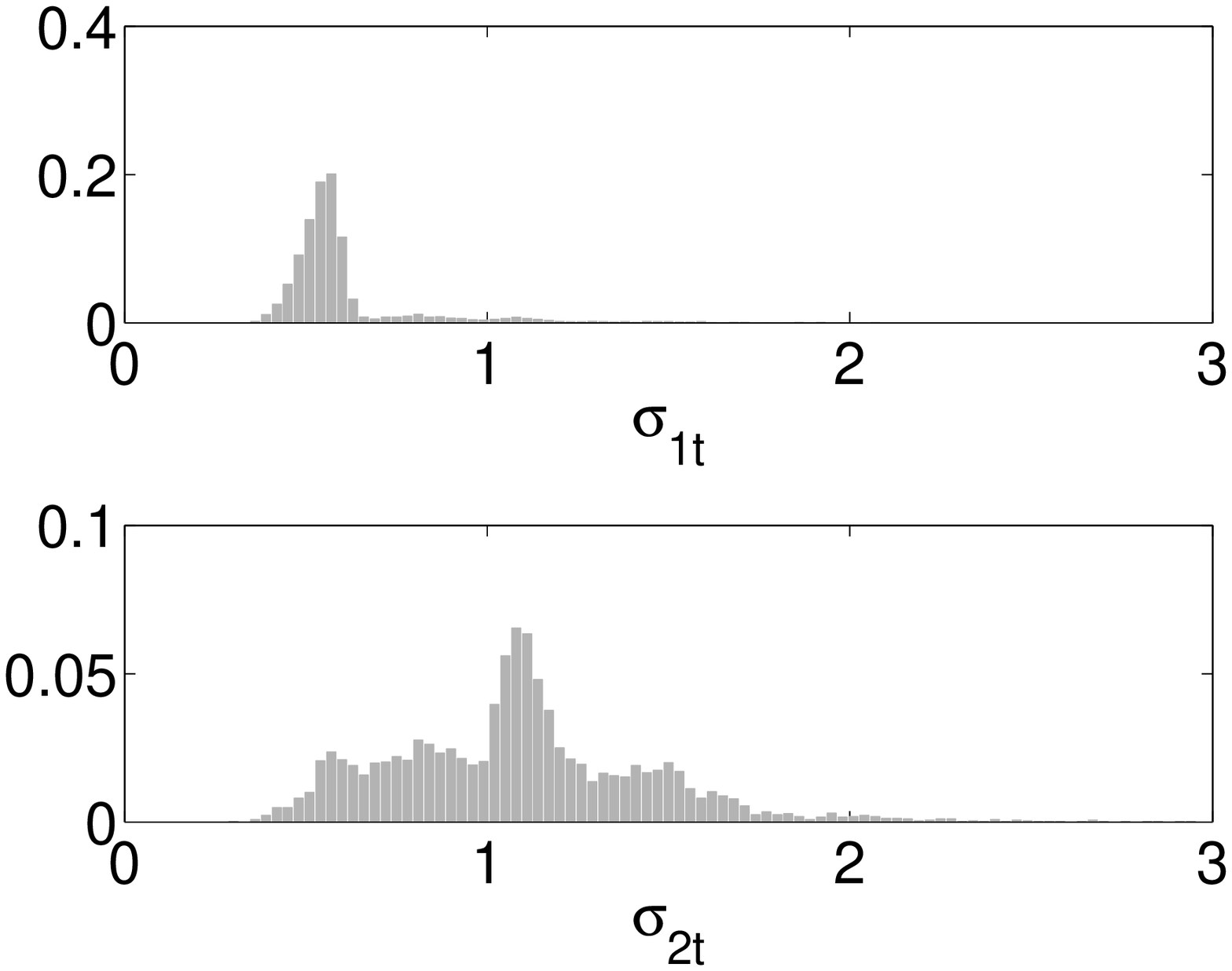}\\
\multicolumn{2}{c}{\textbf{t=450 (1st of March 2009)}}\\
\includegraphics[height=5.5cm,width=6.5cm, angle=0, clip=false]{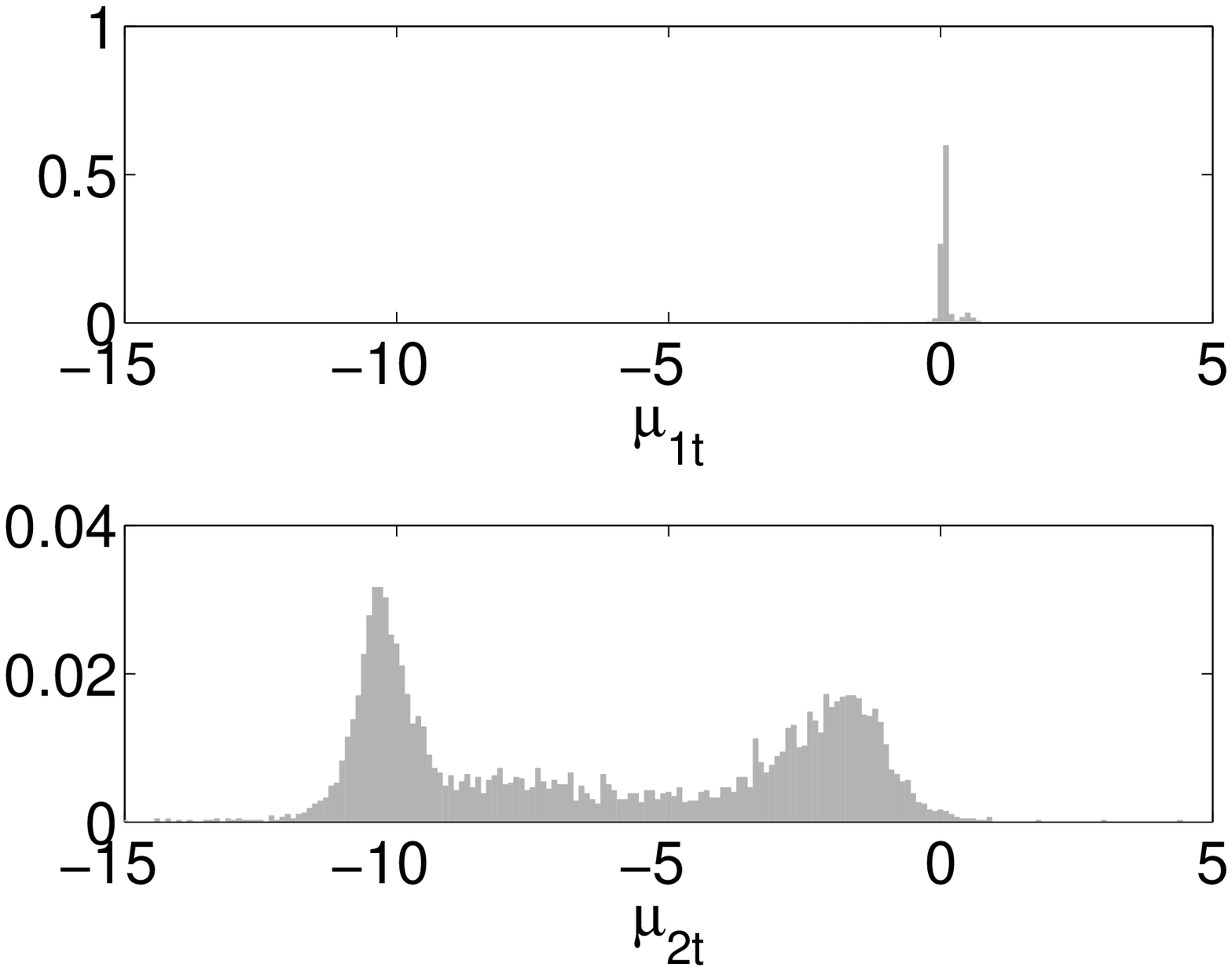}&
\includegraphics[height=5.5cm,width=6.5cm, angle=0, clip=false]{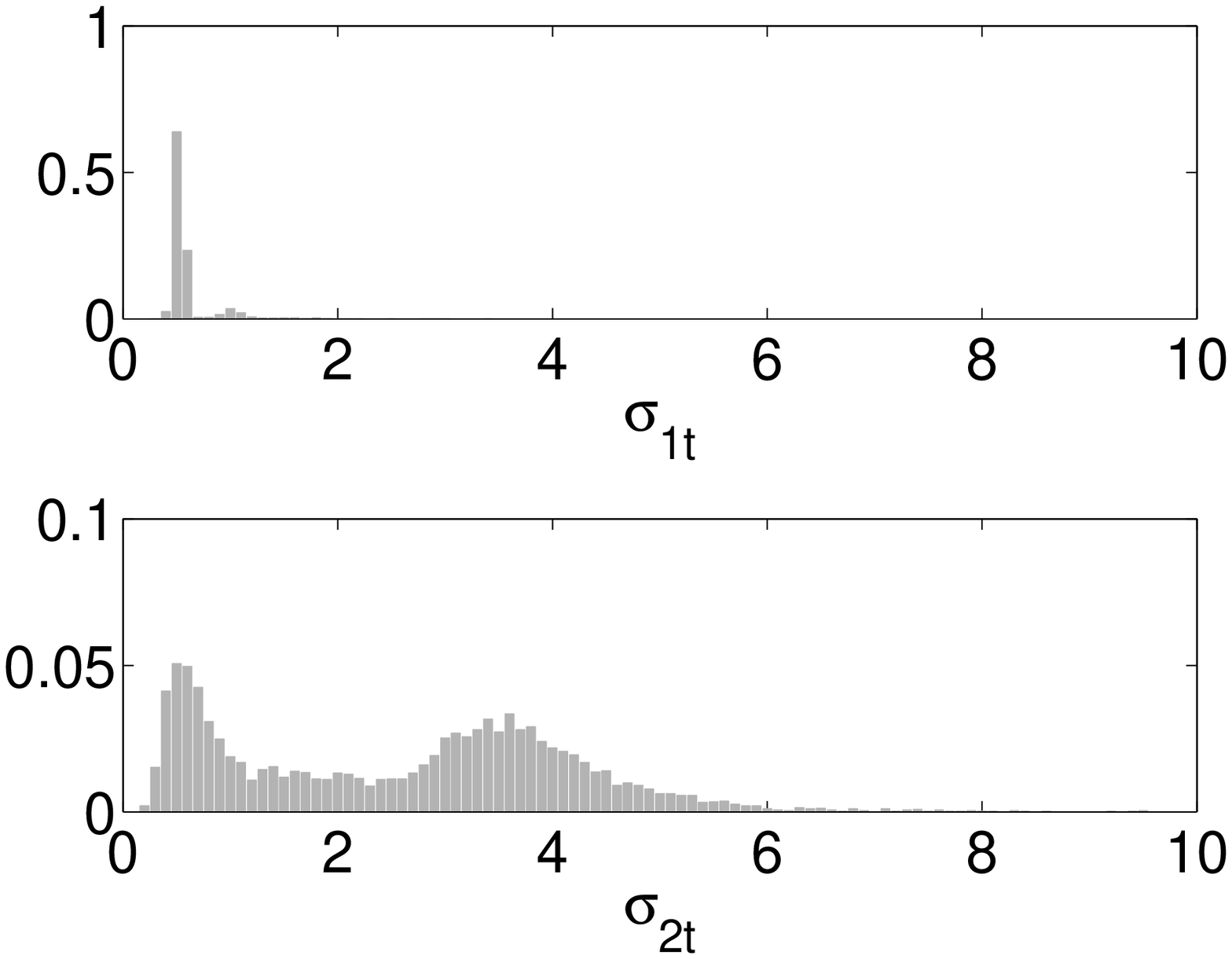}\\
\end{tabular}
\caption{Posterior density approximations for the atoms of the US ($(\mu_{1t},\sigma_{1t})$) and the EU ($(\mu_{2t},\sigma_{2t})$) growth rates $y_{1t}$ and $y_{2t}$
during expansion ($t=430$) and recession period ($t=450$).}\label{atom}
\end{centering}
\end{figure}

Fig. \ref{dataindpSeq} shows the sequence of predictive densities (gray area) indexed by time $t$, for $t=1,\ldots,T$. The predictive density for $y_{it}$ has been estimated conditionally on the whole set of data and has been evaluated sequentially over time at the current values of the predictors $y_{it-1},\ldots,y_{it-p}$, for $i=1,2$. In this figure, the effects of the recession are evident from the presence of  non-negligible probability values in correspondence of extremely negative growth rates that were not realized before 2009. Similarly, we found that, in both the expansion and recession phases, posterior distribution of the atoms exhibit multimodality and asymmetry. As an example Fig. \ref{atom}, shows the approximated posterior of the atoms $\mu_{it}$ and $\sigma_{it}$ in periods of expansion ($t=430$) and recession ($t=450$).  The posterior distribution of $\mu_{it}$ exhibits two modes in the positive half of the real line during an expansion phase and two modes in the negative half during a recession phase (first column of Fig. \ref{atom}). From the second column of the same figure, one can conclude that the volatility posterior distribution for both the US and the EU is more concentrated around lower values in expansion periods.

In order to identify the different components of our DP mixture model, we compute the posterior clustering of the data and the associated values of the atoms for each observations and country. We apply the least square clustering method proposed originally in \cite{Dahl}. The method has been successfully used in many applications (see for example \cite{KimTadesseVannucci} and \cite{RodriguezDunsonGelfand2008}) and is based on the posterior pairwise probabilities of joint classification
$P\{D_{is}=D_{jt}|Y\}$. To estimate this matrix, one can use
the following pairwise probability matrix:
$$
P_{ij,st}=\frac{1}{M}\sum_{l=1}^{M}\delta_{D_{is}^{l}}(D_{jt}^{l})
$$
that is estimated by using every pair of allocation variable $D_{is}^{l}$ $D_{jt}^{l}$, with $s,t=1,\ldots,T$ and over all the $l=1,\dots,M$ MCMC iterations. In \cite{Dahl}'s algorithm, one needs to evaluate $P_{ij,st}$ for $i=j$ and $i=1,2$.
\begin{figure}[t]
\begin{centering}
\begin{tabular}{cc}
\multicolumn{2}{c}{\textbf{Posterior Clustering for the US data}}\\
\includegraphics[height=5cm,width=6cm, angle=0, bbllx=20,bblly=220,bburx=567,bbury=620,clip=]{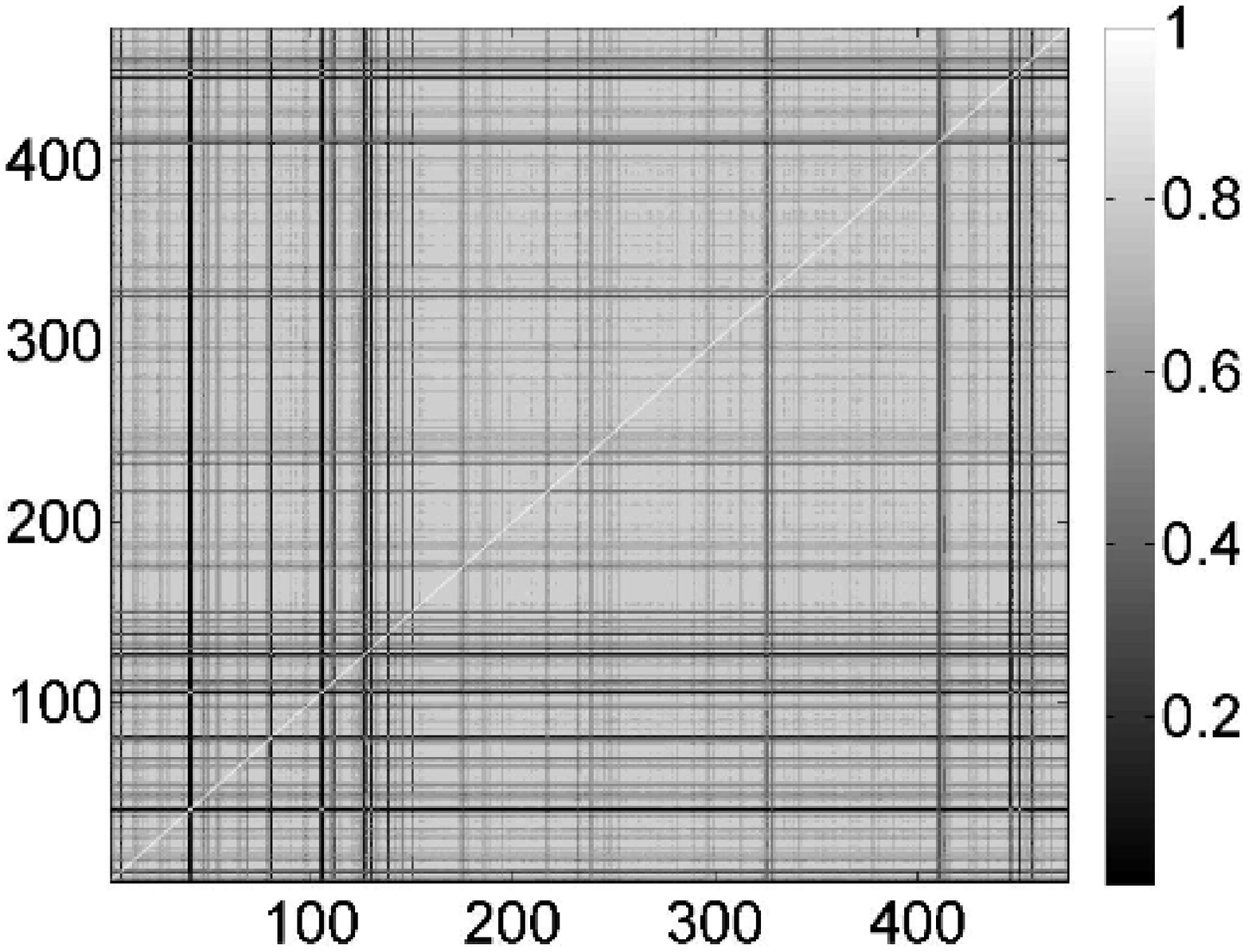}&
\includegraphics[height=5cm,width=6cm, angle=0, bbllx=20,bblly=220,bburx=567,bbury=620,clip=]{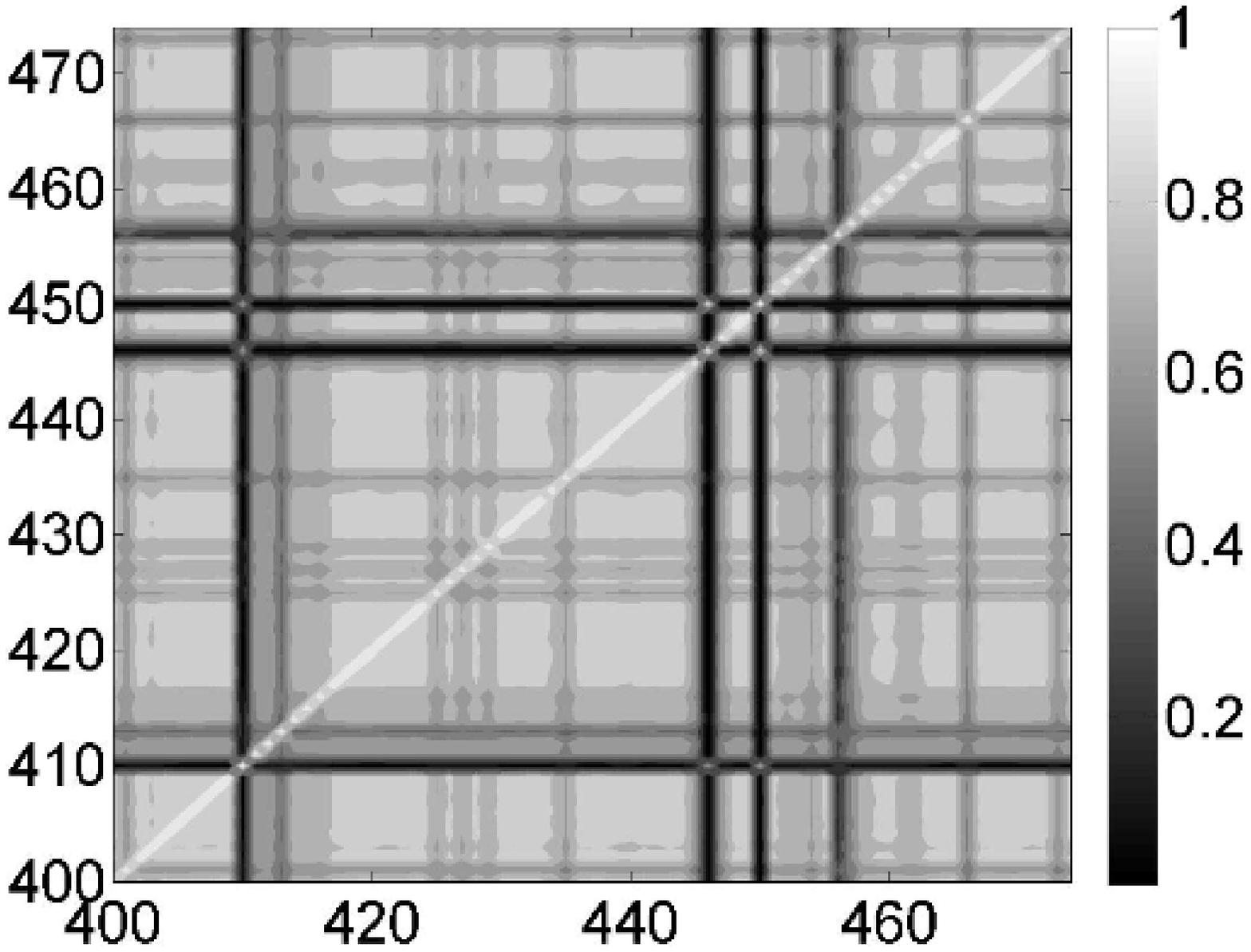}\\
\multicolumn{2}{c}{\textbf{Posterior Clustering for the EU data}}\\
\includegraphics[height=5cm,width=6cm, angle=0, bbllx=20,bblly=220,bburx=567,bbury=620,clip=]{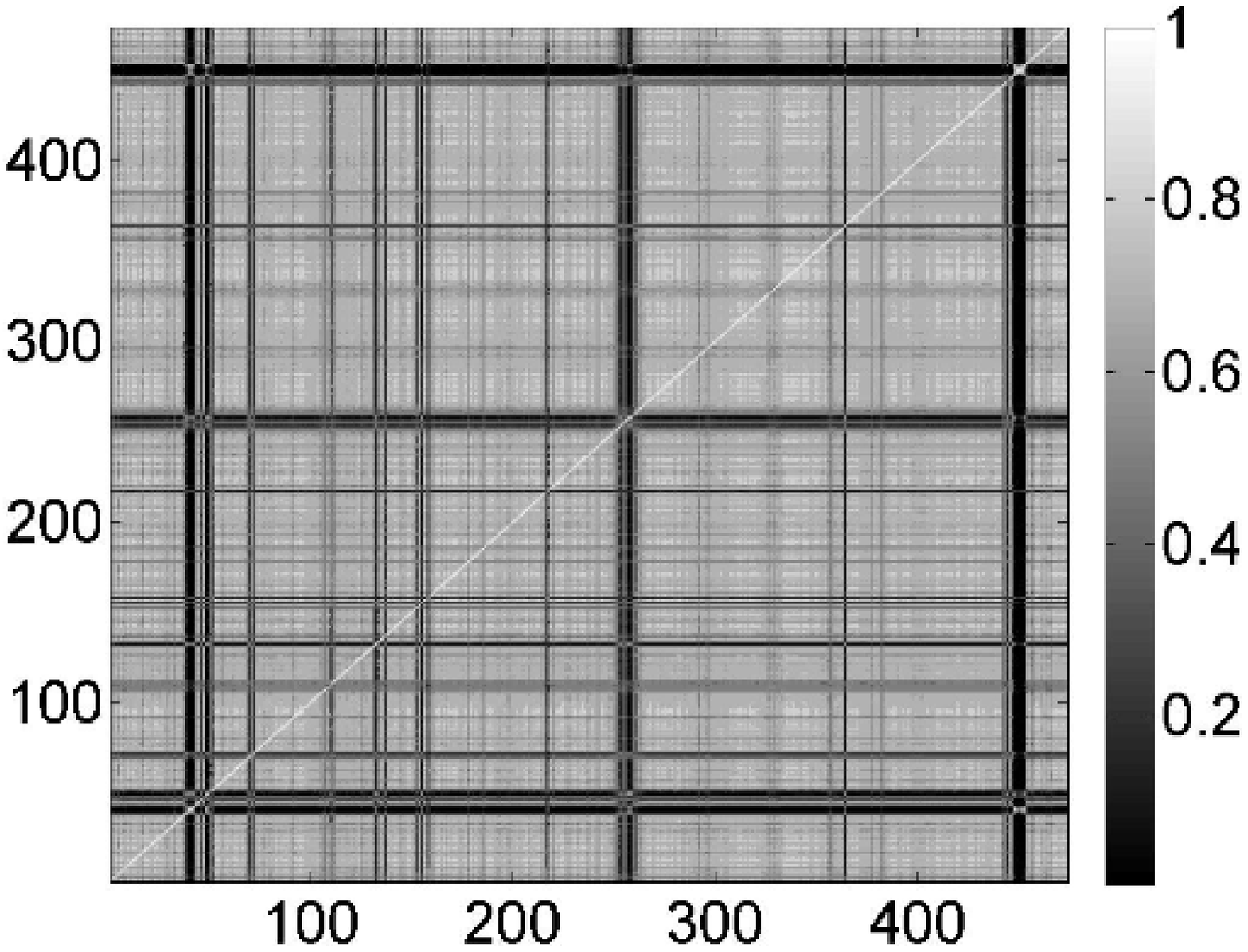}&
\includegraphics[height=5cm,width=6cm, angle=0, bbllx=20,bblly=220,bburx=567,bbury=620,clip=]{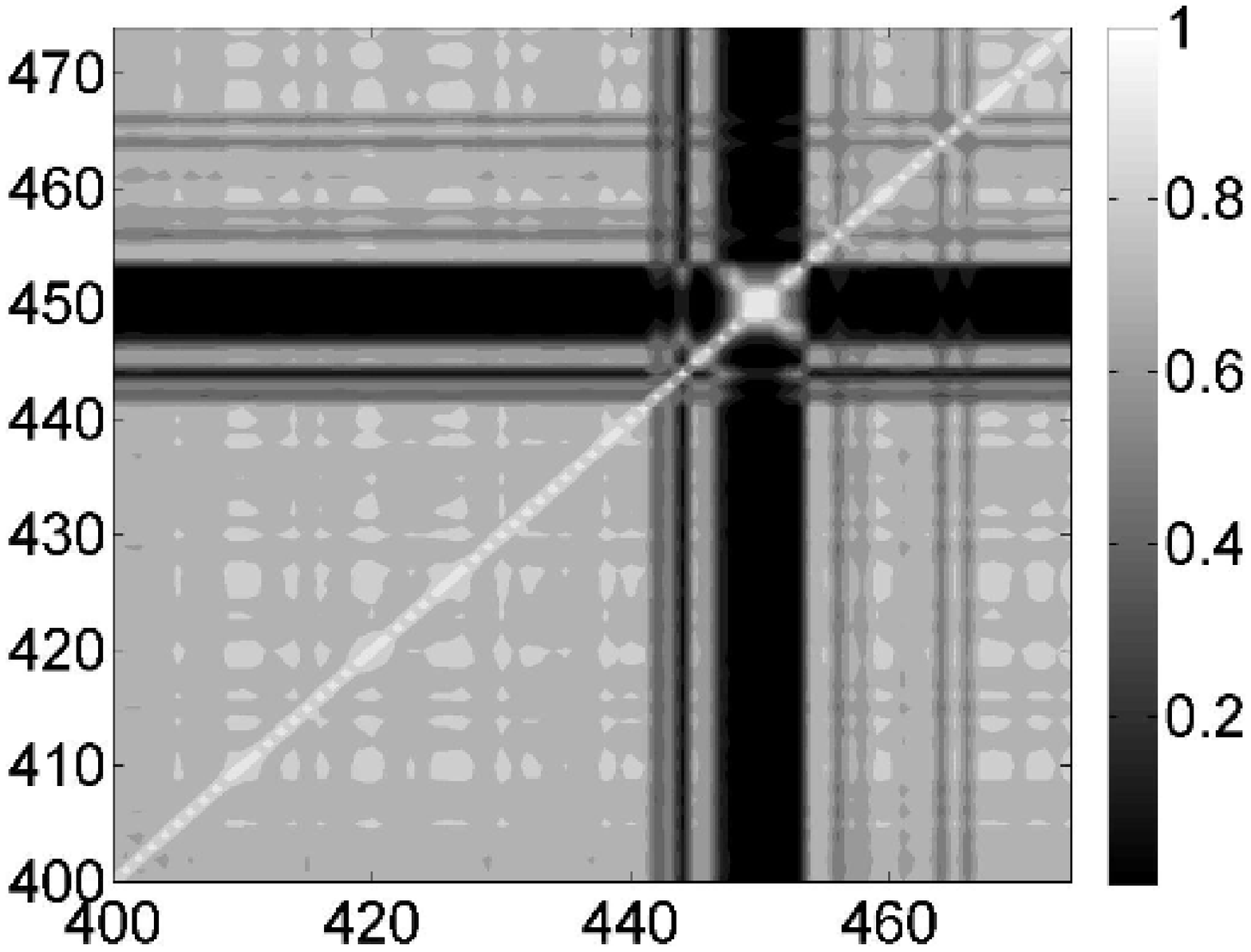}\\
\multicolumn{2}{c}{\textbf{Posterior Common Clustering for the EU and US data}}\\
\includegraphics[height=5cm,width=6cm, angle=0, bbllx=20,bblly=220,bburx=567,bbury=620,clip=]{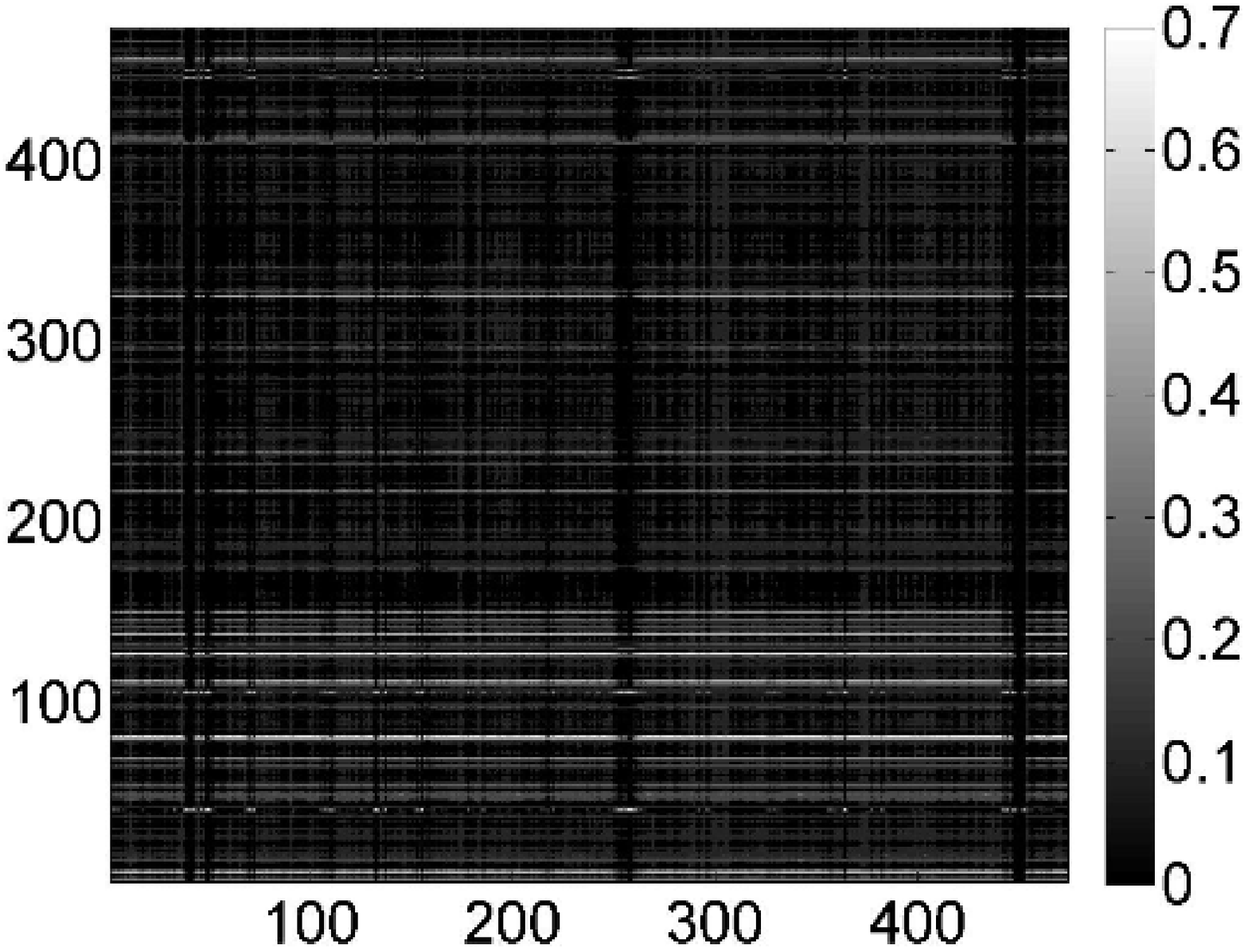}&
\includegraphics[height=5cm,width=6cm, angle=0, bbllx=20,bblly=220,bburx=567,bbury=620,clip=]{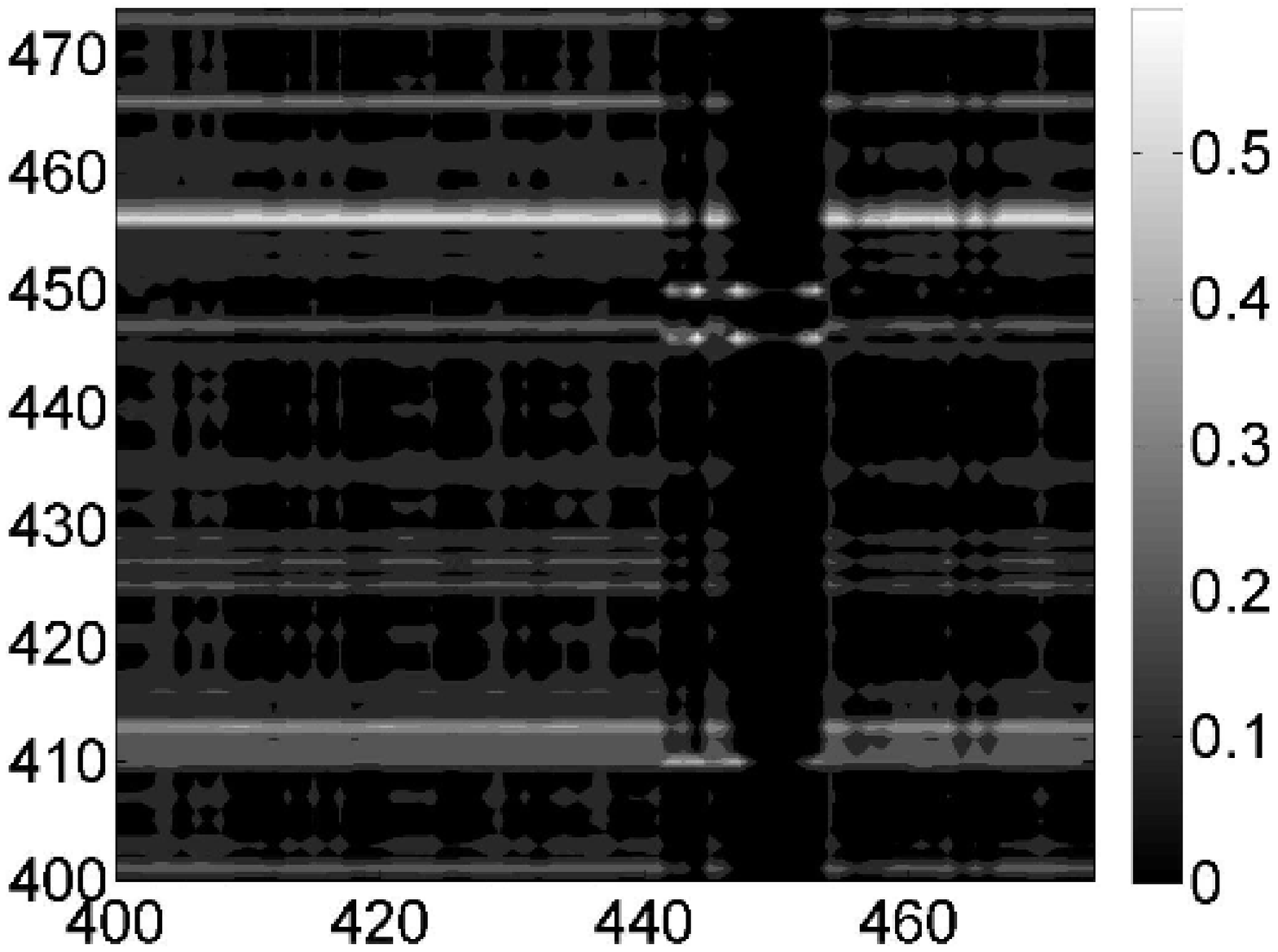}\\
\end{tabular}
\caption{Pairwise posterior probabilities for the clustering of the US data $P_{11,st}$ and the EU data $P_{22,st}$, and the
common clustering between the US and EU data, $P_{12,st}$  for $s,t\in\{1,\ldots,T\}$}\label{clustering}
\end{centering}
\end{figure}
The least square marginal clustering $D_{i,LS}$ is the clustering $D_{i}^{l_{i}}=(D_{i1}^{l_{i}},\ldots,D_{iT}^{l_{i}})$ (see Fig. \ref{clustering}) sampled at the $l_{i}$-th iteration which minimizes the sum of squared deviations from the pairwise posterior probability:
\[
l_{i}=\underset{l\in\{1,\ldots,M\}}{\arg\min}\sum_{t=1}^{T}\sum_{s=1}^{T}\left(\delta_{D_{is}^{l}}(D_{it}^{l})-P_{ii,st}\right)^{2}.
\]

\begin{figure}[t]
\begin{centering}
\includegraphics[height=8cm,width=9cm, angle=0, clip=false]{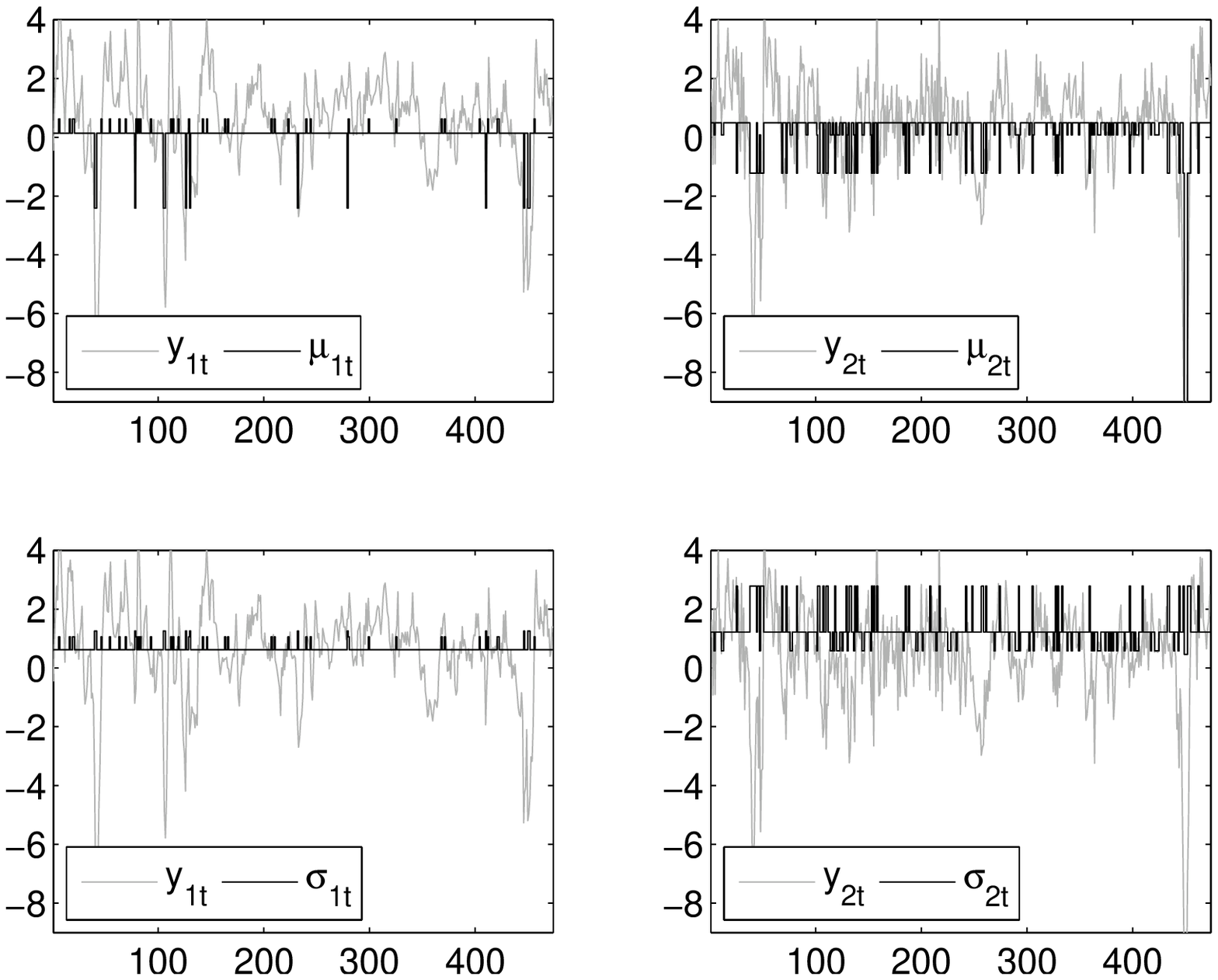}
\caption{Atom $(\hat{\mu}_{it}^{(l_{i})},\sigma_{it}^{(l_{i})})$ associated to the LS marginal clustering $l_{i}$ for $i=1,2$.}\label{atomclust}
\end{centering}
\end{figure}

More specifically, the first row (second row) shows the  posterior probabilities that two observations of the US cycle (EU cycle) belong to the same cluster. In the first column, one can clearly detect the presence of vertical and horizontal dark gray bands. They correspond to observations that do not cluster frequently together with other observations and that are associated with negative growth rates. A similar remark is true for the light gray areas. In the second column of Fig. \ref{clustering}, one can see the different behavior of the clustering for the US and the EU during the 2009 crisis.

Finally, the assumption in Eq. \eqref{atoms1} implies that the set of atoms sampled at every MCMC iteration is the same for the two series. This makes the allocation variables $D_{1t}^{l}$ and $D_{2s}^{l}$ comparable. For this reason, we apply  \cite{Dahl}'s algorithm to study the posterior probability $P_{ij,st}$ that two observations, each one from a different series (i.e. $i\neq j$), belong to the same cluster. That gives a measure of association between the two clustering, induced by the dependent DP, for the two series. The estimated pairwise probability is given in the third row of Fig. \ref{clustering} that shows the probability that two observations, one of the US cycle and another one of the EU cycle, belong to the same cluster. The white and light gray lines show that the two marginal clustering share some atoms.

The least square clustering allows us to find the posterior clustering of the data and to identify the different clusters. For the US cycle, the observations cluster together in three groups (see Fig. \ref{atomclust}) and the atoms associated with the three clusters are $(\mu_{1t},\sigma_{1t})\in\{(-2.4142,1.3625)$, $(0.136,0.6166)$, $(0.6216,1.0618)\}$ and lead to the identification of the cluster as recession, normal expansion, and strong expansion phases. For the EU cycle, the observations are classified in four groups (see Fig. \ref{atomclust}) and the atoms are $((-10.6170,0.6600),$ $(-1.2183,2.7857),$ $(0.0760,0.5794),$ $(0.4909,1.2165))$ and are interpreted as strong recession, normal recession, normal expansion, and strong expansion phases. These results on the features of the cycle phases are coherent with the recent findings in the business cycle literature with an exception for the EU cycle, which presents a fourth cluster of observations with very high negative growth rate $(-10.6170)$ corresponding to the 2009 recession period.

\section{Conclusions}\label{Concl}
We introduce a new class of multivariate dependent Dirichlet processes for modeling vectors of random measures. We discuss some properties of the process. We apply the dependent Dirichlet process to the context of non--parametric Bayesian inference and provide an efficient MCMC algorithm for posterior computation. Since our
process is particularly suitable for groups of data that exhibit a different clustering behavior, we apply it to multiple time series analysis. We provide an original application to the joint analysis of the US and  the EU business cycles and show that our non--parametric Bayesian model is able
able to highlight  some important issues for this kind of data.

\newpage

\appendix\section{The algorithm}
In the Block Gibbs Sampler described in Section \ref{S:Slice}
in principle one needs to sample an infinite number of $V_k$ and $\PPhi_k$.
But in order to proceed with the chain
it suffices to sample a finite number of $V_k$s  to
check condition \eqref{condN} and the finite number of $\PPhi_k$ to be used in \eqref{F4}.

For the sake of clarity we summarize here
the blocked Gibbs sampling algorithm.

\begin{center}
\begin{minipage}[t]{330pt} \par\hrule\vspace{5pt}
{THE ALGORITHM}
\par\vspace{5pt}\hrule
\begin{itemize}
\item INITIAL STEP. Initialize $U$, $D,\tilde \XXi$. With $D$ compute $D^*$ by using \eqref{mstar}.

\item UPDATING STEP. Suppose to have a sample of all the variables involved in the algorithm. This variables that comes from the previous step is labeled with "old". The variables that will be generated in the next step are labeled with "new".

\begin{enumerate}
\item The $(V^*,\tilde \XXi)|{new}$ are sampled by using the $D|{old}$ and \eqref{F2-A} with  Metropolis within Gibbs
 step as described in Subsection \ref{step2}.

\item The $V^{**}|new$ are sampled by using the $D|{old}$ and $\tilde \XXi|{new}$ by a Metropolis step as described
in Subsection \ref{step2}.

\item The $U|{new}$ are sampled by using the $(V^{*},V^{**},\tilde \XXi)|{new}$ and \eqref{F3};

\item $N^*_{i,j}$ are computed by using \eqref{condN}, with $U|{new}$ and $V_j|{new}$.
If some $V_{k}|new$ with $k>D^*|old$ are needed they are sampled
from the prior $P\{V_k \in dv_k|\tilde \XXi\}$.

\item The $\PPhi_j|{new}$
for $j=1,\dots,N^*$, with $N^*:=\max_{i=1,2}\max_{1 \leq j \leq n_{i}}(N^*_{i,j})$,
are sample by using \eqref{F1} and $D|{old}$  as described in \eqref{F1} and in Section \ref{S.S51}.

\item The $D|{new}$ are sampled by using \eqref{F4} with $U|{new}$, $V^{(n)}|new$ and $\PPhi_1|new,\dots,\PPhi_{N^*}|new$.

\end{enumerate}

\end{itemize}\label{algoritmo}

\hrule\vspace{5pt}
\end{minipage}
\end{center}

\section{Proofs}
\begin{proof}[\it Proof of Proposition \ref{corr}]
First of all observe that
\begin{equation}\label{f00}
\begin{split}
\E[\TG_{1}(A)\TG_{2}(B)]& =\sum_{h \geq 1, k \geq 1} \E[\J_A(\PPhi_{1k}) \J_B(\PPhi_{2h})]\E[W_{1k}W_{2h}] \\
& =\sum_{h \geq 1, k \geq 1, h \not=k} \E[\J_A(\PPhi_{1k})] \E[\J_B(\PPhi_{2h})]
\E[W_{1k}]\E[W_{2h}]\\ & \qquad + \sum_{h \geq 1}\E[\J_{A \times B} (\PPhi_{1h},\PPhi_{2h})] \E[W_{1h}W_{2h}].\\
\end{split}
\end{equation}
Now note that the following equalities hold
\begin{equation}\label{f1}
\sum_{h \geq 1} \E\left[S_{1h} S_{2h} \prod_{m \leq h-1} (1-S_{1m})(1-S_{2m}) \right ]
=\frac{\E[S_{11}S_{21}]}{\E[S_{11}]+\E[S_{21}]-\E[S_{11}S_{21}]}
\end{equation}
\begin{equation}\label{f2}
\sum_{h \not = k} \E\left[S_{1h} S_{2k} \prod_{m \leq h-1} (1-S_{1m})\prod_{l \leq k-1} (1-S_{2l})\right  ]
=\frac{\E[S_{11}]+\E[S_{21}]-2\E[S_{11}S_{21}]}{\E[S_{11}]+\E[S_{21}]-\E[S_{11}S_{21}]}
\end{equation}
Combining \eqref{f00} with \eqref{f1}-\eqref{f2} it follows that
$$\E[\TG_1(A) \TG_2(B)]= G_0(A \times B\times\X^{r-2}) \frac{\E[S_{11}S_{21}]}{\E[S_{11}]+\E[S_{21}]-\E[S_{11}S_{21}]}$$
$$\quad+G_{01}(A) G_{02}(B) \frac{\E[S_{11}]+\E[S_{21}]-2\E[S_{11}S_{21}]}{\E[S_{11}]+\E[S_{21}]-\E[S_{11}S_{21}]}.
$$
Since $\E[\TG_i(\cdot)]=G_{0i}(\cdot)$, $i=1,\dots,r$, it follows
\begin{equation}\label{fnn}
Cov[\TG_1(A),\TG_2(B)]=
 \frac{\E[S_{11}S_{21}]}{\E[S_{11}]+\E[S_{21}]-\E[S_{11}S_{21}]} [G_0(A \times B\times\X^{r-2})-G_{01}(A) G_{02}(B) ].
\end{equation}
In a similar fashion
\begin{equation}\label{f4}
\E[\TG_1(A)^2 ]= \frac{G_{01}(A) \E[S_{11}^2]  +2G_{01}^2(A)(\E[S_{11}]-\E[S_{11}^2])}{2\E[S_{11}]-\E[S_{11}^2]}
\end{equation}
\begin{equation}\label{f5}
\E[\TG_2(B)^2 ]= \frac{G_{02}(B) \E[S_{21}^2]  +2\G_{02}^2(B)(\E[S_{21}]-\E[S_{21}^2])}{2\E[S_{21}]-\E[S_{21}^2]}
\end{equation}
and then
\begin{equation}\label{f6}
Var[\TG_1(A)]= G_{01}(A)(1-G_{01}(A)) \frac{ \E[S_{11}^2]}{2\E[S_{11}]-\E[S_{11}^2]}
\end{equation}
\begin{equation}\label{f7}
Var[\TG_2(B)]= G_{02}(B)(1-G_{02}(B)) \frac{ \E[S_{21}^2]}{2\E[S_{21}]-\E[S_{21}^2]}
\end{equation}
\end{proof}

\begin{proof}[\it Proof of Corollary \ref{corrH}]
By direct calculation or using the results in \cite{NadarajahKotz:2005} one obtains
\begin{eqnarray*}
&&\mathbb{E}(S_{21})=\mathbb{E}(S_{11})=\frac{1}{1+\alpha_1+\alpha_2},  \quad
\mathbb{E}(S_{21}^{2})=\mathbb{E}(S_{11}^{2})=\frac{2}{(1+\alpha_1+\alpha_2)(2+\alpha_1+\alpha_2)} \\
&&\mathbb{E}(S_{21}S_{11})=\frac{B(2,\alpha_1)B(2,\alpha_1)B(\alpha_2,\alpha_1+3)}{B(1,\alpha_1)B(1,\alpha_1)B(\alpha_1+1,\alpha_2)}=\frac{\alpha_1+2}{(\alpha_1+1)(\alpha_1+\alpha_2+1)(\alpha_1+\alpha_2+2)}
\end{eqnarray*}
for (H1) and
\begin{eqnarray*}
&&\mathbb{E}(S_{11})=\frac{1}{1+\alpha_1+\alpha_2}, \,\,
\mathbb{E}(S_{21})=\frac{1}{1+\alpha_1},\,\,               \\
&&\mathbb{E}(S_{11}^{2})=\frac{2}{(1+\alpha_1+\alpha_2)(\alpha_1+\alpha_2+2)},\,\,
\mathbb{E}(S_{21}^{2})=\frac{2}{(1+\alpha_1)(2+\alpha_1)}\\
&&\mathbb{E}(S_{21}S_{11})=\frac{B(3,\alpha_1)B(2+\alpha_1,\alpha_2)}{B(1,\alpha_1)B(\alpha_1+1,\alpha_2)}=\frac{2}{(2+\alpha_1)(1+\alpha_1+\alpha_2)}
\end{eqnarray*}
for (H2).
Hence the correlation between the two random measures
is
\[
\begin{split}
Cor( \TG_1(A), \TG_2(A))&=\frac{\E[S_{21}S_{11}]}{1-\E[(1-S_{21})(1-S_{11})]} \sqrt{\frac{(2\E[S_{21}]-\E[S_{21}^2])(2\E[S_{11}]-\E[S_{11}^2])}{ \E[S_{11}^2]\E[S_{21}^2]}}\\
&=\left(\frac{(\alpha_1+\alpha_2+1)(\alpha_1+2)}{2(\alpha_1+1)(\alpha_1+\alpha_2+1)-(\alpha_1+2)}\right)
\end{split}
\]
and
\[
\begin{split}
Cor( \TG_1(A), \TG_2(A))&=\frac{\E[S_{21}S_{11}]}{1-\E[(1-S_{21})(1-S_{11})]} \sqrt{\frac{(2\E[S_{21}]-\E[S_{21}^2])(2\E[S_{11}]-\E[S_{11}^2])}{ \E[S_{11}^2]\E[S_{21}^2]}}\\
&=\frac{2(\alpha_1+1)}{(\alpha_1+2)(2\alpha_1+\alpha_2+1)-\alpha_1}\sqrt{(\alpha_1+1)(\alpha_1+\alpha_2+1)}
\end{split}
\]
for (H1) and (H2), respectively.

\end{proof}

\begin{proof}[\it Proof of \eqref{corr-r>2H2}]
For the sake of simplicity write $V_k$ in place of $V_{k1}$.
Recall that $V_0 \sim Beta(1,\alpha_1)$ and, for $1 \leq k \leq r-1$, $V_k\sim Beta(1+\alpha_1+\dots+\alpha_{k},\alpha_{k+1})$.
Let $1 \leq i < j \leq r$. Since $S_{i1}=V_0 V_1\dots V_{r-i}$, one gets
\[
S_{i,1}S_{j,1}=V_0^2V_1^2\dots V_{r-j}^2 V_{r-j+1}\dots V_{r-i}.
\]
After some computations, using the fact that $V_j$ are independent,
\[
\E[S_{i,1}S_{j,1}]=\frac{2}{(2+\alpha_1+\dots+\alpha_{r-j+1})(1+\alpha_1+\dots+\alpha_{r-i+1})}.
\]
In addition, one has
\[
\E[S_{i,1}]=\frac{1}{1+\alpha_1+\dots+\alpha_{r-i+1}},
\quad \E[S^2_{i,1}]=\frac{2}{(1+\alpha_1+\dots+\alpha_{r-i+1})(2+\alpha_1+\dots+\alpha_{r-i+1})}.
\]
Set
\[
C_{i,j}:=\frac{\E[S_{i,1}S_{j,1}]}{\E[S_{i,1}]+\E[S_{j,1}]-\E[S_{i,1}S_{j,1}]}
\sqrt{\Big (\frac{2\E[S_{i,1}]}{\E[S_{i,1}^2]}-1\Big)\Big(\frac{2\E[S_{j,1}]}{\E[S_{j,1}^2]}-1\Big)}.
\]
Simple algebra gives
\[
\sqrt{\Big (\frac{2\E[S_{i,1}]}{\E[S_{i,1}^2]}-1\Big)\Big(\frac{2\E[S_{j,1}]}{\E[S_{j,1}^2]}-1\Big)}
=\sqrt{(1+\alpha_1+\dots+\alpha_{r-i+1})(1+\alpha_1+\dots+\alpha_{r-j+1})}
\]
and
\[
\begin{split}
& \frac{\E[S_{i,1}S_{j,1}]}{\E[S_{i,1}]+\E[S_{j,1}]-\E[S_{i,1}S_{j,1}]} \\
& \qquad =\frac{2(1+\alpha_1+\dots+\alpha_{r-j+1})}{
2(1+\alpha_1+\dots+\alpha_{r-j+1})^2+(2+\alpha_1+\dots+\alpha_{r-j+1})(\alpha_{r-j+2}+\dots+\alpha_{r-i+1})}. \\
\end{split}
\]
That is
\[
C_{i,j}=\frac{2\sqrt{(1+\alpha_1+\dots+\alpha_{r-i+1})} (1+\alpha_1+\dots+\alpha_{r-j+1})^{\frac{3}{2}}}{
2(1+\alpha_1+\dots+\alpha_{r-j+1})^2+(2+\alpha_1+\dots+\alpha_{r-j+1})(\alpha_{r-j+2}+\dots+\alpha_{r-i+1})}
\]
which gives \eqref{corr-r>2H2} since $C_{i,j}=corr(G_i(A),G_j(A))$.
\end{proof}

{\it Full-conditionals.}
The joint distribution of $[V,\PPhi,U,D,Y,\tilde \XXi]$ is
\begin{equation}\label{densitacong}
\begin{split}
&P\{ V \in dv , \PPhi \in d\varphi, Y \in dy ,U \in du^{(n)},
D =d^{(n)},\tilde \XXi \in (d\alpha_1,d\alpha_2) \} \\
&= \Big [
\prod_{i=1,2}\prod_{j=1}^{n_i}  \J\{u_{ij} < w_{i,d_{i,j}}  \}
\CK (y_{i,j} |\varphi_{d_{i,j}})  \Big ] \otimes_{i=1,2} \otimes_{j=1}^{n_i} dy_{ij} du_{ij}
\\ & \quad  \quad \quad
 \otimes_{k \geq 1}  \Big [ P\{ V_k \in dv_k|\tilde \alpha=(\alpha_1,\alpha_2)\}  \otimes G_0(d\varphi_k )\Big ] \otimes P\{\XXi \in (d\alpha_1,d\alpha_2) \}\\
\end{split}
\end{equation}
where $w_{i,k}=v_{0k}v_{ik}\prod_{j<k} (1-v_{0j}v_{ij})$, with the convenction that
$v_{2k}=1$, for every $k$, under ($H_2$).

\begin{proof}[Proof of \eqref{F2-A}-\eqref{densQ}] From \eqref{densitacong}
one gets
\[
\begin{split}
& P\{ V \in dv,\tilde \XXi \in (d\alpha_1,d\alpha_2)  | Y^{(n)},\PPhi,D\} \\
& \propto \Big [\prod_{i=1,2}\prod_{j=1}^{n_i} w_{i,D_{ij}} \Big ]
  \otimes_{j \geq 1} P\{V_j \in dv_j|\tilde \XXi=(\alpha_1,\alpha_2) \}
P\{\tilde \XXi \in (d\alpha_1,d\alpha_2) \}. \\
\end{split}
\]
Now note that
\[\begin{split}
\prod_{i=1,2}\prod_{j=1}^{n_i} w_{i,D_{ij}} & = \prod_{j=1}^{D^*} v_{0j}^{A_{1j}+A_{2j}} v_{1j}^{A_{1j}}v_{2j}^{A_{2j}} \\
& \quad (1-v_{0j}v_{1j})^{B_{1j}} (1-v_{0j}v_{2j})^{B_{2j}}.  \\
\end{split}
\]
\end{proof}



\end{document}